\documentclass[10pt]{amsart}

\usepackage{times,amsmath,amsbsy,amssymb,amscd,mathrsfs}
\usepackage{slashbox}
\usepackage{graphicx,subfigure,epstopdf,wrapfig,chemarrow}
\usepackage{algorithm2e} 
\usepackage{multicol,multirow}
\usepackage{mathtools}
\usepackage[usenames,dvipsnames,svgnames,table]{xcolor}
\usepackage[numbered]{mcode}
\definecolor{myBlue}{rgb}{0.0,0.0,0.55}
\definecolor{green}{rgb}{0.0,0.7,0.2}
\usepackage[pdftex,colorlinks=true,citecolor=myBlue,linkcolor=myBlue]{hyperref}

\usepackage{comment,enumerate,multicol,xspace}

  \newcounter{mnote}
  \let\oldmarginpar\marginpar
    \renewcommand\marginpar[1]{\-\oldmarginpar[\raggedleft\footnotesize #1]%
    {\raggedright\footnotesize #1}}

\newtheorem{theorem}{Theorem}[section]
\newtheorem{lemma}[theorem]{Lemma}

\newtheorem{example}[theorem]{Example}

\newtheorem{remark}[theorem]{Remark}
\newtheorem{alg}[theorem]{Algorithm}
\newtheorem{assum}[theorem]{Assumption}

\newcommand{\dt}{\,\Delta t}

\newcommand{\dd}{\,{\rm d}}

\newcommand{\bs}{\boldsymbol}

\newcommand{\vertiii}[1]{{\left\vert\kern-0.25ex\left\vert\kern-0.25ex\left\vert #1 
    \right\vert\kern-0.25ex\right\vert\kern-0.25ex\right\vert}}

\DeclareMathOperator*{\img}{img}
\newcommand{\curl}{{\rm curl\,}}
\renewcommand{\div}{\operatorname{div}}
\newcommand{\grad}{{\rm grad\,}}

\usepackage[margins]{trackchanges}
\renewcommand{\initialsOne}{chen}

\begin{document}
\title[MFEM for FHD model]{Energy-preserving Mixed finite element methods for a ferrofluid flow model}
\author{Yongke Wu \and Xiaoping Xie}
\date{\today}
\thanks{This work was supported by the National Natural Science Foundation of China (11971094, 12171340,11771312)}
\thanks{Y.~Wu,  School of Mathematical Sciences, University of Electronic Science and Technology of China, Chengdu 611731, China. Email: wuyongke1982@sina.com}
\thanks{X.~Xie, School of Mathematical, Sichuan University, Chengdu 610065, China. Email: xpxie@scu.edu.cn}


\begin{abstract}
In this paper, we develop a class of mixed finite element methods for the ferrofluid flow model proposed by Shliomis [Soviet Physics JETP, 1972]. We show that the energy stability of the weak solutions to the model is preserved exactly for both the semi- and fully discrete finite element solutions. 
Furthermore, we prove the existence and uniqueness of the discrete solutions   and derive optimal error estimates for both the the semi-  and fully discrete schemes. Numerical experiments confirm the theoretical results. 

\end{abstract}


\maketitle

\section{Introduction}

 Ferrofluids are colloidal liquids consisting of  nanoscale ferromagnetic or ferrimagnetic    particles suspended in  carrier fluids. They   
 have been wildly used in many technical  areas ~\cite{Zahn2001} such as instrumentation, vacuum technology, lubrication, vibration damping and acoustics, and  
 are expected to apply to some biomedical fields ~\cite{Pankhurst2003} like magnetic separation, drugs or radioisotopes targeted by magnetic guidance, hyperthermia treatments, and  magnetic resonance imaging contrast enhancement.

There are two widely accepted   ferrohydrodynamics (FHD) models: one has been obtained by Rosensweig~\cite{Rosensweig1985,Rosensweig1987}, and the other one is due to Shliomis~\cite{Shliomis1972Effective,Shliomis2002Ferrofluids}. These two models both treat ferrofluids as homogeneous monophase fluids. The main difference between them is that Rosensweig's model considers the internal rotation of the nanoparticles, while Shliomis' model deals with the rotation as a magnetic torque. The existence of solutions for these two models were discussed in~\cite{Amirat2008,Amirat2009Strong,Amirat2010UniqueR,Amirat2008GlobalR,Nochetto2019}. For Rosensweig's FHD model, Amirat, Hamdache and Murat \cite{Amirat2008GlobalR} gave the existence of global-in-time weak solutions,  Amirat and Hamdache \cite{Amirat2010UniqueR} showed the local-in-time existence of the unique strong solution, and Nochetto et al. \cite{Nochetto2019} established the global existence of weak solutions in the absence of additional diffusion in the magnetization equation. For Shliomis' model, Amirat and Hamdache \cite{Amirat2008} obtained the existence of global-in-time weak solutions, and in \cite{Amirat2009Strong} derived the existence of local-in-time strong solutions.

Since the FHD models are coupled nonlinear partial differential equations, it is usually difficult  to obtain their  analytical solutions, and the only way to solve them is to seek approximation solutions by using 
  numerical methods. 
  There are limited works  in this research field. 
  In \cite{Snyder2003Finite,Lavrova2006Numerical,Knobloch2010Numerical,Yoshikawa2010Numerical},  several  numerical methods were applied to solve      the reduced  FHD models  where   some nonlinear terms of the original models are dropped so as to get decoupled systems. 
In \cite{Nochetto2015The} Nochetto et al. showed the formal energy stability  of the system of Rosensweig's   model, devised an energy-stable numerical scheme using finite elements, and proved the existence  and, under some further simplifying assumptions, the convergence of the discrete solutions.  We also refer to \cite{Zhang.G;2021} for an unconditionally energy stable fully discrete finite element numerical scheme for a two-phase ferrohydrodynamics model.


In this paper, we shall develop a class of natural energy-preserving mixed finite element methods for Shliomis' FHD model, which describes the flow of an incompressible ferrofluid submitted to an external magnetic field. We first reformulate Shliomis'  model into an equivalent formulation. By choosing proper finite element spaces, the semi-discrete scheme preserves the structure of the continuous equation naturally, thus the energy-preserving property holds for the scheme. We prove that the semi-discrete scheme has a unique solution under  reasonable assumptions, and derive   optimal  error estimates
 by excavating deeply the properties of the $H(\curl)$, $H(\div)$ and $L^2$ finite element spaces; see Lemmas \ref{lem:LpK} and \ref{lem:vhbound}. Second, we apply  the implicit Euler method to discretize the temporal derivatives in the semi-discrete scheme so as to obtain a fully  discrete scheme. 
  We show that the full discretization scheme admits,  similar as the semi-discretization, an energy-preserving property and  has at least one solution. We also obtain   optimal  error estimates for the fully discrete scheme.  We mention that Shliomis'  model  involves  a quadrilinear term that Rosensweig's model does not have, and this will lead to more difficulties in the mathematical analysis.
  

The rest of this paper is organized as follows. In section 2, we introduce several Sobolev spaces, give the governing equations of    Shliomis' FHD model, and construct the weak formulations. In section 3, we recall the finite element spaces and derive some properties of these  spaces. We will show that the semi-discrete scheme preserves an energy similar as the continuous one and give the optimal order error estimates for the semi-discrete scheme. Section 4 will give the full discrete scheme and give the energy estimates and the optimal order error estimates of the full discrete scheme. 

\section{Preliminary}
\subsection{Sobolev spaces}
Let $\Omega \subset \mathbb R^3$ be a bounded and simply connected convex domain with  Lipschitz boundary $\partial\Omega$, and $T>0$ be     the final time. We set $\Omega_T: = \Omega \times (0,T]$ and $\Gamma_T := \partial\Omega \times (0,T]$. Let $\bs n$  be the unit outward normal vector on $\partial\Omega$.

 For any nonnegative integer $m$, we denote by  $H^m(\Omega)$   the usual $m$-th order Sobolev space   with norm $||\cdot||_{m}$ and semi-norm $|\cdot|_{m}$.   In particular, $H^0(\Omega)=L^2(\Omega)$
  denotes  the space of all square integrable 
  functions on $\Omega$,  with the inner product  $(\cdot,\cdot)$   and the  norm $\|\cdot\|$.  For the vector spaces  $(H^m(\Omega))^3$ and $(L^2(\Omega) )^3$, we use the same notations of norm, semi-norm and inner product as those for the scalar cases. We also introduce the spaces
  $$H(\curl) = \{\bs v \in (L^2(\Omega))^3:\ \curl\bs v \in (L^2(\Omega))^3\}$$ 
  and $$ H(\div) = \{\bs v \in (L^2(\Omega))^3:\ \div\bs v \in L^2(\Omega)\},$$ 
 with the norms 
$$
\|\bs v\|_{\curl} = \left( \|\bs v\|^2 + \|\curl\bs v\|^2\right)^{1/2},\qquad\text{and}\qquad \|\bs v\|_{\div} = \left(\|\bs v\|^2 + \|\div\bs v\|^2 \right)^{1/2}.
$$
We set 
  %
\begin{align*}
\bs S & := (H_{0}^{1}(\Omega))^{3} = \{ \bs v \in (H^{1}(\Omega))^{3}:\ \bs v = 0\text{  on  }\partial\Omega\},\\
\bs U &: =  H_{0}(\curl) = \{\bs v \in  H(\curl):\ \bs v \times \bs n = 0\text{  on  }\partial\Omega\},\\
\bs V & :=  H_{0}(\div) = \{\bs v \in  H(\div):\ \bs v \cdot\bs n = 0\text{  on  }\partial\Omega\},\\
W & := L_{0}^{2}(\Omega) = \{v \in L^{2}(\Omega):\ \int_{\Omega}v \dd\bs x = 0\},
\end{align*}
where $\curl\bs v = (\partial_yv_3 - \partial_zv_2, \partial_z v_1 - \partial_x v_3, \partial_x v_2 - \partial_y v_1)^\intercal$, $\div\bs v = \partial_x v_1 + \partial_y v_2 + \partial_z v_3$ for  $\bs v = (v_1,\ v_2,\ v_3)^\intercal$.

For  any scalar- or vector-valued  space $X$, defined on $\Omega$,  with norm $||\cdot||_X$,   we set
\begin{align*}
	L^p([0,T];X):=\left\{v:[0,T]\rightarrow X;\ ||v||_{L^p(X)}<\infty\right\},
\end{align*}
where
\begin{align*}
	||v||_{L^p( X)}:=\left\{ \begin{array}{ll}
	(\int_{0}^{T}||v(\cdot,t)||_X^p)^{1/p} & \text{ if }1\leq p< \infty,\\
	\text{ess}\sup\limits_{0\leq t\leq T}||v(\cdot,t)||_X  & \text{ if } p=\infty.
	\end{array}
	\right.
\end{align*}	
For simplicity, we set $L^p(X):=L^p(0,T; X)$.  For any integer $r\geq0$, the spaces $H^r(X):=H^r(0,T; X)$ and   $C^r(X):=C^r([0,T];X)$  can be defined similarly.

\subsection{Governing equations of the ferrofluid flow}

Consider the flow of an incompressible and viscous Newtonian ferrofluid, filling $\Omega$, under the action of a known external magnetic field $\bs H_{e}$ satisfying $$\bs H_{e}\cdot\bs n = 0 \quad \text{on } \Gamma_T.$$
 The magnetic field $\bs H_{e}$ induces a demagnetizing field $\bs H$ and a magnetic induction $\bs B$ satisfying the law $$\bs B = \bs H +  \bs m, $$
  with $\bs m$ the magnetization inside $\Omega$.
The governing equations of Shliomis'   model  \cite{Shliomis1972Effective,Shliomis2002Ferrofluids} for this FHD flow read as follows: the  fluid velocity $\bs u$,   the fluid pressure $p$, the interior magnetization $\bs m$, and the demagnetizing field $\bs H$ satisfy 
\begin{equation}\label{eq:FHD}
 \left\{
\begin{array}{l}
  \partial_t\bs u + (\bs u \cdot\nabla)\bs u - \eta\Delta\bs u + \nabla p  - \mu_0(\bs m\cdot\nabla)\bs H 
 -\frac{\mu_0}{2}\curl(\bs m\times\bs H) = 0, \\
 \div\bs u = 0 ,\\
\partial_t\bs m + (\bs u\cdot\nabla)\bs m - \sigma\Delta\bs m   - \frac{1}{2} \curl\bs u \times \bs m  
 + \frac{1}{\tau}(\bs m - \chi_0\bs H) 
  + \beta\bs m \times (\bs m \times \bs H)  = 0,\\
  \curl\bs H = 0,\\
   \div(\bs H + \bs m) = -\div\bs H_{e}
\end{array}
\right.
\end{equation}
%
 in $\Omega_T$, equipped with the boundary conditions
\begin{align}\label{eq:bd-1}
&\bs u  = 0,\quad\curl\bs m \times \bs n = 0,\quad \bs m\cdot\bs n = 0, \quad \text{and } \ \bs H \cdot\bs n  = 0 \qquad \text{on }\Gamma_T,
\end{align}
and the initial conditions
\begin{equation}
\label{eq:initial}
\bs u(\bs x,0) = \bs u_0(\bs x),\qquad \bs m(\bs x,0) = \bs m_0(\bs x)\qquad\forall~~~\bs x\in \Omega.
\end{equation}
Here 
$\bs u_0$ and $\bs m_0$ are given functions, 
and the parameters $\eta,\ \mu_0,\ \sigma,\ \tau,\ \chi_0$ and $\beta$ are positive constants and their physical meanings can be found in, for example, \cite{GaspariBloch,Shliomis1972Effective,Shliomis2002Ferrofluids,Torrey1956Bloch}. 

Introduce
\begin{equation}\label{z-k-p}
\bs z: = \bs u \times \bs m, \qquad \bs k: = \curl\bs m, \qquad \tilde p: = p + \frac{1}{2}|\bs u|^2 + \frac{1}{2}\bs m\cdot \bs H,
\end{equation}
 and let   $\varphi \in C^0(H_0^1(\Omega))$ be such that $\bs H=\nabla \varphi$ due to  the fact   $\curl\bs H = 0$. Thus, applying the identities 
\begin{equation*}
 \left\{
\begin{array}{rll}
(\bs u\cdot\nabla)\bs u   &= \frac{1}{2}\nabla(\bs u\cdot\bs u) - \bs u \times \curl\bs u, &\\
(\bs u\cdot\nabla)\bs m  & = \frac{1}{2}\bs u \div\bs m - \frac{1}{2}\bs m\div\bs u + \frac{1}{2}\nabla(\bs u\cdot\bs m) - \frac{1}{2}\bs m \times \curl\bs u\\
 &\qquad  - \frac{1}{2} \bs u\times \curl\bs m - \frac{1}{2}\nabla\times (\bs u \times\bs m),&\\
((\bs m\cdot\nabla)\bs H,\bs v)  & = -(\bs H\cdot\bs v,\div\bs m) - ((\bs m\cdot\nabla)\bs v,\bs H) \\
\ \ & = -\frac{1}{2}(\bs H\cdot\bs v,\div\bs m)-\frac{1}{2}(\bs m\cdot\bs H,\div\bs v) + \frac{1}{2}(\bs v\cdot\bs m,\div\bs H)\\
 & \qquad + \frac{1}{2}(\bs v\times(\curl\bs m),\bs H) + \frac{1}{2}(\bs m\times \curl \bs v,\bs H)
 & \forall~~\bs v \in \bs S
\end{array}
\right.
\end{equation*}
leads to the following weak problem: 
Find 
{  $\bs u \in \mathcal C^1(\bs S)$, $\tilde p\in \mathcal C^0(W)$, $\bs m \in  \mathcal C^1(\bs V)$,   $\bs z\in \mathcal C^0(\bs U)$, $\bs k \in \mathcal C^0(\bs U)$, $\bs H \in  \mathcal C^0(\bs V)$, and $\varphi \in C^0(W)$} such that
\begin{equation}\label{eq:FHD-weak}
 \left\{
\begin{array}{rll}
 (\partial_t\bs u,\bs v)  - (\bs u\times \curl\bs u,\bs v) + \eta(\nabla\bs u,\nabla\bs v)  - (\tilde p,\div \bs v) \quad\qquad &&\\
 -\mu_0b(\bs v;\bs H,\bs m) -\frac{\mu_0}{2}(\bs v\times \bs k,\bs H)  &= 0
&\quad\forall~\bs v \in \bs S,\\
(\div\bs u,q)   &= 0 &\quad\forall~q\in W, \\ 
 (\partial_t\bs m,\bs F) + b(\bs u;\bs m,\bs F) + \frac{1}{2}(\bs u\times\bs F,\bs k) 
  -\frac{1}{2} (\curl\bs z,\bs F) \quad\qquad && \\ 
 + \sigma(\curl\bs k,\bs F)   + \sigma(\div\bs m, \div\bs F) +\frac{1}{\tau}(\bs m,\bs F) \quad\qquad &&\\
 - \frac{\chi_0}{\tau}(\bs H,\bs F) + \beta(\bs m\times (\bs m \times H),\bs F) &= 0
& \quad\forall~\bs F\in \bs V,\\
 (\bs z,\bs\zeta)  - (\bs u\times \bs m,\bs \zeta)&=0 &\quad\forall~\bs\zeta\in \bs U,\\
(\bs k,\bs\kappa) - (\bs m,\curl\bs\kappa)&=0 & \quad\forall~\bs\kappa\in \bs U,\\
(\bs H,\bs G) + (\varphi, \div \bs G) &= 0 &\quad  \forall~ \bs G\in \bs V,\\
(\div\bs H + \div\bs m, r)+(\div \bs H_{e},r)&=0& \quad\forall~r \in W,
\end{array}
\right.
\end{equation}
with the  initial data \eqref{eq:initial}, where $$b(\bs v;\bs m,\bs F):  = \frac{1}{2}(\bs v\cdot\bs F,\div\bs m) - \frac{1}{2}(\bs v\cdot\bs m,\div\bs F).$$
It is obvious that  
\begin{equation}\label{b-vFF}
b(\bs v;\bs F,\bs F)=0, \quad \forall \bs v \in \bs S, \ \bs F\in \bs V.
\end{equation}

For any $t \in [0,T]$,  define the energy 
\begin{align*}
\mathcal E(t) & := \|\bs u(\cdot,t)\|^2 + \|\bs m(\cdot,t)\|^2 + \mu_0\|\bs H(\cdot,t)\|^2.
\end{align*}
To derive an energy estimate, we introduce an inequality first.
\begin{lemma}[\cite{Kirby2015}, Lemma 1]\label{lem:geq}
Suppose that a nonnegative real number $x$ satisfies the quadratic inequality
$$
x^2 \leq \gamma^2 + \beta x
$$
for $\beta,\ \gamma \geq 0$. Then
$$
x \leq \beta + \gamma.
$$
\end{lemma}

We have the following energy estimate.
\begin{theorem}\label{the:energy-con}
Given $\bs H_{e} \in H^{1}(H(\div))$,  let $\bs u \in \mathcal C^1(\bs S)$, $\tilde p\in \mathcal C^0(W)$, $\bs m \in  \mathcal C^1(\bs V)$,   $\bs z\in \mathcal C^0(\bs U)$, $\bs k \in \mathcal C^0(\bs U)$, $\bs H \in  \mathcal C^0(\bs V)$, and $\varphi \in C^0(W)$ solve the weak problem \eqref{eq:FHD-weak}. Then the  energy inequality
$$
\mathcal E(t) + C_1\int_0^t \mathcal F(s)\dd s \leq \mathcal E(0) + C_2\int_0^t \left(\|\bs H_{e}(\cdot,s)\|_{H(\div)}^2 + \|\partial_t \bs H_{e}(\cdot,s)\|^2\right) \dd s
$$
holds for all $t \in [0,T]$, where $C_1$ and $C_2$ are positive constants depending only on $\eta,\ \mu_0,\, \chi_0,\,\tau,\,\beta,\,\sigma$ and $\Omega$, and the dissipated energy $\mathcal F(t)$ is given by
\begin{align*}
\mathcal F(t) & := 2 \left(\eta\|\nabla \bs u\|^2 + \sigma(1+\mu_0)\|\div\bs m\|^2 + \sigma\|\bs k\|^2 + \frac{1}{\tau}\|\bs m\|^2\right.\\
& \quad \qquad+ \left. \frac{1}{\tau}\left[\mu_0(1+\chi_0)+\chi_0 \right]\|\bs H\|^2 + \mu_0\beta\|\bs m \times \bs H\|^2\right).
\end{align*}
\end{theorem}
\begin{proof}
Taking $\bs v = \bs u$ in the first equation of \eqref{eq:FHD-weak}, we get
\begin{align*}
\frac{1}{2}\frac{\dd }{\dd t}\|\bs u\|^2 + \eta\|\nabla\bs u\|^2 & = \mu_0 b(\bs u;\bs H,\bs m)+\frac{\mu_0}{2}(\bs u\times \bs k,\bs H).
\end{align*}
Taking $\bs F = \bs H$ in the third equation of \eqref{eq:FHD-weak} and $\bs G = \curl\bs z$ and $\bs G = \curl\bs k$ in the sixth equation, and using the fact that $\div\curl = \bs 0$,
we obtain
\begin{align*}
\mu_0 b(\bs u;\bs m,\bs H) + \frac{\mu_0}{2}(\bs u \times \bs H,\bs k)   = &-\mu_0(\partial_t\bs m,\bs H) - \mu_0\sigma(\div\bs m,\div\bs H) \\
&  - \frac{\mu_0}{\tau}(\bs m,\bs H) + \frac{\mu_0\chi_0}{\tau}\|\bs H\|^2 + \mu_0\beta\|\bs m \times\bs H\|^2.
\end{align*}
Therefore, 
\begin{equation}
\label{eq:ener-pro}
\begin{split}
\frac{1}{2}\frac{\dd }{\dd t}\|\bs u\|^2 + \eta\|\nabla\bs u\|^2 & = \mu_0(\partial_t\bs m,\bs H) + \mu_0\sigma(\div\bs m,\div\bs H) + \frac{\mu_0}{\tau}(\bs m,\bs H)\\
& \quad - \frac{\mu_0\chi_0}{\tau}\|\bs H\|^2 - \mu_0\beta\|\bs m \times\bs H\|^2.
\end{split}
\end{equation}
Taking $r = \varphi$ in the last equation of \eqref{eq:FHD-weak}, we have
\begin{equation}\label{last}
(\div\bs H + \div\bs m,\varphi) + (\div \bs H_e,\varphi) = 0.
\end{equation}
Taking $\bs G = \bs H$, $\bs G = \bs m$ and $\bs G = \bs H_e$ in the sixth equation of \eqref{eq:FHD-weak}, respectively, we obtain
$$
\|\bs H\|^2 + (\varphi,\div\bs H) = 0,\quad (\bs m,\bs H) + (\varphi,\div\bs m) = 0,\quad (\bs H,\bs H_e) + (\varphi,\div\bs H_e) = 0.
$$
Thus,
$$
(\bs m,\bs H) + \|\bs H\|^2 = -(\bs H_{e},\bs H).
$$
The fact that $\div H_0(\div) = L_0^2(\Omega)$ implies that 
\begin{equation}\label{eq:divH}
\div(\bs H + \bs m) = -\div\bs H_{e}.
\end{equation}
Differentiating \eqref{eq:divH} with respect to $t$, multiplying the resultant equation by $\varphi$, and using the fact that
\begin{align*}
& (\bs H,\partial_t\bs H) + (\div\partial_t\bs H,\varphi) = 0,\quad (\partial_t\bs m,\bs H) + (\varphi,\div\partial_t\bs m) = 0,\\
& (\bs H,\partial_t\bs H_e) + (\varphi,\div\partial_t\bs H_e) = 0,
\end{align*}
 we arrive at 
$$
(\partial_t\bs m,\bs H) = -\frac{1}{2}\frac{\dd}{\dd t}\|\bs H\|^2 - (\partial_t\bs H_{e},\bs H),
$$
which, together with \eqref{eq:ener-pro}, yields

\begin{equation}\label{eq:ener-pro-1}
\begin{array}{ll}
&{\small \frac{1}{2}\frac{\dd }{\dd t}\left(\|\bs u\|^2 + \mu_0\|\bs H\|^2\right) + \eta\|\nabla\bs u\|^2 + \mu_0\sigma\|\div\bs m\|^2 + \frac{\mu_0}{\tau}(1 + \chi_0)\|\bs H\|^2+ \mu_0\beta\|\bs m\times\bs H\|^2 } \\
=& -\frac{\mu_0}{\tau}(\bs H_{e},\bs H) - \mu_0(\partial_t\bs H_{e},\bs H)
 -\mu_0\sigma(\div\bs H_{e},\div\bs m). 
\end{array}
\end{equation}

Taking $\bs F = \bs m$ in the third equation of  \eqref{eq:FHD-weak}, $\bs\zeta = \bs k$ in the fourth equation, $\bs\kappa = \bs z$ and $\bs\kappa = \bs k$ in the fifth equation,  combining the resultant equations with \eqref{last},  and using \eqref{b-vFF}, 
we get
\begin{align*}
\frac{1}{2}\frac{\dd}{\dd t}\|\bs m\|^2 + \sigma\|\bs k\|^2 + \sigma\|\div\bs m\|^2 + \frac{1}{\tau}\|\bs m\|^2 + \frac{\chi_0}{\tau}\|\bs H\|^2 & = - \frac{\chi_0}{\tau}(\bs H_{e},\bs H).
\end{align*}
This relation plus  \eqref{eq:ener-pro-1} implies
\begin{equation*}
\label{eq:pr-ener-fin}
\begin{split}
&\frac{1}{2}\frac{\dd}{\dd t}  (\|\bs u\|^2 + \mu_0\|\bs H\|^2 + \|\bs m\|^2) + \eta\|\nabla\bs u\|^2 + \sigma(1+\mu_0)\|\div\bs m\|^2 + \frac{1}{\tau}\|\bs m\|^2\\
&\quad  + \frac{1}{\tau}\left[ \mu_0(1+\chi_0) + \chi_0\right] \|\bs H\|^2 + \sigma\|\bs k\|^2 + \mu_0\beta\|\bs m \times \bs H\|^2\\
= & -\frac{\mu_0 + \chi_0}{\tau}(\bs H_{e},\bs H) - \mu_0(\partial_t\bs H_{e},\bs H) - \mu_0\sigma(\div\bs H_{e},\div\bs m).
\end{split}
\end{equation*}
Integrate this equation with respect to $t$ on the interval $(0,t)$ and use the Cauchy-Schwarz inequality and Lemma \ref{lem:geq},  we then obtain  the desired result.
\end{proof}

\subsection{Finite element spaces and properties}

We consider some   $H^1-, H(\text{div})-$ and $H(\text{curl})-$ conforming finite element spaces that will be used in the spatial discretization of the weak problem \eqref{eq:FHD-weak}.

Let $\mathcal T_h$ be a quasi-uniform shape regular tetrahedron triangulation of $\Omega$ with mesh size $h: = \max\limits_{K\in \mathcal{T}_h}h_K$, where,  for any  $K\in \mathcal{T}_h$, $h_K$   denotes   its diameter.  For an integer  $l\geq 0$, let  ${\mathbb P}_l(K)$  be the set of polynomials,  defined on $K$, of degree no more than $l$. 

For convenience, throughout the paper we use $a\lesssim b$ ($a\gtrsim b$) to denote $a\leq Cb$ ($a\geq Cb$), where $C$ is a generic positive constant independent of the mesh size $h$ and may be different at its each occurrence. 

We  introduce the following finite dimensional spaces: 
\begin{itemize}
  \item $\bs S_{h} = (S_{h})^{3} \subset \bs S= (H_{0}^{1}(\Omega))^{3}$, where $S_{h}$ is a Lagrange-element space  \cite{Ciarlet1978} for the velocity   $\bs u$, with $ S_{h}|_K\supset {\mathbb P}_{l+1}(K)$ for any  $K$ ;
  \item $\Sigma_{h}\subset H_{0}^{1}(\Omega)$ is a Lagrange-element space with $\Sigma_h|_K\supset {\mathbb P}_l(K)$ for any $K$; $\Sigma_{h}$ may be as same as $S_{h}$;
  \item $\bs U_{h} \subset \bs U=  H_{0}(\curl) $ is an edge-element space \cite{Nedelec1980,Nedelec1986} for the new variables $\bs z $ and $\bs k$, with $ \bs U_h|_K \supset {\mathbb P}_l(K)^3$ for any $K$;
  \item $\bs V_{h} \subset \bs V=  H_{0}(\div) $ is a face-element  space  \cite{Raviart;Thomas1977,Nedelec1980,Brezzi;Douglas;Marini1985,Nedelec1986,Brezzi;Douglas;Duran;Fortin1987,Brezzi;Fortin1991} for the    interior magnetization $\bs m$ and the demagnetizing field $\bs H$, with $ \bs V_h|_K \supset {\mathbb P}_l(K)^3$ for any $K$;

  \item  $W_{h} \subset W=L_0^2(\Omega)$ is  a   piecewise polynomial space for the new variable $\varphi $, with $W_h|_K  \supset {\mathbb P}_l(K)$ for any $K$;

  \item  $L_{h} \subset W$ is  a   piecewise polynomial space for the modified pressure variable $\tilde p$,    with $L_h|_K  \supset {\mathbb P}_l(K)$ for any $K$; In some cases one may take $L_h=W_h$. 
\end{itemize}
In addition,   we make the following assumptions for the above spaces.
\begin{itemize}
 \item[(A1)]    
 The   
 diagram 
\begin{equation}\label{eq:exact_sq}
\begin{CD}
H_0^1@>{\grad}>> \bs U @>{\curl}>> \bs V @>{\div}>> W \\
@VV \pi_h V @ VV \pi ^{c}_h V  @VV \pi ^d_h V @ VV Q_h V \\\
\Sigma_{h} @>{\grad}>> \bs U_{h} @>{\curl}>>  \bs V_{h} @>{\div}>> W_{h}
\end{CD}
\end{equation}
is a commutative exact sequence in the sense that $$\ker(\curl) = \img(\grad), \quad \ker(\div) = \img(\curl).$$
Here $\pi_{h}:\ H_{0}^{1}(\Omega)\rightarrow \Sigma_{h}$, $\pi_{h}^{c}:\ \bs U \rightarrow \bs U_{h}$ and $\pi_{h}^{d}:\ \bs V\rightarrow \bs V_{h}$ are the classical interpolation operators, and $Q_{h}:\ W\rightarrow W_{h}$ is the $L^{2}$ orthogonal projection operator. 
Note that the diagram \eqref{eq:exact_sq} also indicates that 
$$ \grad \Sigma_{h} \subset \bs U_{h}, \quad  \curl \bs U_{h} \subset \bs V_{h}, \quad \div  \bs V_{h} = W_{h}.$$

\item[(A2)]   There holds the  inf-sup condition
\begin{equation}
\label{eq:inf-sup}
\sup\limits_{\bs v_h \in\bs S_h}\frac{(q_h,\div\bs v_h)}{\|\bs v_h\|_1} \gtrsim\|q_h\| \qquad\forall~~~q_h \in L_h.
\end{equation}
\end{itemize}

Note that  there are many combinations of finite element spaces satisfy (A1)  and (A2) (cf. \cite{Brezzi;Fortin1991,Hiptmair2002}).

We define   three discrete weak operators, $$\div_{h}:\bs U_{h} \rightarrow \Sigma_{h},\quad  \curl_{h}:\bs V_{h} \rightarrow \bs U_{h}, \quad \grad_{h}:\ W_{h}\rightarrow \bs V_{h},$$ as the adjoint operators of $-\grad$, $\curl$ and $-\div$, respectively, i.e.,  for  any $\bs u_{h}\in \bs U_{h} $, 
  $\div_{h}\bs u_{h}\in \Sigma_{h}$ satisfies 
\begin{equation}\label{eq:weak-div}
(\div_{h}\bs u_{h},s_{h}) := -(\bs u_{h},\grad s_{h})\qquad\forall~~s_{h} \in \Sigma_{h};
\end{equation}
for any $\bs v_{h}\in \bs V_{h} $,  $\curl_{h} \bs v_{h} \in \bs U_{h}$ satisfies
\begin{equation}\label{eq:weak-curl}
(\curl_{h}\bs v_{h},\bs u_{h}) := (\bs v_{h},\curl\bs u_{h})\qquad\forall~~\bs u_{h} \in \bs U_{h}; 
\end{equation}
and for any $w_{h}\in W_{h} $,
$\grad_{h} w_{h} \in \bs V_{h}$ satisfies 
\begin{equation}
\label{eq:weak-grad}
(\grad_{h} w_{h},\bs v_{h}) := -(w_{h},\div\bs v_{h})\qquad\forall~~ v_{h} \in \bs V_{h}.
\end{equation}
Thus, we have the following reversed ordering exact sequence:
\begin{equation}\label{eq:exact-seq-re}
\begin{CD}
0@<{}<<\Sigma_h @<{\div_h}<< \bs U_h @<{\curl_h}<< \bs V_h @<{\grad_h}<< W_h@<{}<<0
\end{CD}
\end{equation}

Introduce the null spaces of the differential operators
$$
\mathfrak Z_{h}^{c}: = \bs U_{h} \cap \ker(\curl)\quad\text{and}\quad \mathfrak Z_{h}^{d}: = \bs V_{h}\cap \ker(\div),
$$
and the null spaces of the weak differential operators
$$
\mathfrak K_{h}^{c} := \bs U_{h}\cap \ker(\div_{h})\quad\text{and}\quad \mathfrak K_{h}^{d} := \bs V_{h} \cap \ker(\curl_{h}).
$$
Similarly, we use the notations $\mathfrak Z^{c}$ and $\mathfrak Z^{d}$ to denote the null spaces in the continuous level, i.e.
$$
\mathfrak Z^{c} := \bs U \cap \ker(\curl),\quad  \mathfrak Z^{d}: = \bs V\cap \ker(\div).$$
 We also define 
$$
\mathfrak K^{c} := \bs U \cap (\mathfrak Z^{c})^{\bot} \quad\text{and}\quad \mathfrak K^{d} := \bs V \cap (\mathfrak Z^{d})^{\bot}
$$
Let  $\oplus^\bot$ stand for the $L^2$ orthogonal decomposition, then
we have the following Hodge decompositions \cite{Arnold;Falk;Winther2006,Arnold;Falk;Winther2010}:
\begin{align}
\label{eq:hodge-dis-1}
\bs U_{h} & = \mathfrak Z_{h}^{c} \oplus ^{\bot}\mathfrak K_{h}^{c}= \grad\Sigma_{h} \oplus^{\bot} \div_{h}\bs U_{h},\\ 
\label{eq:hodge-dis-2}
\bs V_{h} &= \mathfrak Z_{h}^{d} \oplus^{\bot}\mathfrak K_{h}^{d} = \curl\bs U_{h} \oplus^{\bot} \grad_{h}W_{h}.
\end{align}
And there hold the following discrete Poincar\'e inequalities \cite{Arnold;Falk;Winther2006,Arnold;Falk;Winther2010,Chen2016MultiGrid}:
\begin{equation}\label{eq:P-1}
\|\curl \bs u_{h}\| \gtrsim  \|\bs u_{h}\| \qquad \forall~~\bs u_{h} \in \mathfrak K_{h}^{c},
\end{equation}
\begin{equation}\label{eq:P-2}
\|\div\bs v_{h}\| \gtrsim  \|\bs v_{h}\|\qquad\forall~~\bs v_{h} \in \mathfrak K_{h}^{d}.
\end{equation}
\begin{equation}\label{eq:P-3}
\|\grad_{h} w_{h} \|\gtrsim  \|w_{h}\|\qquad\forall~~w_{h} \in W_{h},
\end{equation}
\begin{equation}\label{eq:P-4}
\|\curl_{h} \bs v_{h}\| \gtrsim \|\bs v_{h}\|\qquad\forall~~\bs v_{h} \in \mathfrak Z_{h}^{d}.
\end{equation}

For any  $s\in H_{0}^{1}(\Omega)$, $\bs u \in \bs U$ and $\bs v \in \bs V$, we define $P_{h}^{g} s \in \Sigma_{h}$, $P_{h}^{c} \bs u \in \mathfrak K_{h}^{c}$ and $P_{h}^{d} \bs v \in \mathfrak K_{h}^{d}$ such that
\begin{equation}\label{eq:pj}
(\grad P_{h}^{g}s,\grad \sigma_{h}) = (\grad s,\grad\sigma_{h})\qquad \forall~~\sigma_{h} \in \Sigma_{h},
\end{equation}
\begin{equation}\label{eq:project-1}
(\curl P_{h}^{c}\bs u,\curl \bs\phi_{h}) = (\curl \bs u,\curl\bs \phi_{h})\qquad\forall~~\bs\phi_{h} \in \mathfrak K_{h}^{c}
\end{equation}
and
\begin{equation}\label{eq:project-2}
(\div P_{h}^{d} \bs v,\div \bs\psi_{h} ) = (\div\bs v,\div\bs\psi_{h})\qquad\forall~~\bs\psi_{h} \in \mathfrak K_{h}^{d}.
\end{equation}
Equation \eqref{eq:pj} determine $P_h^g s \in \Sigma_h$ uniquely due to the Poincar\'e inequality in $H_0^1(\Omega)$,  and \eqref{eq:project-1} and \eqref{eq:project-2} also determine $P_{h}^{c}\bs u \in \mathfrak K_{h}^{c}$ and $P_{h}^{d}\bs v \in \mathfrak K_{h}^{d}$ uniquely,  since the Poincar\'e inequalities \eqref{eq:P-1} and \eqref{eq:P-2} imply that $(\curl(\cdot),\curl(\cdot))$ and $(\div(\cdot),\div(\cdot))$ are inner products on the subspaces $\mathfrak K_{h}^{c}$ and $\mathfrak K_{h}^{d}$, respectively.

The Hodge decomposition on the continuous level implies that for any $\bs u \in \bs U$ and $\bs v \in \bs V$, there exist $u_{1} \in H_{0}^{1}(\Omega)$, $\bs u_{2} \in \mathfrak K^{c}$, $\bs v_{1} \in \mathfrak K^{c}$ and $\bs v_{2} \in \mathfrak K^{d}$ such that
$$
\bs u = \grad u_{1} \oplus^{\bot} \bs u_{2}\quad \text{and}\quad\bs v = \curl\bs v_{1} \oplus^{\bot} \bs v_{2}.
$$
We introduce two projection-based quasi-interpolation operators,  $I_{h}^{c}:\bs U \rightarrow \bs U_{h}$ and $I_{h}^{d}:\bs V \rightarrow \bs V_{h}$,   defined by
\begin{equation}\label{eq:inter-1}
I_{h}^{c} \bs u = \grad P_{h}^{g}\bs u_{1} \oplus^{\bot} P_{h}^{c}\bs u_{2}
\end{equation}
and
\begin{equation}
\label{eq:inter-2}
I_{h}^{d}\bs v = \curl P_{h}^{c}\bs v_{1} \oplus^{\bot} P_{h}^{d}\bs v_{2}.
\end{equation}
 The quasi-interpolation operators have the following properties: 
\begin{lemma}\cite[Lemmas 3.3-3.5]{Wu2020} \label{lem:Ih}
The projection-based quasi-interpolation operators $I_h^c$ and $I_h^d$ have the following properties:
\begin{enumerate}
\item For any $\bs u \in \bs U$ and $\bs v \in \bs V$, there hold
\begin{align*}
& (I_{h}^{c}\bs u,\grad s_{h}) = (\bs u,\grad s_{h})\qquad\qquad\qquad\forall~~s_{h} \in \Sigma_{h},\\
& (I_{h}^{d} \bs v,\curl\bs\phi_{h}) = (\bs v,\curl\bs\phi_{h})\,\qquad\qquad\qquad\forall~~\bs\phi_{h} \in \bs U_{h},\\
& (\curl I_{h}^{c}\bs u,\curl\bs\psi_{h}) = (\curl\bs u,\curl\bs\psi_{h})\,\,\,\qquad\forall~~\bs\psi_{h} \in \bs U_{h},\\
& (\div I_{h}^{d}\bs v,\div\bs\phi_{h}) = (\div\bs v,\div\bs\phi_{h})\qquad\qquad\forall~~\bs\phi_{h} \in \bs V_{h}.
\end{align*}

\item For any $\bs u \in \bs U$ and $\bs v\in \bs V$, there hold
$$
\|\curl  I_{h}^{c}\bs u\| \leq \|\curl\bs u\|,\qquad
\|\div I_{h}^{d}\bs v\| \leq \|\div\bs v\|.
$$

\item For any $\bs u \in  H_{0}(\curl) \cap  H(\div)$ and $\bs v \in  H_{0}(\div) \cap  H(\curl)$, there hold
  $$
  \div_{h}  I_{h}^{c} \bs u = Q_{h}^{g} \div\bs u, \quad \curl_{h}  I_{h}^{d} \bs v = Q_{h}^{c} \curl \bs v,
  $$
  where $Q_{h}^{g}:L^{2}(\Omega) \rightarrow \Sigma_{h}$ and $Q_{h}^{c}: \bs L^{2}(\Omega) \rightarrow \bs U_{h}$ are $L^{2}$ orthogonal projection operators. Therefore,
$$
\|\div_{h} I_{h}^{c} \bs u \|\leq \|\div\bs u\|, \quad\|\curl_{h} I_{h}^{d}\bs v\| \leq \|\curl\bs v\|.
$$ 

\item For any $\bs u \in \bs U\cap  H^{l+1}(\Omega)$ and $\bs v \in \bs V \cap  H^{l+1}(\Omega)$ with $l \geq 0$, there hold
\begin{align*}
\|\bs u -  I_{h}^{c}\bs u\|  \lesssim h^{r}\|\bs u\|_{r}\qquad\text{for}\quad 1\leq r \leq l+1,\\
\|\bs v -  I_{h}^{d}\bs v\|  \lesssim h^{r}\|\bs v\|_{r}\qquad\text{for}\quad 1\leq r \leq l+1.
\end{align*}
Furthermore, if $\curl\bs u\in H^{l+1}(\Omega)$ and $\div\bs v \in H^{l+1}(\Omega)$, there hold
\begin{align*}
\|\curl(I -  I_{h}^{c})\bs u\| \lesssim h^{r} \|\curl\bs v\|_{r}\qquad\text{for}\quad 1\leq r \leq l+1,\\
\|\div(I -  I_{h}^{d})\bs v\| \lesssim h^{r} \|\div\bs v\|_{r}\qquad\text{for}\quad 1\leq r \leq l+1.
\end{align*}
\end{enumerate}
\end{lemma}

We define $Q_{\mathfrak K}^c:\ \mathfrak K_h^c \rightarrow \mathfrak K^c$ and $Q_{\mathfrak K}^{d}:\mathfrak K_{h}^{d} \rightarrow \mathfrak K^{d}$ as the $L^{2}$ orthogonal  projections, i.e., for any $\bs s_h \in \mathfrak K_h^c$ and $\bs v_{h} \in \mathfrak K_{h}^{d}$, $Q_{\mathfrak K}^c\bs s_h \in \mathfrak K^c$ and $Q_{\mathfrak K}^{d} \bs v_{h} \in \mathfrak K^{d}$ satisfy 
$$
(Q_{\mathfrak K}^c \bs s_h,\bs\varphi) = (\bs s_h,\bs\varphi)\qquad \forall~~\bs\varphi \in \mathfrak K^c, 
$$
and
$$
(Q_{\mathfrak K}^{d}\bs v_{h},\bs\phi) = (\bs v_{h},\bs \phi)\qquad \forall~~\bs \phi\in \mathfrak K^{d},
$$
respectively.
The definitions of $Q_{\mathfrak K}^c\bs s_h$ and  $Q_{\mathfrak K}^{d}\bs v_{h}$ imply that 
$$Q_{\mathfrak K}^c\bs s_h = \bs s_h - \grad \gamma, \quad  Q_{\mathfrak K}^{d}\bs v_{h} = \bs v_{h} - \curl \bs\psi,$$
 where $\gamma \in H_0^1(\Omega)$ and $\bs\psi \in \mathfrak K^{c}$ are determined uniquely by  
$$
(\grad \gamma,\grad r) = (\bs s_h,\grad r)\,\,\,\,\qquad\forall~~ r \in H_0^1(\Omega),
$$
and
$$
(\curl\bs\psi,\curl\bs \kappa) = (\bs v_{h},\curl\bs\kappa)\qquad\forall~~\bs \kappa\in\mathfrak K^{c}, 
$$
respectively.
Therefore, we have  $$\curl Q_{\mathfrak K}^c\bs s_h = \curl\bs s_h, \quad \div Q_{\mathfrak K}^{d}\bs v_{h} = \div\bs v_{h}.$$
 Furthermore, we have the following error estimates. 
\begin{lemma}
For any $\bs s_h \in \mathfrak K_h^c$ and $\bs v_{h} \in \mathfrak K_{h}^{d}$, there hold
$$
\|\bs s_h - Q_{\mathfrak K}^c\bs s_h\| \lesssim h\|\curl\bs s_h\|,
$$
$$
\|\bs v_{h} - Q_{\mathfrak K}^{d}\bs v_{h}\| \lesssim h\|\div\bs v_{h}\|.
$$
\end{lemma}
\begin{proof}
We note that the proof of some special cases can be found in \cite{Arnold;Falk;Winther2000,Hiptmair2002,Monk2003}.
Set $\bs s = Q_{\mathfrak K}^c\bs s_h$,  and let $\bs \tau_h \in \mathfrak K_h^c$ be the solution of the equation
\begin{equation}\label{eq:d-1}
(\curl\bs\tau_h,\curl\bs \kappa_h) = (\bs s_h - \bs s,\bs\kappa_h) \qquad\forall~~\bs\kappa_h \in \mathfrak K_h^c.
\end{equation}
Since both $\mathfrak K_h^c$ and $\mathfrak K^c$ are $L^2$ orthogonal to $\mathfrak Z^c_h = \bs U_{h} \cap \ker(\curl)$, the test function space of \eqref{eq:d-1} can be enlarged to $\bs U_h$, i.e., the solution $\bs\tau_h$ of \eqref{eq:d-1} satisfies
$$
(\curl\bs \tau_h,\curl\bs\kappa_h) = (\bs s_h - \bs s,\bs\kappa_h)\qquad\forall\bs \kappa_h \in \bs U_h.
$$
Therefore, by Lemma \ref{lem:Ih} we have  
\begin{align*}
\|\bs s_h - \bs s\|^2 & = (\bs s_h - \bs s,\bs s_h - I_h^c\bs s + I_h^c\bs s -\bs s)\\
& = (\curl\bs\tau_h,\curl(\bs s_h - I_h^c\bs s)) + (\bs s_h - \bs s,I_h^c\bs s - \bs s) \\
& = (\bs s_h - \bs s, I_h^c\bs s - \bs s) \\
& \lesssim h\|\bs  s\|_1\|\bs s_h - \bs s\|.
\end{align*}
Then the first desired estimate follows from the fact that $\|\bs s\|_1 \lesssim \|\curl\bs s\| = \|\curl\bs s_h\|$.

The second estimate follows similarly. In fact,  let us set $\bs v = Q_{\mathfrak K}^{d}\bs v_{h}$,  and let $\bs w_{h} \in \mathfrak K_{h}^{d}$ be the solution of  
\begin{equation}
\label{eq:duality}
(\div\bs w_{h},\div\bs \phi_{h}) = (\bs v_{h} - \bs v,\bs\phi_{h})\qquad\forall~~\bs \phi_{h} \in \mathfrak K_{h}^{d}.
\end{equation}
Since both $\mathfrak K^{d}$ and $\mathfrak K_{h}^{d}$ are $L^{2}$ orthogonal to $\mathfrak Z_{h}^{d}= \bs V_{h}\cap \ker(\div)$, the test function space in \eqref{eq:duality} can be enlarged, i.e., the solution $\bs w_{h}$ of \eqref{eq:duality} satisfies
$$
(\div\bs w_{h},\div\bs \phi_{h}) = (\bs v_{h} - \bs v,\bs\phi_{h})\qquad\forall~~\bs \phi_{h} \in \bs V_{h}.
$$
Then  the desired result follows from
\begin{align*}
\|\bs v_{h} - \bs v\|^{2}& = (\bs v_{h} - \bs v,\bs v_{h} -  I_{h}^{d} \bs v +  I_{h}^{d}\bs v - \bs v) \\
& = (\div\bs w_{h},\div(\bs v_{h} -  I_{h}^{d} \bs v)) + (\bs v_{h} - \bs v, I_{h}^{d}\bs v - \bs v)\\
& = (\div\bs w_{h},\div(\bs v_{h} -  \bs v)) + (\bs v_{h} - \bs v, I_{h}^{d}\bs v - \bs v) \\
& = (\bs v_{h} - \bs v, I_{h}^{d}\bs v - \bs v) \\
& \lesssim h\|\bs v\|_{1} \|\bs v_{h} - \bs v\|
\end{align*}
and 
$
\|\bs v\|_{1} \lesssim \|\div\bs v\| = \|\div\bs v_{h}\|.
$
\end{proof}

Let us define three ``Hodge mappings'', $$Q_{\mathfrak Z}^c:\ \mathfrak Z_h^c \rightarrow \mathfrak Z^c \cap H(\div), \quad Q_{\mathfrak Z}^d:\ \mathfrak Z_h^d \rightarrow \mathfrak Z^d \cap H(\curl) \quad\text{and}\quad Q_{L_0^2}:\ W_h\rightarrow L_0^2(\Omega)\cap H^1(\Omega),$$
 as follows:
$$
(\div (Q_{\mathfrak Z}^c\bs s_h),\div\bs \tau) = (\div_h\bs s_h,\div\bs\tau)\quad\qquad\forall~~\bs s_h \in \mathfrak Z_h^c,\ \bs\tau\in \mathfrak Z^c\cap H(\div),
$$
$$
(\curl (Q_{\mathfrak Z}^d\bs v_h),\curl \bs\phi) = (\curl_h\bs v_h,\curl\bs\phi) \qquad\forall~~\bs v_h \in \mathfrak Z_h^d,\ \bs\phi\in \mathfrak Z^d\cap H(\curl),
$$
$$
(\grad(Q_{L_0^2} w_h),\grad q) = (\grad_h w_h,\grad q)\qquad\forall~~w_h \in W_h,\ q\in L_0^2(\Omega)\cap H^1(\Omega).
$$
\begin{lemma}\label{lem:sp-eq}
There holds
$$
\div (\mathfrak Z^c\cap H(\div)) = \div H(\div) = L^2(\Omega).
$$
\end{lemma}
\begin{proof}
The Hodge decomposition 
$$
L^2(\Omega)^3 = \curl H(\curl) \oplus^\bot \grad  H_0^1(\Omega) = \curl H(\curl) \oplus^\bot \mathfrak Z^c 
$$
implies
$$
H(\div) = L^2(\Omega)^3 \cap H(\div) = \curl H(\curl) \oplus^\bot \mathfrak Z^c \cap H(\div).
$$
Then 
$$
\div H(\div) = \div(\mathfrak Z^c \cap H(\div)),
$$
and the conclusion follows.
\end{proof}

We cite an  estimate from ~\cite[Lemma 2]{Hu;Xu2019}.
\begin{lemma}\label{lem:hmz}
If $\Omega\in \mathbb R^3 $ is a bounded convex polyhedral domain, then there holds
$$
\|\bs B_h - Q_{\mathfrak Z}^d\bs B_h\| \lesssim h\|\curl_h\bs B_h\|\qquad\forall\bs B_h \in \mathfrak Z_h^d.
$$
\end{lemma}

By a similar   proof as that of ~\cite[Lemma 2]{Hu;Xu2019}, we can obtain the following results.
\begin{lemma}
\label{lem:QZ}
If $\Omega$ is a bounded convex polyhedral domain in $\mathbb R^3$, there hold
$$
\|\bs s_h - Q_{\mathfrak Z}^c\bs s_h\| \lesssim h\|\div_h\bs s_h\| \qquad \forall~~\bs s_h \in \mathfrak Z_h^c,
$$
and 
$$
\|q_h - Q_{L_0^2} q_h\| \lesssim h\|\grad_h q_h\|\qquad\forall~~q_h \in W_h.
$$
\end{lemma}
\begin{proof}
For any $\bs s_h \in \mathfrak Z_h^c$, let $\bs s = Q_{\mathfrak Z}^c \bs s_h$, so $\curl(\bs s_h - \pi_{\curl}^h\bs s) = 0$, where $\pi_{\curl}^h$ is the bounded cochain projection to $\bs U_h$~\cite{Falk2014}, and there exists $r_h \in \Sigma_h$ such that $\bs s_h - \pi_{\curl}^h\bs s = \grad r_h$. The definition of $Q_{\mathfrak Z}^c$ and Lemma \ref{lem:sp-eq} imply
$$
(\bs s_h,\grad r_h) = -(\div_h\bs s_h, r_h) = -(\div\bs s, r_h) = (\bs s,\grad r_h).
$$
Thus,
\begin{align*}
\|\bs s_h - \bs s\|^2 & = (\bs s_h - \bs s,\bs s_h-\pi_{\curl}^h\bs s) + (\bs s_h - \bs s,\pi_{\curl}^h\bs s-\bs s) \\
& = (\bs s_h - \bs s,\grad r_h) + (\bs s_h - \bs s,\pi_{\curl}^h\bs s-\bs s)\\
& = (\bs s_h - \bs s,\pi_{\curl}^h\bs s-\bs s) \\
& \lesssim h\|\bs s\|_1 \|\bs s_h - \bs s\|.
\end{align*}
Then the  first desired conclusion follows from the fact that $\curl\bs s = 0$, $\mathfrak Z^c\cap H(\div) \hookrightarrow H^1(\Omega)$ and $\|\div\bs s\| \leq \|\div_h\bs s_h\|$.

The thing left is to show the second conclusion of this lemma. For any $q_h \in W_h$, let $q = Q_{L_0^2} q_h$. Note that there exists a unique $\bs v_h \in \mathfrak K_h^d$ such that $q_h - Q_h q = \div\bs v_h$. Consider the following auxiliary  problem: for any given $\bs f \in L^2(\Omega)^3$, find $w \in L_0^2(\Omega) \cap H^1(\Omega)$ such that
$$
(\grad w,\grad r) = (\bs f,\grad r) \qquad\forall~~r \in L_0^2(\Omega) \cap H^1(\Omega).
$$
The Poincar\'e inequality on the space 
indicates that this problem has a unique solution. 
This means
$$
L^2(\Omega)^3 \subset \grad(L_0^2(\Omega) \cap H^1(\Omega)),
$$
which, together with the definition of $Q_{L_0^2}$, yields
$$
(q_h,\div \bs v_h) = -(\grad_h q_h,\bs v_h) = -(\grad q,\bs v_h) = (q,\div\bs v_h).
$$
Therefore, there holds 
\begin{align*}
\|q_h - q\|^2 & = (q_h - q,q_h -Q_h q) + (q_h - q,Q_hq -q) \\& = (q_h - q,\div\bs v_h) + (q_h - q,Q_hq -q) \\
& \lesssim h\|\grad q\|\|q_h - q\|,
\end{align*}
and the   desired result follows from  the fact that $\|\grad q\| \leq \|\grad_h q_h\|$.
\end{proof}

We have several $L^p$ estimates  ($1\leq p\leq 6$) for the functions in the finite-dimensional spaces $\mathfrak K_h^c$, $\mathfrak K_h^d$ and $W_h$.
\begin{lemma}\label{lem:LpK}
If $\Omega \subset\mathbb R^3$ is a bounded convex Lipschitz polyhedral domain, then for any   integer $1\leq p\leq 6$, there hold
$$
\|\bs s_h\|_{L^p} \lesssim \|\div_h\bs s_h\|\,\,\qquad\forall~\bs s_h \in \mathfrak K_h^c,
$$
$$
\|\bs B_h\|_{L^p}\lesssim \|\curl_h\bs B_h\|\qquad \forall~\bs B_h \in \mathfrak K_h^d,
$$
and
$$
\|q_h\|_{L^p} \lesssim \|\grad_h q_h\|\,\,\qquad\forall~q_h \in W_h.
$$
\end{lemma}
\begin{proof}
For $p = 1,\ 2$ the desired results are just the discrete Poincar\'e inequalities ~\cite{Arnold;Falk;Winther2006,Arnold;Falk;Winther2010,Chen2016MultiGrid}. For $3\leq p \leq 6$,
the triangular inequality implies
$$
\|\bs s_h\|_{L^p} \leq \|\bs s_h - \pi_h^cQ_{\mathfrak Z}^c\bs s_h\|_{L^p} + \|\pi_h^cQ_{\mathfrak Z}^c \bs s_h\|_{L^p}.
$$
The inverse inequality leads to 
\begin{align*}
\|\bs s_h - \pi_h^cQ_{\mathfrak Z}^c\bs s_h\|_{L^p} & \lesssim h^{-3(\frac{1}{2} - \frac{1}{p})}\|\bs s_h - \pi_h^cQ_{\mathfrak Z}^c\bs s_h\| \\
& \leq h^{-3(\frac{1}{2} - \frac{1}{p})}(\|\bs s_h - Q_{\mathfrak Z}^c\bs s_h\| +\|Q_{\mathfrak Z}^c\bs s_h - \pi_h^c Q_{\mathfrak Z}^c \bs s_h\|)\\
& \lesssim \|\div_h\bs s_h\|.
\end{align*}
Using the stability of $\pi_h^c$ in the $L^p$ norm, we have
$$
\|\pi_h^cQ_{\mathfrak Z}^c\bs s_h\|_{L^p} \lesssim \|Q_{\mathfrak Z}^c\bs s_h\|_{L^p} \lesssim \|\bs s_h\|_{L^p} \lesssim \|\div_h\bs s_h\|.
$$
As a result,  the   desired result $\|\bs s_h\|_{L^p} \lesssim \|\div_h\bs s_h\|$  follows. 
Similarly, we can prove the other two inequalities.
\end{proof}

Furthermore, we have two estimates for functions in $\bs U_h$ and $\bs V_h$.
\begin{lemma}\label{lem:vhbound}
If $\Omega\subset \mathbb R^3$ is a bounded convex Lipschitz polyhedral domain, then for any $1\leq p \leq 6$, there hold
$$
\|\bs s_h\|_{L^p}  \lesssim \|\div_h\bs s_h\| + \|\curl\bs s_h\|\qquad\forall~~\bs s_h \in \bs U_h,
$$
and
$$
\|\bs v_{h}\|_{L^{p}} \lesssim \|\div\bs v_{h}\| + \|\curl_{h} \bs v_{h}\|\qquad \forall~~\bs v_h \in \bs V_h.
$$
\end{lemma}
\begin{proof}
For any $\bs s_h \in \bs U_h$, the Hodge decomposition implies that there exist $\bs s_{h,1}\in \mathfrak Z_{h}^{c}$ and $\bs s_{h,2} \in \mathfrak K_{h}^{c}$ such that
\begin{equation}\label{eq:pr-hod}
\bs s_{h} = \bs s_{h,1} \oplus^{\bot} \bs s_{h,2}.
\end{equation}
By Lemma \ref{lem:LpK}, we have
$$
\|\bs s_{h,1}\|_{L^{p}} \lesssim \|\div_{h} \bs s_{h,1}\| = \|\div_{h} \bs s_{h}\|.
$$
Denote $\bs s_{2} = Q_{\mathfrak K}^{c} \bs s_{h,2}$, and let $Q_{h}^{c}:\ L^{2}(\Omega) \rightarrow \bs U_{h}$ be the $L^{2}$ projection operator. Using the embedding result $H(\div) \cap  H_0(\curl) \hookrightarrow  H^{1}(\Omega)\hookrightarrow  L^{p}(\Omega)$, we obtain
\begin{align*}
\|\bs s_{h,2}\|_{L^{p}} &\leq \|Q_{h}^{c}(\bs s_{h,2} - \bs s_{2})\|_{L^{p}} + \|Q_{h}^{c} \bs s_{2}\|_{L^{p}}\\
&\lesssim h^{-3(\frac{1}{2} - \frac{1}{p})}\|Q_{h}^{c} (\bs s_{h,2} - \bs s_{2})\| + \|\bs s_{2}\|_{L^{p}}\\
& \lesssim  \|\curl\bs s_{h,2}\| \\
& = \|\curl\bs s_{h}\|.
\end{align*}
Then the first desired result follows from \eqref{eq:pr-hod} and the  triangular inequality. Similarly, we can prove the second inequality.
\end{proof}

Applying Lemma \ref{lem:vhbound}, we get the following properties of the projection-based quasi-interpolation operators $I_h^c$ and $I_h^d$.
\begin{lemma}
\label{lem:I-h-Linfty}
For any $\bs u \in L^\infty(\Omega) \cap \bs U$ and $\bs v \in L^\infty(\Omega) \cap \bs V$, there hold
$$
\|I_h^c\bs u\|_{L^\infty} \lesssim \|\bs u\|_{L^\infty}
$$
and
$$
\|I_h^d\bs v\|_{L^\infty} \lesssim \|\bs v\|_{L^\infty}.
$$
\end{lemma}
\begin{proof}
We only give the proof of the first inequality, since  the second one follows  similarly. The hodge decomposition implies
$$
\bs U = \mathfrak Z^c \oplus^\bot \mathfrak K^c,
$$
so
$$
L^\infty(\Omega) \cap \bs U = (\mathfrak Z^c \cap L^\infty(\Omega)) \oplus^\bot (\mathfrak K^c \cap L^\infty(\Omega)).
$$
This means there exist $\bs u_1\in \mathfrak Z^c \cap L^\infty(\Omega)$ and $\bs u_2 \in \mathfrak K^c \cap L^\infty(\Omega)$ such that 
$$
\bs u = \bs u_1 \oplus^\bot \bs u_2.
$$
The definition of $I_h^c$ and the first part of Lemma \ref{lem:Ih} imply that
$$
I_h^c\bs u = Q_{\mathfrak Z_h^c}\bs u_1 \oplus^\bot P_h^c\bs u_2,
$$
where $Q_{\mathfrak Z_h^c}:\ L^2(\Omega) \rightarrow \mathfrak Z_h^c$ is the $L^2$ orthogonal projection operator. Therefore, we have
\begin{equation}\label{le2.9-1}
\|I_h^c\bs u\|_{L^\infty} \leq \|Q_{\mathfrak Z_h^c}\bs u_1 \|_{L^\infty} + \|P_h^c\bs u_2\|_{L^\infty} \lesssim \|\bs u_1\|_{L^\infty} +\|P_h^c\bs u_2\|_{L^\infty}.
\end{equation}
For the term $\|P_h^c\bs u_2\|_{L^\infty}$, there holds
\begin{equation}\label{le2.9-2}
P_h^c\bs u_2 = P_h^c(I - \pi_h^c)\bs u_2 + P_h^c\pi_h^c\bs u_2.
\end{equation}
For any $\bs\kappa_h \in \mathfrak K_h^c$, the Poincar\'e inequality \eqref{eq:P-1} implies that there exists a unique $\bs\zeta_h \in \mathfrak K_h^c$ such that $\bs\kappa_h = \curl_h\curl \bs\zeta_h$. Thus, we obtain
\begin{align*}
(P_h^c\pi_h^c\bs u_2,\bs\kappa_h) & = (P_h^c\pi_h^c\bs u_2,\curl_h\curl\bs\zeta_h) = (\curl P_h^c\pi_h^c\bs u_2,\curl\bs\zeta_h) \\
& = (\curl \pi_h^c\bs u_2,\curl\bs\zeta_h) = (\pi_h^c\bs u_2,\bs\kappa_h).
\end{align*}
This implies that
$$
P_h^c\pi_h^c\bs u_2 = Q_{\mathfrak K_h^c} \pi_h^c\bs u_2,
$$
where $Q_{\mathfrak K_h^c}:\ L^2(\Omega)\rightarrow \mathfrak K_h^c$ is the $L^2$ orthogonal projection operator. Thus, 
\begin{equation}\label{le2.9-3}
\|P_h^c\pi_h^c\bs u_2\|_{L^\infty} \lesssim \|\bs u_2\|_{L^\infty}.
\end{equation}
For the term $P_h^c(I - \pi_h^c)\bs u_2$, there holds
\begin{align*}
\|P_h^c(I - \pi_h^c)\bs u_2\|_{L^\infty} & \lesssim h^{-1/2}\|P_h^c(I - \pi_h^c)\bs u_2\|_{L^6} \\
& \lesssim h^{-1/2}\|\curl P_h^c(I - \pi_h^c)\bs u_2\| \\
& \leq h^{-1/2}\|\curl (I - \pi_h^c)\bs u_2\| \\
& \lesssim \|\curl\bs u_2\|_{H^{1/2}} \lesssim \|\bs u_2\|_{L^\infty},
\end{align*}
where in the last inequality we have used the fact that $L^\infty \hookrightarrow H^{3/2}$.  This estimate, together with   \eqref{le2.9-1}-\eqref{le2.9-3}, yields   the desired conclusion.
\end{proof}

Using Lemma \ref{lem:vhbound}, we also have the following Lemma.

\begin{lemma}
\label{lem:K-h}
For any $\bs\psi_{h} \in \mathfrak K_{h}^{d}$ and $\bs\phi_{h} \in \bs V_{h}$, there holds
$$
\|\bs\psi_{h} \times\bs\phi_{h}\| \lesssim \|\div\bs\psi_{h}\|(\|\div\bs\phi_{h}\| + \|\curl_h\bs\phi_{h}\|)
$$
\end{lemma}
\begin{proof}
Since $\bs\psi_{h} \in \mathfrak K_{h}^{d}$, we have $\curl_{h}\bs\psi_{h} = 0$. Let $\bs\psi = Q_{\mathfrak K}^{d} \bs\psi_{h}$, then
\begin{align*}
\|\bs\psi_{h} \times \bs\phi_{h}\| & \leq \|\bs\phi_{h} \times (\bs\psi_{h} - \bs\psi)\| + \|\bs\phi_{h} \times \bs\psi\|.
\end{align*}
The first term can be bounded as 
\begin{align*}
\|\bs\phi_{h} \times (\bs\psi_{h} - \bs\psi)\| & \leq \|\bs\phi_{h}\|_{\infty} \|\bs\psi_{h} - \bs\psi\| \lesssim h\|\bs\phi_{h}\|_{\infty} \|\div\bs\psi_{h}\|  
  \lesssim \|\bs\phi_{h}\|_{L^{3}}\|\div\bs\psi_{h}\|.
\end{align*}
For the second term, we have
\begin{align*}
\|\bs\phi_{h} \times \bs\psi\| & \leq \|\bs\psi\|_{L^{6}} \|\bs\phi_{h}\|_{L^{3}} \lesssim \|\bs\psi\|_{1} \|\bs\phi_{h}\|_{L^{3}} \lesssim \|\div\bs\psi_{h}\|\|\bs\phi_{h}\|_{L^{3}}.
\end{align*}
Then from  Lemma \ref{lem:vhbound} the desired result follows.
\end{proof}

\section{Semi-discrete finite element scheme}
In this section,   we shall give  the semi-discrete scheme of \eqref{eq:FHD},  prove   the existence of semi-discrete   solutions,  and derive   error estimates. 

\subsection{Semi-discretization}

The semi-discrete formulation of the ferrofluid flow model \eqref{eq:FHD} reads as: Find $\bs u_h \in \mathcal C^1(\bs S_h)$, $\tilde p_h\in \mathcal C^0(L_h)$, $\bs m_h \in  \mathcal C^1(\bs V_h)$,  $\bs z_h\in \mathcal C^0(\bs U_h)$, $\bs k_h \in \mathcal C^0(\bs U_h)$, $\bs H_h \in  \mathcal C^{0}(\bs V_h)$, and $\varphi_h \in \mathcal C^0(W_h)$ such that
\begin{equation}\label{eq:FHD-semi}
\left\{
\begin{array}{rll}
 (\partial_t\bs u_h,\bs v_h)  - (\bs u_h\times \curl\bs u_h,\bs v_h) + \eta(\nabla\bs u_h,\nabla\bs v_h) \quad\qquad &&\\
  - (\tilde p_h,\div \bs v_h) 
 -\mu_0b(\bs v_h;\bs H_h,\bs m_h) -\frac{\mu_0}{2}(\bs v_h\times \bs k_h,\bs H_h)  &= 0
&\quad\forall~\bs v_h \in \bs S_h,\\
(\div\bs u_h,q_h)   &= 0 &\quad\forall~q_h\in L_h, \\ \\
 (\partial_t\bs m_h,\bs F_h) + b(\bs u_h;\bs m_h,\bs F_h) + \frac{1}{2}(\bs u_h\times\bs F_h,\bs k_h) 
 \quad\qquad && \\  -\frac{1}{2} (\curl\bs z_h,\bs F_h) 
 + \sigma(\curl\bs k_h,\bs F_h)   + \sigma(\div\bs m_h, \div\bs F_h) \quad\qquad &&\\
 +\frac{1}{\tau}(\bs m_h,\bs F_h) 
 - \frac{\chi_0}{\tau}(\bs H_h,\bs F_h) + \beta(\bs m_h\times (\bs m_h \times H_h),\bs F_h) &= 0
& \quad\forall~\bs F_h\in \bs V_h,\\
 (\bs z_h,\bs\zeta_h)  - (\bs u_h\times \bs m_h,\bs \zeta_h)&=0 &\quad\forall~\bs\zeta_h\in \bs U_h,\\
(\bs k_h,\bs\kappa_h) - (\bs m_h,\curl\bs\kappa_h)&=0 & \quad\forall~\bs\kappa\in \bs U,\\\\
(\bs H_h,\bs G_h) + (\varphi_h, \div \bs G_h) &= 0 &\quad  \forall~ \bs G_h\in \bs V_h,\\
(\div\bs H_h + \div\bs m_h, r_h)+(\div \bs H_{e},r_h)&=0 & \quad\forall~r_h \in W_h,
\end{array}
\right.
\end{equation}
 with the initial data
\begin{equation}\label{ini-semi}
\bs u_h(\cdot,0) = \pi_h\bs u_0(\cdot),\quad \bs m_h(\cdot,0) = I_h^d\bs m_0(\cdot).
\end{equation}

%
%

In the following of this section, we will prove that \eqref{eq:FHD-semi} has solutions. 

First, the choice of finite element spaces $\bs V_h$ and $W_h$ implies the  inf-sup condition
$$
\sup\limits_{\bs\psi_h\in \bs V_h}\frac{(\div\bs \psi_h, r_h)}{\|\bs\psi_h\|_{\div}} \gtrsim \|r_h\|, \quad\forall~r_h \in W_h.
$$
Thus, for a given $\bs m_h \in \bs V_h$, there exists a unique pair $(\bs H_h(\bs m_h) ,\varphi_h(\bs m_h)) \in \bs V_h \times W_h$ 
satisfying the last two equations of \eqref{eq:FHD-semi}, with 
$$
\bs H_h(\bs m_h) = \grad_h \varphi_h(\bs m_h) \in \mathfrak K_h^d \subset \bs V_h
$$
and
$$
\|\bs H_h(\bs m_h)\|_{\div} + \|\varphi_h(\bs m_h)\| \lesssim \|\bs m_h\|_{\div} + \|\bs H_e\|_{\div}.
$$

Second, we define 
$$
\mathfrak K_h^s := \{\bs v_h \in \bs S_h:\ (\div\bs v_h, q_h) = 0,\ \ \forall~~q_h \in L_h\}
$$
and introduce an auxiliary problem:
 Find $\bs u_h \in \mathcal C^1(\mathfrak K_h^s)$ and $\bs m_h \in \mathcal C^1(\bs V_h)$ such that 
\begin{equation}
\label{eq:aux-I}
\left\{
\begin{array}{rll}
(\partial_t\bs u_h,\bs v_h) - (\bs u_h \times \curl\bs u_h,\bs v_h) + \eta(\nabla\bs u_h,\nabla\bs v_h)\hskip2cm & & \\
-\mu_0b(\bs v_h;\bs H_h(\bs m_h),\bs m_h) 
 -\frac{\mu_0}{2}(\bs v_h \times \curl_h\bs m_h,\bs H_h(\bs m_h)) & = 0 
&\forall~\bs v_h\in \mathfrak K_h^s, \\\\
(\partial_t\bs m_h,\bs F_h) + b(\bs u_h;\bs m_h,\bs F_h) + \frac{1}{2}(\bs u_h \times \bs F_h,\curl_h\bs m_h) & & \\
+\frac{1}{\tau}(\bs m_h,\bs F_h) - \frac{1}{2}(\bs u_h\times\bs m_h,\curl_h\bs F_h) & & \\
+ \sigma(\curl_h\bs m_h,\curl_h\bs F_h) + \sigma(\div\bs m_h,\div\bs F_h) & & \\
-\frac{\chi_0}{\tau}(\bs H_h(\bs m_h),\bs F_h)  
+ \beta (\bs m_h \times (\bs m_h \times \bs H_h(\bs m_h)),\bs F_h) & = 0 
& \forall~\bs F_h \in \bs V_h,
\end{array}
\right.
\end{equation}
with the initial data \eqref{ini-semi}.

\begin{lemma}\label{the:solu-semi}
Given $\bs H_e \in H^1(H(\div))$,  the auxiliary problem \eqref{eq:aux-I} has at least one solution $(\bs u_h,\bs m_h)$ satisfying $\bs u_h \in L^\infty(L^2(\Omega))\cap L^2(H^1(\Omega) \cap \mathfrak K_h^s)$ and $\bs m_h \in L^\infty(L^2(\Omega))\cap L^2(\bs V_h)$.
\end{lemma}
\begin{proof}
Let $\{\bs\Phi_i\}_{i = 1}^{s}$ and $\{\bs\Psi_i\}_{i = s+1}^{s+v}$ be  the sets of bases respectively for $\mathfrak K_h^s$ and $\bs V_h$,  satisfying
$$
(\bs\Phi_i,\bs\Phi_j) = \delta_{ij},\ i,j = 1,2,\cdots,s\ 
$$
and 
$$
(\bs\Psi_i,\bs\Psi_j) = \delta_{ij},\ i,j = s+1,s+2,\cdots,s+v.
$$
Let
$$
\bs u_h = \sum\limits_{i = 1}^s x_i(t)\bs\Phi_i,\ \bs m_h = \sum\limits_{i = s+1}^{s+v}x_i(t)\bs\Psi_i,\  \bs u_h(\cdot,0) = \sum\limits_{i = 1}^s x_i^0\bs\Phi_i,\  \bs m_h(\cdot,0) = \sum\limits_{i = s+1}^{s+v} x_i^0\bs\Psi_i.
$$
Then the auxiliary problem \eqref{eq:aux-I} can be written as: Find $x_i(t)$ ($i = 1,2,\cdots,s+v$) such that
{\small \begin{equation}\label{eq:semi-solu1}
\begin{array}{ll}
x_i^\prime(t) - \sum\limits_{j,d = 1}^s(\bs\Phi_j\times \curl \bs\Phi_d,\Phi_i) x_j(t)x_d(t) + \eta\sum\limits_{j = 1}^s(\nabla\bs\Phi_j,\nabla\bs\Phi_i)x_j(t)&\\
-\mu_0\sum\limits_{j,d = s+1}^{s+v} b(\bs\Phi_i;H_h(\bs \Psi_d),\bs\Psi_j) x_j(t)x_d(t)  
- \frac{\mu_0}{2}\sum\limits_{j,d = s+1}^{s+v}(\bs\Phi_i\times \curl_h\bs\Psi_j,\bs H_h(\bs\Psi_d))x_j(t)x_d(t) &= 0
\end{array}
\end{equation}
for $ i= 1,2,\cdots, s$, and 
\begin{equation}\label{eq:semi-solu2}
\begin{array}{rl}
x_i^{\prime}(t) + \sum\limits_{j = 1}^s\sum\limits_{d =s+ 1}^{s+v}b(\bs\Phi_j;\bs\Psi_d,\bs\Psi_i)x_j(t)x_d(t) + \frac{1}{2}\sum\limits_{j = 1}^s \sum\limits_{d = s+1}^{s+v}(\bs\Phi_j\times\bs\Psi_i,\curl_h\bs \Psi_d)x_j(t)x_d(t)&\\
+ \frac{1}{\tau} x_i(t) - \frac{1}{2}\sum\limits_{j = 1}^s\sum\limits_{d = s+1}^{s+v}(\bs\Phi_j\times\bs\Psi_d,\curl_h\bs\Psi_i) x_j(t)x_d(t)+ \sigma\sum\limits_{j = s+1}^{s+v}(\curl_h\bs\Psi_j,\curl_h\bs\Psi_i)x_j(t) &\\
+\sigma\sum\limits_{j = s+1}^{s+v}(\div\bs\Psi_j,\div\bs\Psi_i)x_j(t) - \frac{\chi_0}{\tau}\sum\limits_{j = s+1}^{s+v}(\bs H_h(\bs\Psi_j),\bs\Psi_i)x_j(t)&\\
+ \beta\sum\limits_{j,d,a = s+1}^{s+v}(\bs\Psi_j\times(\bs\Psi_d \times \bs H_h(\bs\Psi_a),\bs\Psi_i)x_j(t)x_d(t)x_a(t) &= 0
\end{array}
\end{equation}
}
for $i = s+1,\cdots, s+v$, with the initial conditions
$$
x_i(0) = x_i^0\qquad i = 1,2,\cdots,s+v.
$$
By Carath\'eodory's theorem \cite[Page 1044]{1985Nonlinear}, the system \eqref{eq:semi-solu1} - \eqref{eq:semi-solu2} admits a local maximal solution on $[0,t]$ for some $t \leq T$. 

Taking $\bs v_h = \bs u_h \in \mathfrak K_h^s$ and $\bs F_h = \bs H_h(\bs m_h)\in \mathfrak K_h^d\subset \bs V_h$ in \eqref{eq:aux-I}, we have
\begin{align*}
 \frac{1}{2}\frac{\dd}{\dd t}\|\bs u_h\|^2 + \eta\|\nabla\bs u_h\|^2 - \mu_0b(\bs u_h;\bs H_h,\bs m_h) - \frac{\mu_0}{2}(\bs u_h \times \curl_h\bs m_h,\bs H_h) = 0 
\end{align*}
and
\begin{eqnarray*}
(\partial_t\bs m_h,\bs H_h) + b(\bs u_h;\bs m_h,\bs H_h) + \frac{1}{2}(\bs u_h \times \bs H_h,\curl_h\bs m_h) + \frac{1}{\tau}(\bs m_h,\bs H_h)&& \\
+ \sigma(\div\bs m_h,\div\bs H_h) - \frac{\chi_0}{\tau}\|\bs H_h\|^2 - \beta\|\bs m_h \times \bs H_h\|^2 &=& 0.
\end{eqnarray*}
Multiplying $-\mu_0$ to the second equation and adding the resultant equation with the first equation, we get
\begin{align*}
&\frac{1}{2}\frac{\dd}{\dd t}\|\bs u_h\|^2 + \eta\|\nabla\bs u_h\|^2 + \frac{\chi_0\mu_0}{\tau}\|\bs H_h\|^2 + \beta\mu_0\|\bs m_h \times \bs H_h\|^2 \\
=& \mu_0(\partial_t\bs m_h,\bs H_h) 
+ \frac{\mu_0}{\tau}(\bs m_h,\bs H_h) + \sigma\mu_0(\div\bs m_h,\div\bs H_h).
\end{align*}
Note that $\bs H_h = \grad_h\varphi_h$ and
\begin{equation}\label{eq:Hrm}
\div\bs H_h = -\div\bs m_h - Q_h\div\bs H_e.
\end{equation}
Testing the above equation with $\varphi_h$ and applying integration by part, we obtain
\begin{equation}\label{Hrm1}
\|\bs H_h\|^2 = -(\bs m_h,\bs H_h) - (\bs H_e,\bs H_h).
\end{equation} 
Differentiating \eqref{eq:Hrm} with respect to $t$, testing the resultant equation with $\varphi_h$, and using integration by part, we have
$$
(\partial_t\bs m_h,\bs H_h) = -\frac{1}{2}\frac{\dd}{\dd t}\|\bs H\|^2 - (\partial_t\bs H_e,\bs H_h).
$$
Thus, there holds
\begin{align*}
&\frac{1}{2}\frac{\dd}{\dd t}(\|\bs u_h\|^2 +\mu_0\|\bs H_h\|^2) + \eta\|\nabla\bs u_h\|^2 + \frac{(\chi_0+1)\mu_0}{\tau}\|\bs H_h\|^2 \\
&\ \ + \beta\mu_0\|\bs m_h \times \bs H_h\|^2   + \sigma\mu_0\|\div\bs m_h\|^2\\
=& -\mu_0(\partial_t\bs H_e,\bs H_h) 
- \frac{\mu_0}{\tau}(\bs H_e,\bs H_h) - \sigma\mu_0(\div\bs H_e,\div\bs H_h).
\end{align*}
Taking $\bs F_h = \bs m_h \in \bs V_h$ in \eqref{eq:aux-I}, and using \eqref{Hrm1}, we get
\begin{align*}
\frac{1}{2}\frac{\dd}{\dd t}\|\bs m_h\|^2 + \frac{1}{\tau}\|\bs m_h\|^2 +\sigma\|\curl_h\bs m_h\|^2 + \sigma\|\div\bs m_h\|^2 + \frac{\chi_0}{\tau}\|\bs H_h\|^2 = -\frac{\chi_0}{\tau}(\bs H_e,\bs H_h).
\end{align*}
 Adding the above two equations together gives
 \begin{align*}
&\frac{1}{2}\frac{\dd}{\dd t}(\|\bs u_h\|^2 +\|\bs m_h\|^2+\mu_0\|\bs H_h\|^2) + \eta\|\nabla\bs u_h\|^2 + \frac{(\chi_0+1)\mu_0+\chi_0}{\tau}\|\bs H_h\|^2 \\
&\ \ + \beta\mu_0\|\bs m_h \times \bs H_h\|^2 +\frac{1}{\tau}\|\bs m_h\|^2 + \sigma\|\curl_h\bs m_h\|^2  + \sigma(1+\mu_0)\|\div\bs m_h\|^2\\
=& -\mu_0(\partial_t\bs H_e,\bs H_h) 
- \frac{\mu_0+\chi_0}{\tau}(\bs H_e,\bs H_h) - \sigma\mu_0(\div\bs H_e,\div\bs H_h).
\end{align*}
Integrating the above equation on the interval $(0,t)$ for any $t \in (0,T]$, and using the H\"older inequality, we obtain
\begin{align*}
&\|\bs u_h\|^2 + \|\bs m_h\|^2 +\mu_0\|\bs H_h\|^2 + 2\eta\int_0^t\|\nabla\bs u_h\|^2\dd s \\
&\ \ +\frac{2(\chi_0+1)\mu_0 + 2\chi_0}{\tau}\int_0^t\|\bs H_h\|^2\dd s 
+ 2\beta\mu_0\int_0^t\|\bs m_h\times \bs H_h\|^2\dd s + \frac{2}{\tau}\int_0^t\|\bs m_h\|^2\dd s\\
&\ \ + 2\sigma\int_0^t\|\curl_h\bs m_h\|^2\dd s 
 + 2\sigma(1+\mu_0)\int_0^t\|\div\bs m_h\|^2\dd s \\
 \leq &\|\bs u_h(\cdot,0)\|^2 + \|\bs m_h(\cdot,0)\|^2 + \|\bs H_h(\cdot,0)\|^2  +2\mu_0\int_0^t\|\partial_t\bs H_e\|\|\bs H_h\|\dd s \\
 &\ \ + 2\frac{\mu_0+\chi_0}{\tau}\int_0^t\|\bs H_e\|\|\bs H_h\| \dd s
  + \sigma\mu_0\int_0^t\|\div\bs H_e\|\|\div\bs H\|\dd s,
\end{align*}
which implies that $\|\bs u_h\|_{L^\infty(L^2)}$, $\|\bs m_h\|_{L^\infty(L^2)}\|$, $\|\bs u_h\|_{L^2(H^1)}\|$, and $\|\bs m_h\|_{L^2(H(\div))}$ are bounded. Therefore, $(\bs u_h,\bs m_h)$ is the global solution of  the auxiliary problem \eqref{eq:aux-I} on $[0,T]$.
\end{proof}

Furthermore, we need the following assumption to show   the auxiliary problem \eqref{eq:aux-I} admits a unique solution.
\begin{assum} \label{Assum-I}
There exist constants $h_0>0$ and $M_0 >0$, such that for any $0 < h \leq h_0$ and any solution $(\bs u_h,\bs m_h)$ of \eqref{eq:aux-I}, there holds 
\begin{equation}\label{assump1}
\max\{\|\bs u_h\|_{L^\infty},\ \|\bs m_h\|_{L^\infty},\|\curl_h\bs m_h\|_{L^\infty},\ \|\div\bs m_h\|_{L^\infty}\} \leq M_0. 
\end{equation}
\end{assum}
\begin{lemma}\label{lem:unique-semi}
Under Assumption \ref{Assum-I}, the auxiliary problem \eqref{eq:aux-I} has a unique solution.
\end{lemma}
\begin{proof}
Assume that $(\bs u_h,\bs m_h)$ and $(\bs u_h^*,\bs m_h^*)$ be any two solution of   \eqref{eq:aux-I}.  We denote $d_u := \bs u_h - \bs u_h^*$ and $d_m: = \bs m_h - \bs m_h^*$, and   have
\begin{equation}\label{eq:unique-semi-1}
\begin{split}
(\partial_td_u,\bs v_h) + d_1(\bs v_h) + \eta(\nabla d_u,\nabla\bs v_h) - d_2(\bs v_h) - d_3(\bs v_h) = 0
\end{split} \qquad \forall~~\bs v_h \in \bs S_h
\end{equation}
and
\begin{equation}\label{eq:unique-semi-2}
\begin{split}
(\partial_t d_m,\bs F_h) + d_4(\bs F_h) + d_5(\bs F_h) + \frac{1}{\tau}(d_m,\bs F_h) - d_6(\bs F_h)  \\
+\sigma(\curl_hd_m,\curl_h\bs F_h) + \sigma(\div d_m,\div\bs F_h) \\
- \frac{\chi_0}{\tau}(\bs H_h(d_m),\bs F_h) + d_7(\bs F_h)
\end{split}\quad\quad\qquad\forall~~\bs F_h\in \bs V_h,
\end{equation}
with
\begin{align*}
d_1(\bs v_h) & = (\bs u_h\times\curl\bs u_h,\bs v_h) - (\bs u_h^*\times\curl\bs u_h^*,\bs v_h),\\
d_2(\bs v_h) & = \mu_0b(\bs v_h;\bs H_h(\bs m_h),\bs m_h) - \mu_0 b(\bs v_h;\bs H_h(\bs m_h^*),\bs m_h^*),\\
d_3(\bs v_h) & = \frac{\mu_0}{2}(\bs v_h \times \curl_h\bs m_h,\bs H_h(\bs m_h)) - \frac{\mu_0}{2}(\bs v_h\times \curl_h\bs m_h^*,\bs H_h(\bs m_h^*)),\\
d_4(\bs F_h) & = b(\bs u_h;\bs m_h,\bs F_h) - b(\bs u_h^*;\bs m_h^*,\bs F_h),\\
d_5(\bs F_h) & = \frac{1}{2}(\bs u_h\times \bs F_h,\curl_h\bs m_h) - \frac{1}{2}(\bs u_h^* \times \bs F_h,\curl_h\bs m_h^*),\\
d_6(\bs F_h) & = \frac{1}{2}(\bs u_h \times \bs m_h,\curl_h\bs F_h) - \frac{1}{2}(\bs u_h^* \times \bs m_h^*,\curl_h\bs F_h),\\
d_7(\bs F_h) & = \beta(\bs m_h\times(\bs m_h \times \bs H_h(\bs m_h)),\bs F_h) - \beta(\bs m_h^*\times(\bs m_h^* \times \bs H_h(\bs m_h^*)),\bs F_h).
\end{align*}
Taking $\bs v_h = d_u$ and $\bs F_h = d_m$ in \eqref{eq:unique-semi-1} and \eqref{eq:unique-semi-2}, respectively, and adding the resultant equations together, we obtain
\begin{align*}
& \frac{1}{2}\frac{\dd}{\dd t}(\|d_u\|^2 +\|d_m\|^2) + \eta\|\nabla d_u\|^2 + \frac{1}{\tau}\|d_m\|^2 + \sigma(\|\curl_h d_m\|^2 + \|\div d_m\|^2) \\
= & \frac{\chi_0}{\tau}(\bs H_h(d_m),d_m) -d_1(d_u) + d_2(d_u) + d_3(d_u)-d_4(d_m) - d_5(d_m) + d_6(d_m) - d_7(d_m).
\end{align*}
Taking $\bs F_h = d_H = \bs H_h(d_m)$ in \eqref{eq:unique-semi-2}, we get
\begin{align*}
(\partial_t d_m,d_H) + d_4(d_H) + d_5(d_H) + \frac{1}{\tau}(d_m,d_H)
 +\sigma(\div d_m,\div d_H) - \frac{\chi_0}{\tau}\|d_H\|^2 + d_7(d_H)& = 0.
\end{align*}
Since $\bs H_h$ satisfying the sixth and seventh equations of \eqref{eq:FHD-semi},
 we can show that $d_H$ satisfying the following linear saddle point problem: Find $d_H\in \bs V_h$ and $d_\varphi\in W_h$ such that
\begin{equation}\label{eq:dH-saddle}
\left\{
\begin{array}{ll}
(d_H,\bs\psi_h) + (d_\varphi,\div\bs\psi_h) = 0 & \forall~~\bs\psi_h\in \bs V_h,\\
(\div d_H,r_h) = - (\div d_m, r_h) & \forall~~r_h \in W_h.
\end{array}
\right.
\end{equation}
Since $\div  \bs V_{h} \subset W_{h}$, this system implies that
$$
d_H = \grad_h d_\varphi\quad\text{and}\quad \div d_H = -\div d_m.
$$
So
$$
(\div d_m,\div d_H) = -\|\div d_H\|^2,\quad \|d_H\|^2 = -(d_m,d_H),\quad\frac{1}{2}\frac{\dd}{\dd t}\| d_H\|^2 = - (\partial_t d_m, d_H).
$$
Therefore, we have
\begin{align*}
& \frac{1}{2}\frac{\dd}{\dd t}(\|d_u\|^2 +\|d_m\|^2 +\mu_0\|d_H\|^2) + \eta\|\nabla d_u\|^2 + \frac{1}{\tau}\|d_m\|^2 + \sigma(\|\curl_h d_m\|^2 + \|\div d_m\|^2) \\
&\quad  + \frac{\chi_0(1+\mu_0) +\mu_0}{\tau}\|d_H\|^2 + \sigma\mu_0\|\div d_H\|^2 \\
=& -d_1(d_u) + \mu_0d_4(d_H) + d_2(d_u) + \mu_0d_5(d_H) + d_3(d_u)-d_4(d_m) - d_5(d_m)\\
& \quad + d_6(d_m) - d_7(d_m) + \mu_0 d_7(d_H).
\end{align*}
Denote 
\begin{align*}
& \mathcal D_1: = \|d_u\|^2 +\|d_m\|^2 +\mu_0\|d_H\|^2,
\\
& \|\bs F_h\|_{1,h}^2: = \|\div\bs F_h\|^2 + \|\curl_h\bs F_h\|^2\qquad\forall~~\bs F_h \in \bs V_h,
\\
& \mathcal D_2: = \eta\|\nabla d_u\|^2 + \frac{1}{\tau}\|d_m\|^2 + \sigma\|d_m\|_{1,h}^2 + \frac{\chi_0(1+\mu_0)+\mu_0}{\tau}\|d_H\|^2\\
&\qquad + \sigma\mu_0\|\div d_H\|^2 + \beta\mu_0\|\bs m_h^* \times d_H\|^2.
\end{align*}
Using the fact that
\begin{align*}
&\mu_0d_4(d_H) + d_2(d_u) = \mu_0b(\bs u_h;d_m,d_H) + \mu_0b(d_u;\bs H_h,d_m),\\
& \mu_0 d_5(d_H) + d_3(d_u) = \frac{\mu_0}{2}\left((\bs u_h\times d_H,\curl_h d_m) + (d_u\times \curl_h d_m,\bs H_h) \right), \\
&d_6(d_m) - d_5(d_m) = \frac{1}{2}\left( (d_u\times \bs m_h,\curl_h d_m) - (d_u \times d_m,\curl_h\bs m_h)\right),\\
& d_7(d_m) = \beta(d_m\times\bs m_h, d_m\times \bs H_h) - \beta(d_m \times\bs m_h, d_m\times d_H) + \beta(d_m\times\bs m_h,\bs m_h \times d_H), \\
& d_7(d_H) = -\beta\|\bs m_h^* \times d_H\|^2    + \beta (d_H \times \bs m_h, d_m\times \bs H_h) + \beta(d_m\times d_H, d_m \times \bs H_h) 
\\
&\qquad\qquad + \beta (d_H \times d_m,\bs m_h\times \bs H_h),
\end{align*} 
we obtain
\begin{align*}
 \frac{1}{2}\frac{\dd}{\dd t}\mathcal D_1 + \mathcal D_2 \leq  & \|\bs u_h\|_{L^\infty}\|\nabla d_u\|\|d_u\| + \frac{\mu_0}{2}\|\bs u_h\|_{L^\infty}(\|\div d_m\|\|d_H\| + \|d_m\|\|\div d_H\|) \\
 & + \frac{\mu_0}{2}(\|\div\bs H_h\|_{L^\infty}\| d_u\|\|\div d_m\|  + \|\bs H_h\|_{L^\infty}\|d_u\|\|\div d_m\|)  \\
& + \frac{\mu_0}{2}(\|\bs u_h\|_{L^\infty}\|d_H\|\|\curl_h d_m\| + \|\bs H_h\|_{L^\infty}\|d_u\|\|\curl_h d_m\|)\\
& +\frac{1}{2} \|\bs m_h\|_{L^\infty}\|d_u\|\|\div d_m\| + \frac{1}{2}\|\div\bs m_h\|_{L^\infty}\|d_u\|\|d_m\| \\
&  + \frac{1}{2}\|\bs m_h\|_{L^\infty}\| d_u\| \|\curl_h d_m\| +\frac{1}{2} \|\curl_h\bs m_h\|_{L^\infty}\|d_u\|\|d_m\|\\
& + \beta\|\bs H_h\|_{L^\infty}\|\bs m_h\|_{L^\infty}\|d_m\|^2 +\beta \|\bs m_h\|_{L^\infty}\|d_m\| \|d_m\|_{1,h}\|\div d_H\| \\
& + \beta \|\bs m_h\|_{L^\infty}^2\|d_H\|\|d_m\|  + \beta\mu_0\|\bs m_h\|_{L^\infty}\|\bs H_h\|_{L^\infty} \|\| d_H\| \|d_m\|\\
& +\beta\mu_0 \|\bs H_h\|_{L^\infty}\|\div d_H\|\|d_m\|_{1,h}\|d_m\| +\beta\mu_0 \|\bs H_h\|_{L^\infty}\|\bs m_h\|_{L^\infty}\|d_H\|\|d_m\|\\
\leq  & M_0\|d_u\| (\|\nabla d_u\| + (\mu_0 + 2)\| d_m\|_{1,h}) +  M_0\frac{\mu_0}{2} \|d_m\|  \|\div d_H\| \\
& + \mu_0M_0\|d_H\|\| d_m\|_{1,h} + \beta M_0^2\|d_m\|\|d_m\|_{1,h} + \beta(1+2\mu_0) M_0^2\|d_m\|\|d_H\| \\
& + \beta M_0(1+\mu_0)\|d_m\|\|d_m\|_{1,h}\|\div d_H\|
\\
\leq & (\frac{M_0^2}{2\epsilon_1}+\frac{M_0^2(\mu_0+1)^2}{2\epsilon_2} )\|d_u\|^2  + (\frac{\mu_0^2 M_0^2}{2\epsilon_4} + \frac{\beta^2(1+2\mu_0)^2M_0^4}{2\epsilon_6})\|d_H\|^2 \\
& + (\frac{M_0^2\mu_0^2}{8\epsilon_3} + \frac{\beta^2M_0^4}{2\epsilon_5} + \frac{2\beta^2 M_0^4(1+\mu_0)^2}{\epsilon_7} )\|d_m\|^2  \\
& + \frac{\epsilon_1}{2}\|\nabla d_u\|^2 + \frac{\epsilon_2 + \epsilon_4+\epsilon_5}{2}\|d_m\|_{1,h}^2 + \frac{\epsilon_3+\epsilon_7}{2}\|\div d_H\|^2  + \frac{\epsilon_6}{2}\|d_m\|^2.
\end{align*}
Take $\epsilon_i > 0$ ($i = 1,2,\cdots,7$) with
$$
\epsilon_1 \leq \eta,\ \ \epsilon_6 \leq \frac{2}{\tau},\ \ \epsilon_2+\epsilon_4+\epsilon_5 \leq 2\sigma,\ \ \epsilon_3 + \epsilon_7 \leq 2\sigma\mu_0,
$$ 
then by  the assumption \eqref{assump1} we get
$$
\frac{\dd \mathcal D_1}{\dd t} \leq C \mathcal D_1,
$$
which, together with the Gr\"onwall's inequality, implies for any $t \in (0,T]$
$$
\mathcal D_1(t) \leq \mathcal D_1(0)e^{Ct}.
$$
Note that $\mathcal D_1(0) = 0$, therefore, we have
$$
\bs u_h = \bs u_h^*,\quad \bs m_h = \bs m_h^*.
$$
The desired result follows.
\end{proof}

We are ready to show the existence and uniqueness of the solution for the semi-discrete scheme \eqref{eq:FHD-semi}.
\begin{theorem}
\label{the:ex-un}
Given $\bs H_e \in H^1(H(\div))$, the semi-discrete scheme \eqref{eq:FHD-semi} has at least one solution. Furthermore, under  Assumption I \eqref{eq:FHD-semi} admits a unique solution. 
\end{theorem}
\begin{proof}
In view of Lemmas \ref{the:solu-semi} and \ref{lem:unique-semi}, we only need to prove that \eqref{eq:FHD-semi} and the auxiliary problem \eqref{eq:aux-I} are equivalent. 

 Assume that $\bs u_h \in \mathfrak K_h^s$ and $\bs m_h \in \bs V_h$ solve  \eqref{eq:aux-I}. Then 
there exists a unique $(\bs H_h, \varphi_h)\in \bs V_h \times W_h$ satisfying the sixth and seventh equations of \eqref{eq:FHD-semi}. Let $\bs z_h = Q_h^c(\bs u_h \times \bs m_h)$ and $\bs k_h = \curl_h \bs m_h$, it is easy to check that $(\bs z_h,\bs k_h)$ is the unique solution of the fourth and fifth equations of \eqref{eq:FHD-semi}. The definition of $\mathfrak K_h^s$ implies $\bs u_h$ satisfying the second equation of \eqref{eq:FHD-semi}. The inf-sup condition \eqref{eq:inf-sup} implies that there exists a unique $\tilde p_h\in L_h$ such that $\tilde p_h$, $\bs u_h$, $\bs m_h$, $\bs H_h$, and $\bs k_h$ satisfy the first equation of \eqref{eq:FHD-semi}. From the second equation of  \eqref{eq:aux-I} it is easy to see that the third equation of \eqref{eq:FHD-semi} holds.

It is obvious that the solution of \eqref{eq:FHD-semi} solves the auxiliary problem \eqref{eq:aux-I}.
\end{proof}

Now we turn to the energy estimation of the semi-discrete scheme \eqref{eq:FHD-semi}.
Define the energy $\mathcal E_h(t)$ of the semi-discrete scheme \eqref{eq:FHD-semi} at time $t\in [0,T]$ as
\begin{align*}
\mathcal E_h(t) & = \|\bs u_h(\cdot,t)\|^2 + \|\bs m_h(\cdot,t)\|^2 + \mu_0\|\bs H_h(\cdot,t)\|^2.
\end{align*}
Since the semi-discrete scheme \eqref{eq:FHD-semi} inherits the structure of the weak formulation \eqref{eq:FHD-weak}, by a similar  proof as that of  Theorem \ref{the:energy-con}  we obtain the following energy estimate.
\begin{theorem}
\label{the:energy-est-con-semi-dis}
Given $\bs H_{e} \in \bs H^{1}( H(\div))$,  let  $\bs u_{h} \in \mathcal C^{1}(\bs S_{h})$, $\bs m_{h} \in  \mathcal C^1(\bs V_{h})$, $\bs H_{h} \in  \mathcal C^0(\bs V_{h})$, $\tilde p_{h}\in \mathcal C^0(L_{h})$, $\bs z_{h}\in \mathcal C^0(\bs U_{h})$, and $\bs k_{h} \in \mathcal C^0(\bs U_{h})$ solve the   the semi-discrete scheme \eqref{eq:FHD-semi}. Then the  energy inequality
$$
\mathcal E_h(t) + C_1\int_0^t \mathcal F_h(s)\dd s \leq \mathcal E_h(0) + C_2\int_0^t \left(\|\bs H_{e}(\cdot,s)\|_{H(\div)}^2 + \|\partial_t \bs H_{e}(\cdot,s)\|^2\right) \dd s
$$
holds for all $t \in [0,T]$, where $C_1$ and $C_2$ are positive constants depending only on the data  $\eta,\ \mu_0,\, \chi_0,\,\tau,\,\beta,\,\sigma$ and $\Omega$, and the dissipated energy $\mathcal F_h(t)$ is given by
\begin{align*}
\mathcal F_h(t) &: = 2 \left(\eta\|\nabla \bs u_h\|^2 + \sigma(1+\mu_0)\|\div\bs m_h\|^2 + \sigma\|\bs k_h\|^2 + \frac{1}{\tau}\|\bs m_h\|^2\right.\\
& \qquad\quad  +  \left. \frac{2}{\tau}\left[\mu_0(1+\chi_0)+\chi_0 \right]\|\bs H_h\|^2 + \mu_0\beta\|\bs m_h \times \bs H_h\|^2\right).
\end{align*}
\end{theorem}

\subsection{Error analysis}  
We first make the following regularity assumptions for the   solution  $(\bs u$, $\tilde p$, $\bs m$, $\bs z$, $\bs k$, $\bs H$, $\varphi )$ of the weak problem \eqref{eq:FHD-weak}: 

\begin{assum}\label{assum-II} 
 Assume that $\bs u \in \mathcal C^1(\bs S)$, $\tilde p\in \mathcal C^0(W)$, $\bs m \in  \mathcal C^1(\bs V)$,   $\bs z\in \mathcal C^0(\bs U)$, $\bs k \in \mathcal C^0(\bs U)$, $\bs H \in  \mathcal C^0(\bs V)$, and $\varphi \in C^0(W)$ solve \eqref{eq:FHD-weak} and satisfy  
\begin{equation*}
\left\{
\begin{array}{lll}
\bs u \in   H^{l+2}(\Omega)\cap W^{1,\infty}(\Omega), & \partial_t\bs u \in \bs S\cap  H^{l+2}(\Omega), & \tilde p \in   H^{l+1}(\Omega),\\
   \bs m,\ \bs H \in    H^{l+1}(\Omega)\cap W^{1,\infty}(\Omega), & \partial_t\bs m \in \bs V \cap  H^{l+1}(\Omega), & \div\bs m,\ \div\bs H \in H^{l+1}(\Omega),\\
  \bs k\in    H^{l+1}(\Omega)\cap L^\infty(\Omega), & \bs z \in   H^{l+1}(\Omega), & \curl\bs k,\ \curl\bs z\in H^{l+1}(\Omega).
  \end{array}
  \right.
  \end{equation*}
\end{assum}

To derive  error estimates for the semi-discrete formulation \eqref{eq:FHD-semi}, we 
introduce some notations:
\begin{align*}
&\xi_u := \pi_h \bs u - \bs u_h,\qquad \xi_m :=  I_h^d\bs m - \bs m_h,\qquad \xi_z :=  I_h^c\bs z - \bs z_h, \\
&\xi_k := I_h^c \bs k - \bs k_h,\quad\xi_H :=  I_h^d \bs H - \bs H_{h},\quad \xi_p := Q_h\tilde p - \tilde p_h,\\
&\theta_u := \pi_h\bs u - \bs u,\quad \theta_m :=  I_h^d\bs m - \bs m,\quad\,\, \theta_z :=  I_h^c\bs z - \bs z, \\ 
& \theta_k := I_h^c\bs k - \bs k,\quad \theta_H: =  I_h^d \bs H - \bs H,\quad \theta_p :=  Q_h\tilde p - \tilde p.
\end{align*}
Here we recall that   $\pi_{h}:\ H_{0}^{1}(\Omega)\rightarrow \Sigma_{h}$ is  the  standard  Lagrange interpolation operator, $Q_{h}:\ W\rightarrow W_{h}$ is the  $L^{2}$ orthogonal projection operator, and $I_{h}^{c}:\bs U \rightarrow \bs U_{h}$ and $I_{h}^{d}:\bs V \rightarrow \bs V_{h}$ are the quasi-interpolation operators  defined in \eqref{eq:inter-1}-\eqref{eq:inter-2}.  Thus, under Assumption \eqref{assum-II} the following interpolation error estimates hold: 
\begin{equation}\label{regularity}
\left\{\begin{array}{l}
\|\theta_u\| +h \|\nabla \theta_u\| \lesssim h^{l+2}\|\bs u\|_{l+2}, \quad  \|\theta_p\|  \lesssim h^{l+1}\|\tilde    p\|_{l+1},\\
\|\theta_m\|+\|\div \theta_m\| \lesssim h^{l+1}(\|\bs m\|_{l+1}+\|\div \bs m\|_{l+1}), \\
\|\theta_H\|+\|\div \theta_H\| \lesssim h^{l+1}(\|\bs H\|_{l+1}+\|\div \bs H\|_{l+1}),\\
\|\theta_z\| + \|\curl  \theta_z\| \lesssim h^{l+1}(\|\bs z\|_{l+1}+\|\curl \bs z\|_{l+1}), \\
 \|\theta_k\|+\|\curl \theta_k\| \lesssim h^{l+1}(\|\bs k\|_{l+1}+\|\curl \bs k\|_{l+1}).
 \end{array}
 \right.
\end{equation}
In view of the above estimates, we note that in the sequel  the hidden constant  factor $C$ in all the estimates may depend on the regularity terms    $\|\bs u\|_{l+2}$, $\|\tilde    p\|_{l+1}$, $\|\bs m\|_{l+1},$ $\|\div \bs m\|_{l+1}$, $\|\bs H\|_{l+1}$, $\|\div \bs H\|_{l+1}$, $\|\bs z\|_{l+1},$ $\|\curl \bs z\|_{l+1},$ $\|\bs k\|_{l+1},$ $\|\curl \bs k\|_{l+1}.$   

  Taking the test functions of the first and third equations of \eqref{eq:FHD-weak} in the finite element spaces $\bs S_h$ and $\bs V_h$, respectively, and subtracting the first and third equations in \eqref{eq:FHD-semi} from the resultant equations, respectively,  we get the following equations of the qualities $\xi$ and $\theta$:
\begin{align}
\label{eq:err-semi-1}
 &(\partial_t\xi_u,\bs v_h)+  \eta(\nabla\xi_u,\nabla\bs v_h) - (\xi_p,\div\bs v_h) \\
 = &f_{1}(\bs v_h) + f_2(\bs v_h)   
 - f_3(\bs v_h)  
 + (\partial_t\theta_u,\bs v) + \eta(\nabla\theta_u,\nabla\bs v) - (\theta_p,\div\bs v_h),
\quad
\forall \bs v_h\in \bs S_h,\nonumber
\end{align}
\begin{align}\label{eq:err-semi-2}
&(\partial_t \xi_m,\bs F_h) - \frac{1}{2}(\curl\xi_z,\bs F_h) + \sigma(\curl\xi_k,\bs F_h) + \frac{1}{\tau}(\xi_m,\bs F_h) \\
&\quad + \sigma(\div\xi_m,\div\bs F_h)
 - \frac{\chi_0}{\tau}(\xi_H,\bs F_h) \nonumber \\
  =& \frac{1}{\tau}(\eta_m,\bs F_h)  
 - \frac{\chi_0}{\tau}(\theta_H,\bs F_h) 
 + (\partial_t\theta_m,\bs F_h) - \frac{1}{2}(\curl\theta_z,\bs F_h) \nonumber \\
&\quad  + \sigma(\curl\theta_k,\bs F_h) 
 - f_4(\bs F_h) - f_5(\bs F_h) - f_6(\bs F_h),
\qquad\forall~~\bs F_h \in \bs V_h,\nonumber
\end{align}
where
\begin{equation*}
\begin{array}{l}
f_{1}(\bs v_h) = \mu_0b(\bs v_h;\bs H,\bs m) - \mu_0b(\bs v_h;\bs H_h,\bs m_h),
\\
f_2(\bs v_h) = \frac{\mu_0}{2}(\bs v_h\times \bs k,\bs H) - \frac{\mu_0}{2}(\bs v_h\times \bs k_h,\bs H_h),
\\
f_3(\bs v_h) = (\bs u \times \curl\bs u - \bs u_h\times\curl\bs u_h,\bs v_h),
\\
f_4(\bs F_h) = b(\bs u;\bs m,\bs F_h) - b(\bs u_h;\bs m_h,\bs F_h),
\\
f_5(\bs F_h) = \frac{1}{2}(\bs u\times\bs F_h,\bs k) - \frac{1}{2}(\bs u_h \times\bs F_h,\bs k_h),
\\
f_6(\bs F_h) = \beta(\bs m\times(\bs m \times\bs H) - \bs m_h\times(\bs m_h \times \bs H_h),\bs F_h).
\end{array}
\end{equation*}

Define
\begin{equation}\label{J1-J4}
\left\{
\begin{array}{ll}
J_{1} &: = f_{1}(\xi_{u}) + f_{2}(\xi_{u}) -f_3(\xi_u) + \mu_{0}f_{4}(\xi_{H}) + \mu_{0}f_{5}(\xi_{H}) + \mu_{0}f_{6}(\xi_{H}),\\
J_{2} & := -\frac{\mu_{0}}{\tau}(\theta_{m},\xi_{H}) + \frac{\mu_{0}\chi_{0}}{\tau}(\theta_{H},\xi_{H}) - \mu_{0}(\partial_{t}\theta_{m},\xi_{H}) + \frac{1}{2}\mu_{0}(\curl\theta_{z},\xi_{H}) \\
& \quad- \sigma\mu_{0}(\curl \theta_{k},\xi_{H}) 
+ (\partial_{t}\theta_{u},\xi_{u}) + \eta(\nabla\theta_{u},\nabla\xi_{u})  + (\tilde p - \tilde p_h,\div\xi_{u})
\\
J_{3} &: = \frac{1}{\tau}(\theta_{m},\xi_{m}) - \frac{\chi_{0}}{\tau}(\theta_{H},\xi_{m}) - (\partial_{t}\theta_{m},\xi_{m}) - \frac{1}{2}(\curl \theta_{z},\xi_{m})\\
&\quad  + \sigma(\curl\theta_{k},\xi_{m}) + \sigma(\theta_{k},\xi_{k}),
\\
J_{4} &: = \frac{1}{2}(\curl\xi_{z},\xi_{m}) - f_{4}(\xi_{m}) - f_{5}(\xi_{m}) - f_{6}(\xi_{m}).
\end{array}
\right.
\end{equation}
We have the following identity of $\xi$ and $\theta$:
\begin{lemma}\label{lem:lem:err-eq-1}
We have
\begin{align*}
& \frac{1}{2}\frac{\dd }{\dd t}\left(\|\xi_{u}\|^{2} + \mu_{0}\|\xi_{H}\|^{2}  + \|\xi_{m}\|^{2} \right) + \eta\|\nabla\xi_{u}\|^{2} + \frac{1}{\tau}(\mu_{0}(1+\chi_{0})+ \chi_{0})\|\xi_{H}\|^{2}  \\
&\quad + \sigma\|\xi_{k}\|^{2} + \frac{1}{\tau}\|\xi_{m}\|^{2} + \sigma\mu_{0}\|\div\xi_H\|^2 + \sigma\|\div\xi_{m}\|^{2}\\
 =& J_{1} + J_{2} + J_{3} + J_{4}.
\end{align*}
\end{lemma}
\begin{proof}
Taking $\bs v_h = \xi_u \in\bs S_h$ in \eqref{eq:err-semi-1} and $\bs F_h = \xi_H \in \bs V_h$ in \eqref{eq:err-semi-2}, we obtain
\begin{equation}\label{eq:err-semi-base-1}
\begin{split}
&\frac{1}{2}\frac{\dd}{\dd t}\|\xi_u\|^2 + \eta\|\nabla\xi_u\|^2 \\
=& f_{1}(\xi_u) + f_2(\xi_u)- f_3(\xi_u) + (\partial_t\theta_u,\xi_u) + \eta(\nabla\theta_u,\nabla\xi_u) + (\tilde p-\tilde p_h,\div\xi_u)
\end{split}
\end{equation}
and
\begin{equation}\label{eq:err-semi-base-2}
\begin{split}
&(\partial_t\xi_m,\xi_H)  + \sigma(\div\xi_m,\div\xi_H) + \frac{1}{\tau}(\xi_m,\xi_H) - \frac{\chi_0}{\tau}\|\xi_H\|^2\\
 =& \frac{1}{\tau}(\theta_m,\xi_H) 
-\frac{\chi_0}{\tau}(\theta_H,\xi_H) + (\partial_t\theta_m,\xi_H) - \frac{1}{2}(\curl\theta_z,\xi_H) + \sigma(\curl\theta_k,\xi_H) 
\\ 
&\quad  - f_4(\xi_H) - f_5(\xi_H) - f_6(\xi_H).
\end{split}
\end{equation}
Multiplying $-\mu_0$ to the two sides of \eqref{eq:err-semi-base-2} and adding the resultant  equation  to \eqref{eq:err-semi-base-1},  we have
\begin{equation}\label{eq:err-semi-base-3}
\begin{split}
\frac{1}{2}\frac{\dd }{\dd t}\|\xi_{u}\|^{2} + \eta\|\nabla\xi_{u}\|^{2}  + \frac{\mu_{0}\chi_{0}}{\tau}\|\xi_{H}\|^{2} 
- \mu_{0}(\partial_{t}\xi_{m},\xi_{H})  \quad & \\
- \sigma\mu_{0}(\div\xi_{m},\div\xi_{H}) - \frac{\mu_{0}}{\tau}(\xi_{m},\xi_{H}) & = J_{1} + J_{2}.
\end{split}
\end{equation}

From the last equations of \eqref{eq:FHD} and   \eqref{eq:FHD-semi}, we get 
\begin{align*}
&\div\bs H + \div\bs m  = -\div H_{e},\\
& \div\bs H_{h} + \div\bs m_{h}  = -Q_{h}\div H_{e}.
\end{align*}
Therefore, we have
\begin{equation}
\label{eq:mid-1}
\div\xi_{H} + \div\xi_{m} = -(I - Q_{h})\div \bs H_{e} + \div\theta_{H} + \div\theta_{m}.
\end{equation}
The sixth equations of \eqref{eq:FHD-weak} and \eqref{eq:FHD-semi} and the first item of Lemma \ref{lem:Ih} imply $\xi_{H} \in \mathfrak K_{h}^d$.
Thus, there exists $\phi_{h} \in W_{h}$ such that $\xi_{H} = \grad_{h}\phi_{h}$. Test \eqref{eq:mid-1} with $\phi_{h}$, use integration by part and   the fact that there exists $\bs\psi_{h} \in \bs V_{h}$ such that $\phi_{h} = \div\bs\psi_{h}$, then we obtain
\begin{equation}\label{eq:err-semi-H}
\|\xi_{H}\|^{2} + (\xi_{m},\xi_{H}) = 0.
\end{equation}
Differentiating \eqref{eq:mid-1} with respect to $t$, testing the resultant equation by $\phi_{h}$ and  using integration by part, we have
\begin{equation}\label{eq:err-semi-MH}
(\partial_{t}\xi_{m},\xi_{H}) = -\frac{1}{2}\frac{\dd }{\dd t}\|\xi_{H}\|^{2}.
\end{equation}

Finally, take $\bs F_h = \xi_{m} \in \bs V_h$ in \eqref{eq:err-semi-2}, use \eqref{eq:err-semi-H} and the fact that
$
(\curl\xi_{k},\xi_{m}) = \|\xi_{k}\|^{2} - (\theta_{k},\xi_{k}) ,
$
then we get
\begin{align}\label{eq:err-semi-base-4}
\frac{1}{2}\frac{\dd }{\dd t}\|\xi_{m}\|^{2} + \sigma\|\xi_{k}\|^{2} + \sigma\|\div\xi_{m}\|^{2} + \frac{1}{\tau}\|\xi_{m}\|^{2} + \frac{\chi_{0}}{\tau}\|\xi_{H}\|^{2} = J_{3} + J_{4}.
\end{align}
Adding \eqref{eq:err-semi-base-3} to \eqref{eq:err-semi-base-4} yields the desired result.
\end{proof}

%

Lemmas \ref{lem:J23}-\ref{lem:J5} are devoted to the estimation of  of $J_i$ $(i = 1,2,3,4)$.  

\begin{lemma}
\label{lem:J23}
Under Assumption \ref{assum-II}, we have 
\begin{align*}
J_{2} & \lesssim h^{l+1}\left( \|\xi_{H}\| + h \|\xi_{u}\| +  \|\nabla\xi_{u}\| +\|\xi_{p}\|\right),  \\
J_{3} & \lesssim h^{l+1} \left(\|\xi_{m}\| + \|\xi_{k}\| \right).
\end{align*}
\end{lemma}
\begin{proof}
Note that
\begin{align*}
 (\tilde p - \tilde p_h,\div\xi_u) & = -(\theta_p,\div\xi_u) + (\xi_p,\div\xi_u). 
\end{align*}
Using the fact that $(\xi_p,\div\bs u_h) = (\xi_p,\div\bs u) = 0$, we have
$$
(\tilde p - \tilde p_h,\div\xi_u)= -(\theta_p,\div\xi_u) + (\xi_p,\div\theta_u).
$$
Then the desired results follow from  the definitions of $J_2$ and $J_3$ in \eqref{J1-J4},  the properties of the classical interpolation operator $\pi_{h}$, Lemma \ref{lem:Ih} and the Cauchy-Schwarz inequality.
\end{proof}

To estimate $J_1$ and $J_4$, we need some inequalities about $\xi_H$, $\xi_m$ and $\theta_H$. 
\begin{lemma}\label{lem:hm}
Under Assumption \ref{assum-II},  we have 
{
\begin{align*}
\|\xi_H \times \xi_m\| & \lesssim \|\div\xi_H\|(h^{l+1}  +\|\xi_k\| +\|\div\xi_m\|), \\
\|\xi_m \times \theta_m\| & \lesssim h^{l+1/2} (h^{l+1}  + \|\xi_k\| + \|\div\xi_m\|), \\
\|\xi_m\times \theta_H\| & \lesssim h^{l+1/2} (h^{l+1}  +\|\xi_k\| + \|\div\xi_m\|).
\end{align*}}
\end{lemma}
\begin{proof}
The sixth equation of \eqref{eq:FHD-semi}  implies that 
$$
\curl_{h}\bs H_{h} = 0,
$$
which, together with the fact that $\curl\bs H = 0$, gives
$$
\curl_{h}\xi_{H} = \curl_{h}(I_{h}^{d}\bs H) - \curl_{h}\bs H_{h} = Q_{h}^{c} \curl\bs H = 0.
$$
This means $\xi_{H} \in \mathfrak K_{h}^{d}$. The H\"older inequality and Lemma \ref{lem:K-h} indicate
\begin{equation}\label{eq:H-T-M}
\|\xi_{H} \times \xi_{m}\| \lesssim \|\div\xi_{H}\|(\|\curl_{h}\xi_{m}\| + \|\div\xi_{m}\|).
\end{equation}
The inverse inequality and Lemma \ref{lem:vhbound} imply
{\begin{align}\label{eq:MH-eta0}
\|\xi_m \times \theta_m\|& \leq \|\xi_m\|_{L^\infty} \|\theta_m\|  \lesssim h^{-1/2}\|\xi_m\|_{L^6}\|\theta_m\| \\
& \lesssim h^{l+1/2}\|\bs m\|_{l+1}(\|\curl_h\xi_m\| + \|\div\xi_m\|) \nonumber
\end{align}
and}
\begin{equation}
\label{eq:MH-eta}
\begin{split}
\|\xi_{m} \times \theta_{H}\| & \leq \|\xi_{m}\|_{\infty}\|\theta_{H}\| \lesssim h^{-1/2}\|\xi_{m}\|_{L^{6}}\|\theta_{H}\| \\
&\lesssim h^{l+1/2}\|\bs H\|_{l+1} (\|\curl_{h}\xi_{m}\| + \|\div\xi_{m}\|).
\end{split}
\end{equation}
The fifth equation of 
\eqref{eq:FHD-semi} yields
$$ \bs k_h  = \curl_h\bs m_h,
$$
which, together with the fact that $
\bs k = \curl \bs m$, leads to 
\begin{align*}
\|\curl_h\xi_m\| & =  \|Q_h^c \bs k - \bs k_h\| \leq \|Q_h^c(I - I_h^c)\bs k\| +\|\xi_k\| \lesssim h^{l+1}\|\bs k\|_{l+1} + \|\xi_k\|.
\end{align*}
Combining this inequality with the estimates \eqref{eq:H-T-M}-\eqref{eq:MH-eta} gives the desired results.
\end{proof}


%
%

%
\begin{lemma}\label{lem:J2}
Under Assumption \ref{assum-II},  we have 
\begin{align*}
J_{1} + \frac{\beta}{4}\|\bs m_h \times\xi_H\|^2 \lesssim \ &   h^{l+1}\|\xi_u\| +  h^{l+1}(\|\xi_H\| + \|\div\xi_H\|) 
 \\
 & + h^{l+1/2} \|\xi_k\|\|\nabla\xi_u\| + h^{l+1/2}  \|\xi_H\|\|\nabla\xi_u\|  + h^{2(l+1)} \\
 & +  h^{2l+1} (h^{l+1}  + \|\xi_k\| + \|\div\xi_m\|)^2  +  \|\xi_u\|\|\xi_m\|_{\div} 
 \\ 
&+ \|\xi_m\|_{\div}\|\div\xi_H\| + \|\xi_u\|\|\xi_k\| + \|\xi_H\|\|\xi_k\| \\
&+  \|\nabla\xi_u\|\|\xi_u\|  + \|\xi_H\|\|\xi_m\|  +  \|\xi_m\|^2.
\end{align*}

\end{lemma}

\begin{proof}
For the terms $f_1(\xi_u)$ and $f_4(\xi_H)$, we   get
\begin{align*}
&|f_1(\xi_u)  + \mu_0f_4(\xi_H)| \\
 \leq& \mu_{0}|b(\xi_{u};\theta_{H},\bs m)|  + \mu_0 |b(\xi_{u}; I_{h}^{d}\bs H,\theta_{m})|
  + \mu_0 |b(\xi_{u}; I_{h}^{d}\bs H,\xi_{m})| \\
 & + \mu_0 |b(\theta_{u};\bs m,\xi_{H})|  + \mu_0 |b(\pi_{h}\bs u;\theta_{m},\xi_{H})| + \mu_0 |b(\pi_{h}\bs u;\xi_{m},\xi_{H})| \\
\leq & \|\div\bs m\|_{L^\infty}\|\xi_u\|\|\theta_H\|_{\div}    + \|\xi_u\|\|\bs H\|_{1,\infty}\|\theta_m\|_{\div}  + \|\bs H\|_{1,\infty}\|\xi_u\|\|\xi_m\|_{\div} \\
& + \|\bs m\|_{1,\infty}\|\xi_H\|_{\div}\|\theta_u\| + \|\pi_h\bs u\|_{L^\infty}\|\xi_H\|_{\div}\|\theta_m\|_{\div} + \|\pi_h\bs u\|_{L^\infty}\|\xi_m\|_{\div}\|\xi_H\|_{\div}\\
\lesssim  & ( \|\div \bs m\|_{L^\infty}\|\div\bs H\|_{l+1} + \|\bs H\|_{1,\infty}\|\div\bs m\|_{l+1}) h^{l+1}\|\xi_u\|\\
& + (h\|\bs m\|_{1,\infty}\|\bs u\|_{l+2} + \|\bs u\|_{L^\infty}\|\div\bs m\|_{l+1} )  h^{l+1}\|\div\xi_H\| + \|\bs H\|_{1,\infty}\|\xi_u\|\|\xi_m\|_{\div}  \\
& + \|\bs u\|_{L^\infty}\|\xi_m\|_{\div}\|\div\xi_H\|,
\end{align*}
where in the last inequality we have used the fact that $\xi_H \in \mathfrak K_h^d$ and $\|\xi_H\|_{\div} \lesssim \|\div\xi_H\|$.

For the terms $f_2(\xi_u)$ and $f_5(\xi_H)$,  we have  
\begin{align*}
&f_2(\xi_u)  + \mu_0 f_5(\xi_H)  \\
=& -\frac{\mu_0}{2}(\xi_u\times  \theta_{k},\bs H)  + \frac{\mu_0}{2} (\xi_{u}\times \xi_{k},\bs H) +\frac{\mu_0}{2} (\xi_{u}\times\xi_{k},\theta_{H}) \\
& -\frac{\mu_{0}}{2}  (\xi_u\times \bs k,\theta_H)    - \frac{\mu_0}{2}(\theta_{u}\times\xi_{H},\bs k)  +\frac{\mu_0}{2}  (\xi_{u}\times\xi_{H},\theta_{k})\\
& -\frac{\mu_0}{2}(\pi_{h}\bs u\times \xi_{H},\theta_{k})  +\frac{\mu_0}{2} (\pi_{h}\bs u \times\xi_{H},\xi_{k})  \\
 \lesssim  & \|\bs H\|_{L^\infty}\|\xi_u\|\|\theta_k\| + \|\bs H\|_{L^\infty}\|\xi_u\|\|\xi_k\| + \|\xi_u\|_{L^6}\|\xi_k\|_{L^3}\|\theta_H\| \\
 & + \|\bs k\|_{L^\infty}\|\xi_u\|\|\theta_H\| + \|\bs k\|_{L^\infty}\|\xi_H\|\|\theta_u\| + \|\xi_u\|_{L^6}\|\xi_H\|_{L^3}\|\theta_k\| \\
 & + \|\bs u\|_{L^\infty}\|\xi_H\|\|\theta_k\| + \|\bs u\|_{L^\infty}\|\xi_H\|\|\xi_k\| \\
 \lesssim & ( \|\bs k\|_{L^\infty}\|\bs H\|_{l+1} +\|\bs H\|_{L^\infty}\|\bs k\|_{l+1}) h^{l+1}\|\xi_u\| + ( h\|\bs k\|_{L^\infty}\|\bs u\|_{l+2}+\|\bs u\|_{L^\infty}\|\bs k\|_{l+1}) h^{l+1}\|\xi_H\| \\
 &+ \|\bs H\|_{l+1}  h^{l+1/2}\|\xi_k\|\|\nabla\xi_u\|  + \|\bs k\|_{l+1} h^{l+1/2}\|\xi_H\|\|\nabla\xi_u\| + \|\bs H\|_{L^\infty}\|\xi_u\|\|\xi_k\| + \|\bs u\|_{L^\infty}\|\xi_H\|\|\xi_k\|,
\end{align*}
where in the last inequality, we have used the fact   $H^1(\Omega) \hookrightarrow  L^6(\Omega)$ and the inverse inequality $\|\xi_k\|_{L^3} \lesssim h^{-1/2}\|\xi_k\|$.

For the term $f_3(\xi_u)$, it holds
\begin{align*}
f_3(\xi_u)  = &  -(\theta_{u}\times\curl\bs u,\xi_{u}) - (\pi_{h}\bs u \times\curl\theta_{u},\xi_{u})	+ (\pi_{h}\bs u\times \curl\xi_{u},\xi_{u}) \\
\lesssim & \|\bs u\|_{1,\infty}\|\theta_u\|\xi_u\| + \|\bs u\|_{L^\infty}\|\theta_u\|_1\|\xi_u\| + \|\bs u\|_{\infty}\|\nabla\xi_u\|\|\xi_u\| \\
\lesssim &  ( h\|\bs u\|_{1,\infty}\|\bs u\|_{l+2} + \|\bs u\|_{L^\infty}\|\bs u\|_{l+2}) h^{l+1}\|\xi_u\| + \|\bs u\|_{L^\infty}\|\nabla\xi_u\|\|\xi_u\|.
\end{align*}

For the term $f_6(\xi_H)$, it holds
{\begin{align*}
f_6(\xi_H) = & \beta \left\{(\theta_{m}\times \xi_H,\bs m \times \bs H) + (\xi_{H} \times \xi_{m},\bs m \times \bs H) - (\xi_{H} \times \bs m_h,\theta_{m} \times\bs H)) \right. \\
& \quad + \left.(\xi_{H} \times  \bs m_h,\xi_{m} \times \bs H) + (\xi_{H} \times \bs m_h,\bs m_h \times \theta_H) - \|\bs m_h \times \xi_H\|^2\right\}\\
\leq & \beta\left\{ \|\bs m\times\bs H\|_{L^\infty}\|\theta_m\|\|\xi_H\| + \|\bs m \times \bs H\|_{L^\infty}\|\xi_H\|\|\xi_m\| + \|\xi_H\times \bs m_h\|\|\theta_m\times \bs H\|\right. \\
&\quad + \left. \|\xi_H\times \bs m_h\|\|\xi_m \times \bs H\| + \|\xi_H \times \bs m_h\| \|\bs m_h \times\theta_H\| - \|\bs m_h \times \xi_H\|^2 \right\} \\
\leq & \beta \left\{ \|\bs m\times\bs H\|_{L^\infty}\|\theta_m\|\|\xi_H\| + \|\bs m \times \bs H\|_{L^\infty}\|\xi_H\|\|\xi_m\|  + \frac{1}{4}\|\bs m_h \times \xi_H\|^2 \right. \\
& \quad + \|\theta_m \times \bs H\|^2 +\frac{1}{4}\|\xi_H \times \bs m_h\|^2 + \|\xi_m \times \bs H\|^2  + \frac{1}{4}\|\xi_H\times \bs m_h\|^2\\
& \quad\left. + \|\bs m_h \times \theta_H\|^2  - \|\bs m_h \times \xi_H\|^2\right\} \\
\leq & \beta\left\{ \|\bs m\times\bs H\|_{L^\infty}\|\theta_m\|\|\xi_H\| + \|\bs m \times \bs H\|_{L^\infty}\|\xi_H\|\|\xi_m\|  + \|\theta_m \times \bs H\|^2\right. \\
&\quad \left. + \|\xi_m \times \bs H\|^2 + 2\|\xi_m \times \theta_H\|^2 + 2\|I_h^d\bs m \times\theta_H\|^2 - \frac{1}{4}\|\bs m_h \times\xi_H\|^2\right\},
\end{align*}}
which, together with Lemma \ref{lem:hm}, gives
{\begin{align*}
f_6(\xi_H) +  \frac{\beta}{4}\|\bs m_h \times\xi_H\|^2\lesssim &  \|\bs m \times\bs H\|_{L^\infty}\|\bs m\|_{l+1} h^{l+1}\|\xi_H\| +  \|\bs m \times\bs H\|_{L^\infty} \|\xi_H\|\|\xi_m\| \\
&+ (\|\bs H\|_{L^\infty}^2\|\bs m\|_{l+1}^2 + \|\bs m\|_{L^\infty}^2\|\bs H\|_{l+1}^2) h^{2(l+1)} + \|\bs H\|_{L^\infty}^2\|\xi_m\|^2 \\
& + \|\bs H\|_{l+1} h^{2l+1}(h^{l+1}\|\bs k\|_{l+1} + \|\xi_k\| + \|\div\xi_m\|)^2.
\end{align*}}

Finally, the desired result follows from the above obtained estimates and  the definition of $J_{1}$.
\end{proof}

For the term  $J_{4}$,  we have the following conclusion.
\begin{lemma}
\label{lem:J5}
Under Assumption \ref{assum-II}, we have 
\begin{align*}
J_{4} - \frac{\beta}{4}\|\bs m_h\times\xi_H\|^2 \lesssim \ &  h^{l+1}\|\xi_k\| +   h^{l+1}\|\xi_m\|_{\div} +  h^{l+1} \|\xi_z\| \\
&  + h^{l+1/2} \|\nabla\xi_u\|(\|\xi_k\| +\|\xi_m\|_{\div}) \\
& + h^{2l+1} (h^{l+1}+\|\xi_k\|+\|\div\xi_m\|)^2\\
& + h^{2l+3/2} (h^{l+1}  + \|\xi_k\| + \|\div\xi_m\|) \\
& +  h^{l+1/2} \|\xi_m\|(h^{l+1}  + \|\xi_k\| + \|\div\xi_m\|)\\
& + \|\xi_u\|\|\xi_k\| +  \|\xi_u\|\|\xi_m\|_{\div}  + \|\xi_m\|^2 .
\end{align*}

\end{lemma}
\begin{proof}
The fifth equations of \eqref{eq:FHD-weak} and \eqref{eq:FHD-semi} imply that
\begin{align*}
(\xi_{k},\bs\kappa) = (\xi_{m},\curl\bs\kappa) + (\theta_{k},\bs\kappa) 
\qquad \forall~~\bs\kappa\in \bs U_{h}.
\end{align*}
Taking $\bs\kappa = \xi_z\in \bs U_h$ in this equation, we have
$$
(\curl\xi_{z},\xi_{m}) = (\xi_{k},\xi_{z}) - (\theta_{k},\xi_{z}). 
$$
The fourth equations of \eqref{eq:FHD-weak} and \eqref{eq:FHD-semi} indicate that
$$
(\xi_{z},\bs\zeta) = (\bs u \times\bs m - \bs u_{h} \times \bs m_{h},\bs\zeta) + (\theta_{z},\bs \zeta)\qquad \forall~~~\bs \zeta \in \bs U_{h}.
$$
Taking $\bs\zeta = \xi_k \in \bs U_h$ in this relation, we get
\begin{align*}
(\curl\xi_{z},\xi_{m}) & = (\bs u \times \bs m - \bs u_{h} \times \bs m_{h},\xi_{k}) + (\theta_{z},\xi_{k}) - (\theta_{z},\xi_{z})
\\
& = -(\theta_{u} \times \bs m,\xi_{k}) + (\xi_{u}\times\bs m,\xi_{k}) + (\xi_{u} \times\theta_{m},\xi_{k}) - (\pi_{h}\bs u \times\theta_{m},\xi_{k})\\
&\quad - (\xi_{u} \times \xi_{m},\xi_{k}) + (\pi_{h}\bs u\times \xi_{m},\xi_{k}) + (\theta_{z},\xi_{k}) - (\theta_{z},\xi_{z}),
\end{align*}
which, together with the fact that
\begin{align*}
f_{5}(\xi_{m})  = &-\frac{1}{2}(\theta_{u} \times \xi_{m},\bs k) + \frac{1}{2}(\xi_{u} \times \xi_{m},\bs k) + \frac{1}{2}(\xi_{u} \times \xi_{m},\theta_{k}) - \frac{1}{2}(\pi_{h}\bs u \times \xi_{m},\eta_{k})\\
&  - \frac{1}{2}(\xi_{u} \times \xi_{m},\xi_{k}) + \frac{1}{2}(\pi_{h}\bs u \times \xi_{m},\xi_{k}),
\end{align*}
yields
\begin{align*}
&\frac{1}{2}(\curl\xi_{z},\xi_{m}) - f_{5}(\xi_{m}) \\
=& -\frac{1}{2}(\theta_{u} \times \bs m,\xi_{k}) + \frac{1}{2}(\xi_{u}\times\bs m,\xi_{k}) +\frac{1}{2} (\xi_{u} \times\theta_{m},\xi_{k})\\
&\quad - \frac{1}{2}(\pi_{h}\bs u \times\theta_{m},\xi_{k}) - \frac{1}{2}(\theta_{u} \times \xi_{m},\bs k) + \frac{1}{2}(\xi_{u}\times \xi_{m},\bs k) \\
&\quad + \frac{1}{2}(\xi_{u}\times\xi_{m},\theta_{k}) - \frac{1}{2}(\pi_{h}\bs u \times \xi_{m},\theta_{k})
+ \frac{1}{2}(\theta_{z},\xi_{k}) - \frac{1}{2}(\theta_{z},\xi_{z})\\
 \lesssim& \|m\|_{L^\infty}\|\theta_u\|\|\xi_k\| + \|\bs m\|_{L^\infty}\|\xi_u\|\|\xi_k\| + \|\xi_u\|_{L^6}\|\xi_k\|_{L^3}\|\theta_m\| \\
&\quad  + \|\bs u\|_{L^\infty}\|\theta_m\|\|\xi_k\| + \|\bs k\|_{L^\infty}\|\theta_u\|\|\xi_m\| + \|\bs k\|_{L^\infty}\|\xi_u\|\|\xi_m\| \\
&\quad + \|\xi_u\|_{L^6}\|\xi_m\|_{L^3}\|\theta_k\| + \|\bs u\|_{L^\infty}\|\xi_m\|\|\theta_k\| + \|\theta_z\|\|\xi_k\| + \|\theta_z\|\|\xi_z\| \\
 \lesssim&   h^{l+1}(\|\xi_k\| + \|\xi_m\| +\|\xi_z\|)   +  h^{l+1/2}\|\nabla\xi_u\| (\|\xi_k\| +\|\xi_m\|) +  \|\xi_u\|(\|\xi_k\| + \|\xi_m\|).
\end{align*}

For the term $f_4(\xi_m)$, we have 
\begin{align*}
f_{4}(\xi_{m})& = - b(\theta_{u};\bs m,\xi_{m}) + b(\xi_{u};\bs m,\xi_{m}) + b(\xi_{u};\theta_{m},\xi_{m}) - b(\pi_{h}\bs u;\theta_{m},\xi_{m}) \\
& \lesssim \|\bs m\|_{1,\infty}\|\theta_u\| \|\xi_m\|_{\div} + \|\bs m\|_{1,\infty}\|\xi_u\|\|\xi_m\|_{\div}\\
&\quad  + \|\xi_u\|_{L^\infty}\|\theta_m\|_{\div}\|\xi_m\|_{\div} + \|\bs u\|_{L^\infty}\|\theta_m\|_{\div}\|\xi_m\|_{\div} \\
& \lesssim   h^{l+1} \|\xi_m\|_{\div} +  h^{l+1/2}\|\nabla\xi_u\|\|\xi_m\|_{\div}  +  \|\xi_u\|\|\xi_m\|_{\div},
\end{align*}
where in the last inequality, we have used the inverse inequality $\|\xi_u\|_{L^\infty} \lesssim h^{-1/2}\|\xi_u\|_{L^6}$ and the inclusion result $H^1(\Omega) \hookrightarrow L^6(\Omega)$.

For the term $f_6(\xi_m)$, we have 
{
\begin{align*}
f_{6}(\xi_{m})  =& \beta\left\{(\theta_{m} \times\xi_{m},\bs m \times \bs H) +  (I_{h}^{d}\bs m \times\xi_{m},\theta_{m}\times\bs H) + (\xi_{m}\times I_{h}^{d}\bs m,\xi_{m} \times \bs H) \right.\\
&\quad + (\xi_{m}\times  \theta_m,\xi_{m} \times\theta_{H}) - (\xi_m\times\theta_{m}, I_{h}^{d}\bs m\times \theta_{H}) + (\xi_{m} \times  I_{h}^{d}\bs m,\xi_{m}\times\theta_{H}) \\
&\left.\quad - (\xi_{m} \times  I_{h}^{d}\bs m,I_{h}^{d}\bs m \times \theta_{H}) + (\xi_m \times\theta_m,\bs m_h \times \xi_H)\right\}\\
 \leq & \beta\left\{ \|\bs m \times \bs H\|_{L^\infty}\|\theta_m \|\|\xi_m\| + \|\bs m\|_{L^\infty}\|\bs H\|_{L^\infty}\|\xi_m\|\|\theta_m\|  + \|\bs m\|_{L^\infty}\|\bs H\|_{L^\infty}\|\xi_m\|^2 \right. \\
& \quad + \|\xi_m\times\theta_m\|\|\xi_m\times\theta_H\| + \|\bs m\|_{L^\infty}\|\xi_m\times\theta_m\|\|\theta_H\|  +\|\bs m\|_{L^\infty}\|\xi_m\|\|\xi_m\times\theta_H\|
\\
&\left. \quad + \|\bs m\|_{L^\infty}^2\|\xi_m\|\|\theta_H\| + \|\xi_m\times\theta_m\|\|\bs m_h \times\xi_H\| \right\}\\
\leq & \beta \left\{ \|\bs m \times \bs H\|_{L^\infty}\|\theta_m \|\|\xi_m\| + \|\bs m\|_{L^\infty}\|\bs H\|_{L^\infty}\|\xi_m\|\|\theta_m\|  + \|\bs m\|_{L^\infty}\|\bs H\|_{L^\infty}\|\xi_m\|^2 \right. \\
& \quad + \|\xi_m\times\theta_m\|\|\xi_m\times\theta_H\| + \|\bs m\|_{L^\infty}\|\xi_m\times\theta_m\|\|\theta_H\|  +\|\bs m\|_{L^\infty}\|\xi_m\|\|\xi_m\times\theta_H\|
\\
&\left. \quad + \|\bs m\|_{L^\infty}^2\|\xi_m\|\|\theta_H\| + \|\xi_m\times\theta_m\|^2 + \frac{1}{4}\|\bs m_h \times\xi_H\|^2 \right\},
 \end{align*}}
 which, together with Lemma \ref{lem:hm}, means that 
 \begin{align*}
f_{6}(\xi_{m})-\frac{\beta}{4}\|\bs m_h \times \xi_H\|^2 \lesssim \ &   h^{l+1}\|\xi_m\| +  \|\xi_m\|^2+  h^{2l+1}(h^{l+1}  +\|\xi_k\| + \|\div\xi_m\|)^2 \\
 & \quad + h^{2l+3/2}(h^{l+1} + \|\xi_k\| + \|\div\xi_m\|) \\
 &\quad +  h^{l+1/2}\|\xi_m\|(h^{l+1} +\|\xi_k\| + \|\div\xi_m\|) \\
 &\quad +   h^{2l+1}(h^{l+1} +\|\xi_k\| + \|\div\xi_m\|)^2.
 \end{align*}
 
Combining all the above estimates gives the desired result.
\end{proof}

 \begin{lemma}
 \label{lem:xiz}
Under Assumption \ref{assum-II}, we have 
 \begin{align*}
 \|\xi_{z}\| \lesssim\  &   h^{l+1}  +  h^{l+1/2}\|\nabla\xi_u\| + \|\xi_u\|.
\end{align*}
 \end{lemma}
 \begin{proof}
The fourth equations of \eqref{eq:FHD-weak} and \eqref{eq:FHD-semi} imply that
 $$
 (\xi_{z},\bs \zeta_{h}) = (\bs u \times\bs m - \bs u_{h}\times\bs m_{h},\bs\zeta_{h}) + (\theta_{z},\bs\zeta_{h})\qquad\forall~~\bs\zeta_{h}\in \bs U_{h}.
 $$
Taking $\bs\zeta_h = \xi_z\in \bs U_h$ in the above equation, we have
\begin{align*}
 \|\xi_{z}\|^2  = & (\bs u \times\bs m - \bs u_{h}\times\bs m_{h},\xi_z)  + (\theta_{z},\xi_z)\\
= & -(\theta_u\times \bs m,\xi_z) + (\xi_u\times \bs m ,\xi_z) + (\xi_u \times \theta_m,\xi_z)  - (\pi_h\bs u\times \theta_m,\xi_z)+ (\theta_z,\xi_z)  \\
\leq & \|\bs m\|_{L^\infty}\|\theta_u\|\|\xi_z\| + \|\bs m\|_{L^\infty}\|\xi_u\|\|\xi_z\| + \|\xi_u\|_{L^\infty}\|\theta_m\|\|\xi_z\| \\
&\quad + \|\pi_h\bs u\|_{L^\infty}\|\theta_m\|\|\xi_z\| + \|\theta_z\| \|\xi_z\| \\
\lesssim  &      h^{l+1}\|\xi_z\| +   h^{l+1/2} \|\nabla\xi_u\| \|\xi_z\|  + \|\xi_u\|\|\xi_z\|,
\end{align*}
where in the last inequality, we have used Lemma \ref{lem:vhbound} and the estimate of $\|\curl_h\xi_m\|$ in the proof of Lemma \ref{lem:hm}. Then the desired result follows by cancelling $\|\xi_z\|$ from the above inequality.
%
%
 \end{proof}
 
 \begin{lemma}
 \label{lem:xip}
 Under Assumption \ref{assum-II},
 for any $t \in (0,T]$ we have
 \begin{align*}
\int_{0}^{t}\|\xi_{p}\|\dd t & \lesssim \|\xi_u\| +  t  h^{l+2}  +   h^{l+1} +  h^{l+1} \int_0^t\|\nabla\xi_u\|\dd t  +   h^{l+1/2}\int_0^t\|\xi_k\|\dd t \\
& \quad +  \int_0^t\|\div\xi_H\|\dd t + \|\div\bs H_h\|_{L^\infty(L^3)}\int_0^t\|\xi_m\|\dd t + \int_0^t\|\xi_k\|\dd t \\
& \quad + \int_0^t\|\nabla \xi_u\|\dd t + \int_0^t \left(\|\nabla\xi_u\|^2 + \|\xi_k\|\|\div\xi_H\| \right) \dd t.
\end{align*}
 \end{lemma}
 \begin{proof}
 The inf-sup condition \eqref{eq:inf-sup} and  equation \eqref{eq:err-semi-1} imply that for any $0\leq t \leq T$,
\begin{align*}
\int_{0}^{t}\|\xi_{p}\|\dd t  & \lesssim \sup\limits_{\bs v_{h}\in \bs S_{h}}\int_{0}^{t}\frac{(\xi_{p},\nabla\cdot\bs v_{h})}{\|\bs v_{h}\|_{1}}\dd t \\
& \lesssim \|\xi_{u}(\cdot,t)\| +  \int_{0}^{t}\|\nabla\xi_{u}\|\dd t +   h^{l+2} \int_{0}^{t}\|\partial_{t}\bs u\|_{l+2}\dd t  \\
&\quad +  h^{l+1}\int_{0}^{t}\|\bs u\|_{l+2}\dd t + h^{l+1}\int_0^t\|\tilde p\|_{l+1}\dd t  + T_1
\end{align*}
with
$$
T_1 := \sup\limits_{\bs v_{h} \in \bs S_{h}}\int_{0}^{t}\frac{|f_{1}(\bs v_{h}) +  f_{2}(\bs v_{h}) - f_3(\bs v_h)|}{\|\bs v_{h}\|_{1}}\dd t,
$$
where $f_1(\bs v_h), f_2(\bs v_h)$ and $ f_3(\bs v_h)$ are as same as those  in  \eqref{eq:err-semi-1}.

For any $\bs v_h \in \bs S_h$, we have
\begin{align*}
|f_1(\bs v_h)| =& \mu_0 
|b(\bs v_{h};\xi_{H},\bs m) - b(\bs v_{h};\theta_{H},\bs m) - b(\bs v_{h};\bs H_{h},\theta_{m}) + b(\bs v_{h};\bs H_{h},\xi_{m}) | \\
\lesssim & \|\bs m\|_{1,\infty}\|\bs v_h\|\|\xi_H\|_{\div} + \|\bs m\|_{1,\infty}\|\bs v_h\|\|\theta_H\|_{\div}  \\
& + \|\bs v_h\|_{L^6}\|\div\bs H_h\|_{L^3}\|\theta_m\|_{\div} + \|\bs v_h\|_{L^6}\|\div\bs H_h\|_{L^3}\|\xi_m\|_{\div} \\
\lesssim & \|\bs v_h\|\|\xi_H\|_{\div} +    h^{l+1} \|\bs v_h\|   +  \|\nabla\bs v_h\| \|\div\bs H_h\|_{L^3} (h^{l+1}+ \|\xi_m\|), \\
|f_2(\bs v_h)|  = & \frac{\mu_{0}}{2} \left| (\bs v_{h} \times\xi_{k},\bs H) - (\bs v_{h}\times \theta_{k},\bs H) + (\bs v_{h}\times \xi_k,\theta_{H}) \right.\\
& \quad\quad \left. -(\bs v_h \times I_h^c\bs k,\theta_H) - (\bs v_{h}\times \xi_k,\xi_{H}) + (\bs v_h \times I_h^c\bs k,\xi_H)\right| \\
\lesssim & \|\bs H\|_{L^\infty}\|\bs v_h\| \|\xi_k\| + \|\bs H\|_{L^\infty}\|\bs v_h\|\|\theta_k\| + \|\bs v_h\|_{L^6} \|\xi_k\|_{L^3} \|\theta_H\|\\
& + \|I_h^c\bs k\|_{L^\infty}\|\bs v_h\|\|\theta_H\| + \|\bs v_h\|_{L^6}\|\xi_H\|_{L^3} \|\xi_k\| + \|\bs v_h\|\|I_h^c\bs k\|_{L^\infty}\|\xi_H\|\\
\lesssim &  \|\bs v_h\|\|\xi_k\| +  h^{l+1}\|\bs v_h\| + \|\xi_k\| \|\nabla\bs v_h\| (h^{l+1/2}  +   \|\div\xi_H\|)  + \|\bs v_h\|\|\xi_H\|,
\end{align*}
%
and 
\begin{align*}
|f_{3}(\bs v_{h})| = &  \left|(\xi_{u}\times\curl\bs u,\bs v_{h}) - (\theta_{u}\times\curl\bs u,\bs v_{h})  + (\xi_u\times \curl\theta_{u},\bs v_{h})\right. \\
& \left. - (\pi_h\bs u\times \curl\theta_{u},\bs v_{h}) - (\xi_u\times\curl\xi_u,\bs v_h) + (\pi_h\bs u\times\curl\xi_u,\bs v_h)\right| \\
\lesssim & \|\bs u\|_{1,\infty}\|\xi_u\|\|\bs v_h\| + \|\bs u\|_{1,\infty}\|\theta_u\| \|\bs v_h\| + \|\nabla\xi_u\|\|\theta_u\|_1\|\nabla\bs v_h\| \\
& + \|\bs u\|_{1,\infty}\|\bs v_h\| \|\theta_u\|_1 + \|\nabla\bs v_h\|  \|\nabla\xi_u\|^2 + \|\bs u\|_{1,\infty}\|\nabla\xi_u\|\|\bs v_h\| \\
\lesssim & \|\xi_u\|\|\bs v_h\| +  h^{l+2}\|\bs v_h\| +  h^{l+1}\|\nabla\xi_u\| \|\bs v_h\| \\
& +   h^{l+1}\|\bs v_h\| + \|\nabla\bs v_h\|  \|\nabla\xi_u\|^2 +  \|\nabla\xi_u\|\|\bs v_h\| .
\end{align*}
As a result,  the desired result follows.
 \end{proof}

Denote 
$$
\mathfrak E_h(t): = \|\xi_{u}(\cdot,t)\|^{2} + \|\xi_{H}(\cdot,t)\|^{2} + \|\xi_{m}(\cdot,t)\|^{2}
$$ 
and
$$
\mathfrak F_h(t) := \|\nabla\xi_{u}(\cdot,t)\|^{2} + \|\xi_{k}(\cdot,t)\|^{2} + \|\xi_{H}(\cdot,t)\|_{\div}^{2} + \|\xi_{m}(\cdot,t)\|_{\div}^{2} .
$$
We have the following Lemma.

\begin{lemma}\label{lem:xi-estimates}
Under Assumption \ref{assum-II} and the condition  $\|\div\bs H_h\|_{L^\infty(L^3)}\leq C$, for any $t \in [0,T]$ we have  
$$
\int_{0}^{t}\left( \mathfrak E_h(r) + \int_{0}^{r}\mathfrak F_h(s)\dd s\right)\dd r \lesssim h^{2(l+1)}.
$$
\end{lemma}
\begin{proof}
By Lemmas \ref{lem:lem:err-eq-1}, \ref{lem:J23}, \ref{lem:J2} and \ref{lem:J5}, we have
\begin{equation}\label{eq:lem-pro}
\begin{split}
& \frac{1}{2}\frac{\dd }{\dd t}\left(\|\xi_{u}\|^{2} + \mu_{0}\|\xi_{H}\|^{2}  + \|\xi_{m}\|^{2} \right) + \eta\|\nabla\xi_{u}\|^{2} + \frac{1}{\tau}(\mu_{0}(1+\chi_{0})+ \chi_{0})\|\xi_{H}\|^{2}  \\
& \quad  + \sigma\|\xi_{k}\|^{2}  + \frac{1}{\tau}\|\xi_{m}\|^{2} + \sigma\mu_{0}\|\div\xi_H\|^2+ \sigma\|\div\xi_{m}\|^{2}  \\
\lesssim& 
 h^{l+1}  \mathfrak T_{1} + h^{l+1/2}\mathfrak T_{2} + \mathfrak T_{3} ,
\end{split}
\end{equation}
with 
\begin{align*}
\mathfrak T_{1} &: = \|\nabla\xi_{u}\| + \|\xi_H\|_{\div} + \|\xi_m\|_{\div} + \|\xi_k\| + \|\xi_z\| + \|\xi_p\|,\\
\mathfrak T_{2} & := \|\nabla\xi_u\|(\|\xi_k\|+\|\xi_H\| + \|\xi_m\|_{\div}) 
 + \|\xi_m\|(\|\xi_k\| + \|\div\xi_m\|) ,\\
\mathfrak T_{3} &: = \|\xi_{u}\|(\|\xi_{k}\| + \|\xi_{m}\|_{\div}+ \|\nabla\xi_u\|)  + \|\xi_m\|_{\div}\|\div\xi_H\| + \|\xi_H\|\|\xi_k\|  \\
&\qquad + \|\xi_{m}\|^2 + \|\xi_m\| \|\xi_H\|.
\end{align*}
%

For any $t \in [0,T]$, using Lemmas \ref{lem:xiz} and \ref{lem:xip}, we get
\begin{align*}
\int_{0}^{t}\mathfrak T_{1} \dd t & \lesssim h^{l+1} + \|\xi_{u}\| + \int_{0}^{t} \left(\|\xi_{H}\|_{\div} + \|\xi_{m}\|_{\div} + \|\nabla\xi_{u}\| + \|\xi_{k}\| \right)\dd t\\
& \quad + \int_0^t(\|\nabla\xi_u\|^2 + \|\xi_k\|\|\div\xi_H\|)\dd t.
\end{align*}
%
Then, integrate \eqref{eq:lem-pro} on the interval $(0,t)$, and we obtain
\begin{align*}
\mathfrak E_h(t)  + \int_{0}^{t}\mathfrak F_h(r) \dd r
& =h^{l+1} \int_{0}^{t}\mathfrak T_{1}\dd t + h^{l+1/2}\int_{0}^{t}\mathfrak T_{2}\dd t  + \int_{0}^{t}\mathfrak T_{3}\dd t \\
& \lesssim h^{2(l+1)} + h^{l+1}\mathfrak E_h^{1/2}(t) + \int_{0}^{t}\mathfrak E_h^{1/2}(r)\mathfrak F_h^{1/2}(r) \dd r \\
& \lesssim  h^{2(l+1)} + h^{l+1}\left(\mathfrak E_h(t) + \int_0^t\mathfrak F(r)\dd r\right)^{1/2} \\
& \quad + \left(\int_0^t \left(\mathfrak E_h(r) + \int_0^r\mathfrak F_h(s)\dd s\right)\right)^{1/2}\left(\mathfrak E_h(t) + \int_0^t\mathfrak F_h(r)\dd r\right)^{1/2},
\end{align*}
which shows 
\begin{align*}
 \mathfrak E_h(t) + \int_{0}^{t}\mathfrak F_h(r) \dd r  & \lesssim h^{2(l+1)} + \int_{0}^{t}(\mathfrak E_h(r) + \int_{0}^{r}\mathfrak F_h(l)\dd l)\dd r.
\end{align*}
Integrating this inequality on the interval $(0,\bar t)$ for any $\bar t \in [0,T]$ yields
$$
\int_{0}^{\bar t}\left(\mathfrak E_h(t) + \int_{0}^{t}\mathfrak F_h(r) \dd r \right) \dd t \lesssim { \bar t }^{2}h^{2(l+1)} + \int_{0}^{\bar t}\left[ \int_{0}^{t} \left(\mathfrak E_h(r) + \int_{0}^{r}\mathfrak F_h(s)\dd s \right) \dd r \right]\dd t,
$$
which, together with the Gronwall's inequality \cite[Theorem 1.2.2]{Pachpatte1998}, implies the desired result.
\end{proof}

Define the errors
$$
\text{err}_{1}(t) := \|\bs u(\cdot,t) - \bs u_{h}(\cdot,t)\|^{2} + \|\bs m(\cdot,t) - \bs m_{h}(\cdot,t)\|^{2} + \|\bs H(\cdot,t)-\bs H_{h}(\cdot,t)\|^{2}
$$
and
\begin{align*}
\text{err}_{2}(t) & := \|\nabla(\bs u - \bs u_{h})(\cdot,t)\|^{2} + \|\div(\bs m - \bs m_{h})(\cdot,t)\|^{2} + \|\div(\bs H - \bs H_{h})(\cdot,t)\|^{2}.
\end{align*}
We finish this section by giving   the following error estimates for the semi-discrete scheme \eqref{eq:FHD-semi}:
\begin{theorem}
Under Assumption \ref{assum-II} and the condition 
$\|\div\bs H_h\|_{L^\infty(L^3)}\leq C$, for any $t \in [0,T]$    we have  
\begin{align*}
&\int_{0}^{t}\left( \text{err}_{1}(s) + \int_{0}^{s} \text{err}_{2}(r)\dd r\right) \dd s \lesssim h^{2(l+1)},\\
&\int_{0}^{t}\|\tilde p(\cdot,s) - \tilde p_{h}(\cdot,s)\|^2\dd s \lesssim h^{2(l+1)}.
\end{align*}
\end{theorem}
\begin{proof}
For any $s \in [0,t]$, the triangular inequality implies
\begin{align*}
\text{err}_{1}(s)& \lesssim \|\theta_{u}(\cdot,s)\|^{2} + \|\theta_{m}(\cdot,s)\|^{2} + \|\theta_{H}(\cdot,s)\|^{2} + \mathfrak E(s),\\
\text{err}_{2}(s) &\lesssim \|\nabla\theta_{u}(\cdot,s)\|^{2} + \|\div\theta_{m}(\cdot,s)\|^{2} + \|\div\theta_{H}(\cdot,s)\|^{2} \\
& \quad + \|\nabla\xi_{u}(\cdot,s)\|^{2} + \|\div\xi_{m}(\cdot,s)\|^{2} + \|\div\xi_{H}(\cdot,s)\|^{2},
\end{align*}
and
$$
\|\tilde p(\cdot,s) - \tilde p_{h}(\cdot,s)\| \leq \|\theta_{p}(\cdot,s)\| + \|\xi_{p}(\cdot,s)\|.
$$
Then Lemmas \ref{lem:xip} and \ref{lem:xi-estimates} yield the desired results.
\end{proof}

\section{Fully discrete finite element scheme }
In this section, we give a fully discrete scheme for the ferrofluid model \eqref{eq:FHD}, based on the semi-discretization \eqref{eq:FHD-semi}. We will prove that this full-discretization  also  preserves the energy exactly and has optimal convergence.

\subsection{Full discretization}
For any positive integer $N$, let
$$
\mathcal J_{\dt} = \{t_{n}:\ t_{n} = n\dt,\ 0\leq n\leq N \}
$$
be a uniform partition of the time interval $[0,T]$ with the time step size $\dt = T/N$. For any $1 \leq n \leq N$,  we  denote by $I_{n}$ the interval $(t_{n-1},t_{n})$ and by  $v^{n}$ the numerical approximation of $v(t_{n})$  for any quantity $v(t)$ (both scalar or vector). We also set 
$$
\delta_{t}v^{n} := \frac{v^{n} - v^{n-1}}{\dt}.
$$ 
Then the fully discrete scheme of the ferrofluid flow model \eqref{eq:FHD} 
reads as: Given $\bs u_h^{n-1} \in S_h$ and $\bs m_h^{n-1}\in \bs V_h$ for $n=1,2,\cdots$, find $\bs u_h^{n} \in \bs S_h$, $\tilde p_h^{n}\in L_h$, $\bs m_h^{n} \in  \bs V_h$,  $\bs z_h^{n}\in \bs U_h$, $\bs k_h^{n} \in \bs U_h$, $\bs H_h^{n} \in  \bs V_h$, and $\varphi_h^n \in W_h$ such that
\begin{equation}\label{eq:FHD-full}
\left\{
\begin{array}{rll}
 (\delta_t\bs u_h^n,\bs v_h)  - (\bs u_h^n\times \curl\bs u_h^{n},\bs v_h) + \eta(\nabla\bs u_h^n,\nabla\bs v_h) & & \\- (\tilde p_h^n,\div \bs v_h) 
 -\mu_0b(\bs v_h;\bs H_h^n,\bs m_h^n) 
 -\frac{\mu_0}{2}(\bs v_h\times \bs k_h^n,\bs H_h^n) & = 0
&\quad\forall~\bs v_h \in \bs S_h,\\
(\div\bs u_h^n,q_h)   & = 0 &\quad\forall~q_h\in L_h, \\ 
 (\delta_t\bs m_h^n,\bs F_h) + b(\bs u_h^n;\bs m_h^n,\bs F_h) + \frac{1}{2}(\bs u_h^n\times\bs F_h,\bs k_h^n) 
 +\frac{1}{\tau}(\bs m_h^n,\bs F_h) \\
  -\frac{1}{2} (\curl\bs z_h^n,\bs F_h) 
  + \sigma(\curl\bs k_h^n,\bs F_h)  + \sigma(\div\bs m_h^n, \div\bs F_h)& & \\
  - \frac{\chi_0}{\tau}(\bs H_h^n,\bs F_h) 
  + \beta(\bs m_h^n\times (\bs m_h^n \times H_h^n),\bs F_h) &= 0
& \quad\forall~\bs F_h\in \bs V_h,\\
 (\bs z_h^n,\bs\zeta_h)  - (\bs u_h^n\times \bs m_h^n,\bs \zeta_h) & = 0 &\quad\forall~\bs\zeta_h\in \bs U_h,\\
(\bs k_h^n,\bs\kappa_h) - (\bs m_h^n,\curl\bs\kappa_h)& = 0 & \quad\forall~\bs\kappa_h\in \bs U_h,\\
(\bs H_h^n,\bs G_h ) + (\varphi_h^n,\div\bs G_h) & = 0 & \quad\forall~\bs G_h\in \bs V_h, \\
(\div\bs H_h^n + \div\bs m_h^n, r_h) +(\div \bs H_{e}^n,r_h) & = 0& \quad\forall~r_h \in W_h.
\end{array}
\right.
\end{equation}

In the rest of this subsection, we will focus on the solvability of the  scheme \eqref{eq:FHD-full}. 
Denote $$\mathfrak V_{h}: = \mathfrak K_{h}^{s} \times \bs V_{h} \times \mathfrak K_{h}^{d},\quad \mathfrak W_{h}: = \mathfrak K_{h}^{s} \times \bs V_{h} \times W_{h},$$
 and  define a functional $\mathcal A(\cdot,\cdot): ~\mathfrak V_{h} \times \mathfrak W_{h} \rightarrow \mathbb R$ by
\begin{align*}
\mathcal A(\mathcal U_{h}^{n},\mathcal V_{h}) &: = (\bs u_{h}^{n},\bs v_h) + \mu_{0}(\bs m_{h}^{n},\bs F_h)  - \dt (\bs u_{h}^{n} \times \curl\bs u_{h}^{n},\bs v_h) + \dt\eta(\nabla\bs u_{h}^{n},\nabla\bs v_h) \\
& - \dt\mu_{0}b(\bs v_h;\bs H_{h}^{n},\bs m_{h}^{n}) - \frac{\mu_{0}\dt}{2} (\bs v_h \times \bs \curl_{h}\bs m_{h}^{n},\bs H_{h}^{n}) + \dt\mu_{0}b(\bs u_{h}^{n};\bs m_{h}^{n},\bs F_h) \\
& + \frac{\mu_{0}\dt}{2} (\bs u_{h}^{n} \times \bs F_h,\curl_{h}\bs m_{h}^{n}) - \frac{\mu_{0}\dt}{2} (\bs u_{h}^{n} \times \bs m_{h}^{n},\curl_{h}\bs F_h) + \frac{\mu_{0}\dt}{\tau}(\bs m_{h}^{n},\bs F_h) \\
& + \sigma\mu_{0}\dt(\curl_{h}\bs m_{h}^{n},\curl_{h}\bs F) + \mu_{0}\sigma\dt (\div\bs m_{h}^{n},\div\bs F_h)+ \mu_{0}\dt(\div\bs H_{h}^{n},\phi_h)  \\
& - \frac{\mu_{0}\chi_{0}\dt}{\tau}(\bs H_{h}^{n},\bs F_h) + \beta \mu_{0} \dt(\bs m_{h}^{n} \times (\bs m_{h}^{n} \times \bs H_{h}^{n}),\bs F_h) +\mu_{0}\dt (\div\bs m_{h}^{n},\phi_h)  ,
\end{align*}
with
$$
\mathcal U_{h}^{n} = (\bs u_{h}^{n},\bs m_{h}^{n},\bs H_{h}^{n}) \in\mathfrak V_{h},\qquad \mathcal V_{h} = (\bs v_h,\bs F_h,\phi_h)\in\mathfrak W_{h}.
$$
Define  the   norms $\|\cdot\|_{\mathfrak V}$ and $\|\cdot\|_{\mathfrak W}$ on $\mathfrak V_h$ and $\mathfrak W_h$, respectively  as follows:
\begin{align*}
\|\mathcal U_h\|_{\mathfrak V}^2 & := \|\bs u_h\|^2 + \dt\eta\|\nabla\bs u_h\|^2 + \frac{\mu_0\dt(2+\chi_0)}{\tau}\|\bs H_h\|^2 + \mu_0\sigma\dt\|\div\bs H_h\|^2 \\
& \quad + \frac{\mu_0\dt}{\tau}\|\bs m_h\|^2 + \mu_0\sigma\dt \|\curl_h\bs m_h\|^2 + \mu_0\sigma\dt\|\div\bs m_h\|^2, \\
\|\mathcal V_{h}\|_{\mathfrak W}^{2} &:  = \|\bs v_{h}\|^{2} + \dt \eta\|\nabla\bs v_{h}\|^{2} + \dt \|\bs F_{h}\|^{2} + \dt \|\curl_{h}\bs F_{h}\|^{2}  + \dt\|\div \bs F_{h})\|^{2} + \dt \|\phi_h\|^{2}.
\end{align*}

Introduce an auxiliary problem: Find $\mathcal U_{h}^{n} = (\bs u_{h}^{n},\bs m_{h}^{n},\bs H_{h}^{n})\in \bs S_{h} \times \bs V_{h} \times \mathfrak K_{h}^{d}$ and $\tilde p_{h}^n \in L_{h}$ such that
\begin{equation}\label{eq:aux-pro}
\left\{
\begin{array}{ll}
\mathcal A(\mathcal U_{h}^{n},\mathcal V_{h})   + \dt(\tilde p_{h}^{n},\div\bs v_{h})= \bs l(\mathcal V_{h})&  \forall~\mathcal V_{h} = (\bs v_h,\bs F_h,\phi_h)\in \bs S_{h} \times \bs V_{h} \times W_{h},\\
(\div\bs u_{h}^{n},q_{h}) = 0 &\forall~~~q_{h} \in L_{h},
\end{array}
\right.
\end{equation} 
with $$\bs l (\mathcal V_{h}) := (\bs u_{h}^{n-1},\bs v_h) + \mu_{0}(\bs m_{h}^{n-1},\bs F_h) - \mu_{0}\dt(Q_{h}\div \bs H_{e}^{n},\phi_h).$$
{\begin{remark}
From the exact sequence \eqref{eq:exact_sq} we know that $\div\mathfrak K_h^d = W_h$ and that for any $\phi_h \in W_h$, there exists a unique $\bs v_h \in \mathfrak K_h^d$ such that $\div\bs v_h = \phi_h$. Since $\div$ is a linear operator, it holds   $\dim \mathfrak K_h^d = \dim W_h$. Therefore, the matrix representation of $\mathcal A$ in \eqref{eq:aux-pro} is a square matrix.
\end{remark}}

We have the following lemma.
\begin{lemma}\label{lem:equ-solu}
The full-discrete form \eqref{eq:FHD-full} and the auxiliary problem \eqref{eq:aux-pro} are equivalent in the following sense:
\begin{itemize}
\item if $\bs u_h^{n} \in \bs S_h$, $\tilde p_h^{n}\in L_h$, $\bs m_h^{n} \in  \bs V_h$,  $\bs z_h^{n}\in \bs U_h$, $\bs k_h^{n} \in \bs U_h$, $\bs H_h^{n} \in  \bs V_h$, and $\varphi_h^n \in W_h$ solve \eqref{eq:FHD-full}, then $\mathcal U_{h}^{n} = (\bs u_{h}^{n},  \bs m_{h}^{n},\bs H_{h}^{n}) \in \mathfrak V_{h}$ and $\tilde p_{h}^{n}\in L_h$ solve \eqref{eq:aux-pro}; 

\item if $\mathcal U_{h}^{n} = (\bs u_{h}^{n},\bs m_{h}^{n},\bs H_{h}^{n}) \in \mathfrak V_{h}$ and $\tilde p_{h}^{n}\in L_h$  solve  \eqref{eq:aux-pro}, then $\bs u_{h}^{n}$, $\tilde p_{h}^{n}$, $\bs m_{h}^{n}$,  $\bs z_{h}^{n} = Q_{h}^{c}(\bs u_{h}^{n} \times \bs m_{h}^{n})$, $\bs k_{h}^{n} = \curl_{h}\bs m_{h}^{n}$, and   a unique $\varphi_h^n\in W_h$ satisfying $\bs H_h^n = \grad_h\varphi_h^n$ solve \eqref{eq:FHD-full}.
\end{itemize}
\end{lemma}
\begin{proof}
We assume that $\bs u_h^{n} \in \bs S_h$, $\tilde p_h^{n}\in L_h$, $\bs m_h^{n} \in  \bs V_h$,  $\bs z_h^{n}\in \bs U_h$, $\bs k_h^{n} \in \bs U_h$, $\bs H_h^{n} \in  \bs V_h$, and $\varphi_h^n \in W_h$ solve \eqref{eq:FHD-full}. The fourth, fifth and sixth equations of \eqref{eq:FHD-full} imply that  
$$\bs k_h^n = \curl_h \bs H_h^n, \quad \bs z_h^n = Q_h^c (\bs u_h^n \times \bs m_h^n), \quad \bs H_h^n = \grad_h\varphi_h^n \in \mathfrak K_h^d.$$
Substitute them  into the first and third equations of \eqref{eq:FHD-full}, then we get
\begin{equation*}
\left\{
\begin{array}{rll}
 (\bs u_h^n,\bs v_h) - \dt(\bs u_h^n \times \curl\bs u_h^{n},\bs v_h) + \eta \dt(\nabla\bs u_h^n,\nabla \bs v_h) & &  \\ - \dt(\tilde p_h^n,\div\bs v_h) 
- \mu_0 \dt b(\bs v_h;\bs H_h^n,\bs m_h^n) & & \\
- \frac{\mu_0\dt}{2}(\bs v_h \times \curl_h \bs m_h^n,\bs H_h^n) - (\bs u_h^{n-1},\bs v_h) & = 0
& \quad\forall~~\bs v_h \in \bs S_h,
\\\\
 (\bs m_h^n,\bs F_h) + \dt b(\bs u_h^n;\bs m_h^n,\bs F_h) + \frac{\dt}{2}(\bs u_h^n \times \bs F_h,\curl_h\bs m_h^n) & & \\
 - \frac{\dt}{2} (\bs u_h^n \times \bs m_h^n ,\curl_h\bs F_h) + \sigma \dt (\curl_h\bs m_h^n,\curl_h \bs F_h) & & \\
 + \sigma \dt(\div\bs m_h^n,\div\bs F_h) + \frac{\dt}{\tau}(\bs m_h^n,\bs F_h) - \frac{\chi_0\dt}{\tau}(\bs H_h^n,\bs F_h) & & \\
 + \beta\dt (\bs m_h^n\times (\bs m_h^n \times \bs H_h^n),\bs F_h) - (\bs m_h^{n-1},\bs F_h) & = 0
&\quad \forall~~\bs F_h \in \bs V_h,
\end{array}
\right.
\end{equation*}
which, together with the last and  second equations of \eqref{eq:FHD-full},  gives \eqref{eq:aux-pro}.

On the other hand, we assume that $\mathcal U_{h}^{n} = (\bs u_{h}^{n},\bs m_{h}^{n},\bs H_{h}^{n}) \in \mathfrak V_{h}$ and $\tilde p_{h}^{n}\in L_h$ solve \eqref{eq:aux-pro}, then it is easy to check that $\bs z_{h}^{n} = Q_{h}^{c}(\bs u_{h}^{n} \times \bs m_{h}^{n})$ and $\bs k_{h}^{n} = \curl_{h}\bs m_{h}^{n}$ satisfy the fourth and fifth equations of \eqref{eq:FHD-full}. The definition of $\mathfrak K_h^d$ implies that  there exists a unique $\varphi_h^n \in W_h$ such that $\bs H_h^n = \grad_h\varphi_h^n$ satisfying the sixth equation of \eqref{eq:FHD-full}. 
Taking $\mathcal V_h =(\bs v_h,\bs 0,0) \in \mathfrak W_h$, $q_h = 0 \in L_h$ and $\mathcal V_h =(\bs 0,\bs 0,0)\in \mathfrak W_h$, $q_h\in L_h$    in \eqref{eq:aux-pro}, respectively, we get the first and the second equations of \eqref{eq:FHD-full}. Taking $\mathcal V_h = (\bs 0,\bs F_h,0)\in \mathfrak W_h$, $q_h = 0 \in L_h$ in \eqref{eq:aux-pro}, we get the third equation of \eqref{eq:FHD-full}.  For any $r_h \in W_h$, we take $\mathcal V_h = (\bs 0,\bs 0,r_h) \in \mathfrak W_h$ in \eqref{eq:aux-pro}, we get the last equation of \eqref{eq:FHD-full}.
\end{proof}

In the following, we will prove that the auxiliary problem \eqref{eq:aux-pro} has solutions, and then Lemma \ref{lem:equ-solu} implies the full-discrete scheme \eqref{eq:FHD-full} also has solutions.  

We first show that the following inf-sup condition holds.
 \begin{lemma}
\label{lem:inf-sup-dis}
We have
$$
\sup\limits_{\mathcal V_{h} \in \mathfrak W_{h}}\frac{\mathcal A(\mathcal U_{h},\mathcal V_{h})}{\||\mathcal V_{h}\||_{\mathfrak W}} \geq \alpha\||\mathcal U_{h}\||_{\mathfrak V}\qquad\forall~~~\mathcal U_{h} = (\bs u_{h},\bs m_{h},\bs H_{h}) \in \mathfrak V_{h}.
$$
\end{lemma}
\begin{proof}
For any $\mathcal U_{h} = (\bs u_{h},\bs m_{h},\bs H_{h}) \in \mathfrak V_{h}$, taking $\mathcal V_{h,1} = (\bs u_{h},-\bs H_{h},0) \in \mathfrak W_{h}$ and using the fact that $\curl_{h}\bs H_{h} = 0$, we have
\begin{align*}
\mathcal A(\mathcal U_{h},\mathcal V_{h,1}) & = \|\bs u_{h}\|^{2} + \dt\eta\|\nabla\bs u_{h}\|^{2} + \frac{\mu_{0}\chi_{0}\dt}{\tau}\|\bs H_{h}\|^{2} + \beta\mu_{0}\dt \|\bs m_{h} \times \bs H_{h}\|^{2} \\
&\quad - \mu_{0}\sigma\dt(\div\bs m_{h},\div\bs H_{h}) - \mu_{0}(\frac{\dt}{\tau}+1)(\bs m_{h},\bs H_{h}).
\end{align*}
Taking $\mathcal V_{h,2} = (\bs 0,\bs m_{h},0) \in \mathfrak W_h$, we get
\begin{align*}
\mathcal A(\mathcal U_{h},\mathcal V_{h,2}) & = \mu_{0}\sigma\dt\|\curl_{h}\bs m_{h}\|^{2} + \mu_{0}\sigma\dt\|\div\bs m_{h}\|^{2} + \mu_0(\frac{\dt}{\tau}+1)\|\bs m_{h}\|^{2} \\
&\quad - \frac{\mu_{0}\chi_{0}\dt}{\tau}(\bs m_{h},\bs H_{h}).
\end{align*}
Note that $\bs H_{h} \in \mathfrak K_{h}^{d}$ implies that there exists $\varphi_{h} \in W_{h}$ such that $\bs H_{h} = \grad_{h}\varphi_{h}$. 
Taking $\mathcal V_{h,3} = (\bs 0,\bs 0,-(\frac{(1+\chi_{0})}{\tau} +1/\dt)\varphi_{h} + \sigma\div\bs H_{h}) \in \mathfrak W_h$, we obtain
\begin{align*}
\mathcal A(\mathcal U_{h},\mathcal V_{h,3}) & = (\frac{\mu_{0}\dt (1+\chi_{0})}{\tau}+\mu_0)\|\bs H_{h}\|^{2} + \mu_{0}\sigma\dt(\div\bs H_{h},\div\bs m_{h}) \\
&\quad + (\frac{\mu_{0}\dt(1+\chi_{0})}{\tau}+\mu_0)(\bs m_{h},\bs H_{h}) + \mu_{0}\sigma\dt\|\div\bs H_{h}\|^{2}
\end{align*}
Let $\mathcal V_{h} = \mathcal V_{h,1} + \mathcal V_{h,2} + \mathcal V_{h,3}\in \mathfrak W_h$, then we have
\begin{align*}
\mathcal A(\mathcal U_{h},\mathcal V_{h}) & = \|\bs u_{h}\|^{2} + \dt \eta\|\nabla\bs u_{h}\|^{2} + (\frac{\mu_{0}\dt(1+2\chi_{0})}{\tau}+\mu_0)\|\bs H_{h}\|^{2} \\
&\quad  + \mu_{0} \sigma\dt\|\div\bs H_{h}\|^{2} + (\frac{\mu_{0}\dt}{\tau}+1)\|\bs m_{h}\|^{2} + \mu_{0}\sigma\dt\|\curl_{h}\bs m_{h}\|^{2} \\
&\quad +  \mu_{0}\sigma\dt\|\div\bs m_{h}\|^{2}  + \beta\mu_{0}\dt\|\bs m_{h} \times \bs H_{h}\|^{2} \\
&\geq \  \||\mathcal U_{h}\||_{\mathfrak V}^{2}.
\end{align*}
Using the fact that $\|\varphi_h\| \lesssim \|\grad_h \varphi_h\| = \|\bs H_h\|$, we obtain
\begin{align*}
\||\mathcal V_{h}\||_{\mathfrak W}^{2} &  = \|\bs u_{h}\|^{2} + \dt \eta\|\nabla\bs u_{h}\|^{2} + \dt \|\bs m_{h} - \bs H_{h}\|^{2} + \dt \|\curl_{h}(\bs m_{h} - \bs H_{h})\|^{2}\\
&\quad  + \dt\|\div(\bs m_{h} - \bs H_{h})\|^{2} + \dt \|((1+\chi_{0})/\tau + 1/\dt)\varphi_{h} -\sigma\div\bs H_{h}\|^{2}\\
& \lesssim \||\mathcal U_{h}\||_{\mathfrak V}^{2}.
\end{align*}
Then the desired result follows.
\end{proof}

We define a linear operator $\Pi:~~\mathfrak V_{h} \rightarrow \mathfrak W_{h}$ by
$$
\Pi\mathcal U_{h} := (\bs u_{h},\bs m_{h} - \bs H_{h},-\frac{1+\chi_{0}}{\tau} w_{h} + \sigma\div\bs H_{h}) \in\mathfrak W_{h}
$$
 for any $\mathcal U_h = (\bs u_h,\bs m_h,\bs H_h) \in \mathfrak V_h$, with $w_{h} \in W_{h}$ satisfying $\bs H_{h} = \grad_{h}w_{h}$. We have the following properties of $\Pi$.
\begin{lemma}\label{lem:Pi-s}
The linear operator $\Pi$ is surjective, and the adjoint operator $\Pi^{\intercal}$ of $\Pi$ with respect to the $L^2$ inner product is injective.
\end{lemma}
\begin{proof}
For any $\mathcal V_{h} = (\bs v_{h},\bs F_{h},q_{h})\in\mathfrak W_{h}$, there exist a unique pair $w_{h} \in W_{h}$ and $\bs H_{h} \in \mathfrak K_{h}^{d}$ such that
$$
\bs H_{h} = \grad_{h} w_{h}\qquad\text{and}\qquad q_{h} = -\frac{1+\chi_{0}}{\tau} w_{h} + \sigma\div\bs H_{h},
$$
since $\bs H_h$ and $w_h$ solve the saddle point problem
\begin{align*}
\left\{
\begin{array}{rll}
(\bs H_h,\bs G_h) + (w_h,\div\bs G_h) &= 0 & \forall~\bs G_h \in \bs V_h,\\
\sigma(\div\bs H,r_h) - \frac{1+\chi_0}{\tau}(w_h,r_h) & = (q_h,r_h) & \forall~r_h \in W_h.
\end{array}
\right.
\end{align*}
 Therefore, choosing $\mathcal U_{h} = (\bs v_{h},\bs F_{h} + \bs H_{h},\bs H_{h}) \in \mathfrak V_{h}$, we have $\Pi\mathcal U_{h} = \mathcal V_{h}$. The surjective of $\Pi$ implies that the adjoint operator $\Pi^{\intercal}$ of $\Pi$ with respect to the $L^{2}$ inner product is injective.
\end{proof}

We define $\mathcal A_{h}:~\mathfrak V_{h} \rightarrow \mathfrak W_{h}$ as 
$$
(\mathcal A_{h}(\mathcal U_{h}),\mathcal V_{h}) = \mathcal A(\mathcal U_{h},\mathcal V_{h})\qquad\forall~~\mathcal U_{h} \in \mathfrak V_{h},~~\mathcal V_{h} \in \mathfrak W_{h}.
$$
The proof of Lemma \ref{lem:inf-sup-dis} also implies that
\begin{equation}
\label{eq:inf-sup-op}
(\mathcal A_{h}(\mathcal U_{h}) ,\Pi\mathcal U_{h}) \geq\alpha \||\mathcal U_{h}\||_{\mathfrak V}^{2},\quad\||\Pi\mathcal U_{h}\||_{\mathfrak W} \lesssim \||\mathcal U_{h}\||_{\mathfrak V},\quad\forall~\mathcal U_{h} \in \mathfrak V_{h}.
\end{equation}
We consider the following problem: Find $\mathcal U_{h}^n \in \mathfrak V_{h}$ such that
\begin{equation}
\label{eq:ker-aux}
\mathcal A(\mathcal U_{h}^n,\mathcal V_{h}) = \bs l(\mathcal V_{h})\qquad\forall~~~\mathcal V_{h} \in \mathfrak W_{h}.
\end{equation}
Here $\bs l(\mathcal V_{h})$ is the same as in \eqref{eq:aux-pro}. 
The problem \eqref{eq:ker-aux} is equivalent to the operator form
$$
\mathcal A_{h}(\mathcal U_{h}^n) = Q_{h}^{\mathfrak W}\bs l,
$$
where $Q_{h}^{\mathfrak W}$ is the $L^{2}$ projection operator to $\mathfrak W_{h}$.
We have the following Lemma.

\begin{lemma}\label{lem:existence-sol-aux-ker}
The problem \eqref{eq:ker-aux} has at least one solution $\mathcal U_{h}^n \in \mathfrak V_{h}$.
\end{lemma}
\begin{proof}
Let $\Phi:~~\mathfrak V_{h} \rightarrow \mathfrak W_{h}$ be defined as
$$
\Phi(\mathcal U_{h}^n) := \mathcal A_{h}(\mathcal U_{h}^n) - Q_{h}^{\mathfrak W}\bs l.
$$
Then, for any $\mathcal U_{h}^n\in \mathfrak V_{h}$, we have
\begin{align*}
(\Pi^{\intercal}\Phi(\mathcal U_{h}^n),\mathcal U_{h}^n) & = (\Phi(\mathcal U_{h}^n),\Pi\mathcal U_{h}^n) = (\mathcal A_{h}(\mathcal U_{h}^n),\Pi\mathcal U_{h}^n) - \bs l(\Pi\mathcal U_{h}^n) \\
& \geq (\alpha \||\mathcal U_{h}^n\||_{\mathfrak V} - \||\bs l\||_{\mathfrak W})\||\mathcal U_{h}^n\||_{\mathfrak V}.
\end{align*}
Taking $c = \||\bs l\||_{\mathfrak W}/\alpha$, we get
$$
(\Pi^{\intercal}\Phi(\mathcal U_{h}^n),\mathcal U_{h}^n) \geq 0\qquad\forall~~~\mathcal U_{h}^n \in \mathfrak V_{h}\quad\text{with}\quad\||\mathcal U_{h}^n\||_{\mathfrak V} = c.
$$
By \cite[Chapter IV, Corollary 1.1]{Girault1986}, there exists an element $\mathcal U_{h}^n \in \mathfrak V_{h}$ such that 
$$
\Pi^{\intercal}\Phi(\mathcal U_{h}^n) = 0\quad\text{with} \quad \||\mathcal U_{h}^n\||_{\mathfrak V} \leq c. 
$$
The injective of the linear operator $\Pi^{\intercal}$ implies that
$$
\Phi(\mathcal U_{h}^n) = 0\quad\text{with} \quad \||\mathcal U_{h}^n\||_{\mathfrak V} \leq c.
$$
This completes the proof.
\end{proof}

We are now at a position to state the following existence  theorem.
\begin{theorem}\label{the:ex-solu}
The full-discrete scheme \eqref{eq:FHD-full} has at lest one solution.
\end{theorem}
\begin{proof}
By Lemma \ref{lem:equ-solu}, we only need to prove that \eqref{eq:aux-pro} has at lest one solution $\mathcal U_h^n = (\bs u_h^n,\bs m_h^n,\bs H_h^n) \in \mathfrak V_h$ and $\tilde p_h^n \in L_h$. From Lemma \ref{lem:existence-sol-aux-ker},   let $\mathcal U_h^n$ be a solution  of  \eqref{eq:ker-aux}, and we   determine $\tilde p_h^n \in L_h$ such that
$$
\dt(\tilde p_h^n,\div\bs v_h) = \bs l (\mathcal V_h) - \mathcal A(\mathcal U_h^n,\mathcal V_h)\qquad\forall~\mathcal V_h \in \bs S_h \times V_h \times W_h.
$$
The existence of $\tilde p_h^n\in L_h$ is then guaranteed by the inf-sup condition \eqref{eq:inf-sup}.
\end{proof}

For  the  fully discrete scheme \eqref{eq:FHD-full}, define the energy 
\begin{align*}
\mathcal E_h^n& := \|\bs u_h^n\|^2 + \|\bs m_h^n\|^2 + \mu_0\|\bs H_h^n\|^2.
\end{align*}
Then the following energy estimate holds.
\begin{theorem}
\label{them:energy-dis-f}
Assume that $\bs u_h^{n} \in \bs S_h$, $\tilde p_h^{n}\in L_h$, $\bs m_h^{n} \in  \bs V_h$,  $\bs z_h^{n}\in \bs U_h$, $\bs k_h^{n} \in \bs U_h$, $\bs H_h^{n} \in  \bs V_h$, and $\varphi_h^n \in W_h$ solve  \eqref{eq:FHD-full}, then
we have  
\begin{align*}
\mathcal E_h^n+2\dt\mathcal F_{h}^{n}  \leq   \mathcal E_h^{n-1}
+C \dt \left(\frac{\mu_{0} + \chi_{0}}{\tau}\|\bs H_{e}^{n}\|^{2} + \mu_{0}\|\delta_{t}\bs H_{e}^{n}\|^{2} + \mu_{0}\sigma\|\div\bs H_{e}^{n}\|^{2}\right)
\end{align*}
with
\begin{align*}
\mathcal F_{h}^{n} & = \eta\|\nabla\bs u_{h}^{n}\|^{2} + \sigma(\mu_{0}+1)\|\div\bs m_{h}^{n}\|^{2} + \sigma \|\bs k_{h}^{n}\|^{2} \\
&\quad + \frac{1}{\tau}\|\bs m_{h}^{n}\|^{2} + \frac{1}{\tau}[\mu_{0}(1+\chi_{0}) +\chi_{0}]\|\bs H_{h}^{n}\|^{2} + \mu_{0}\beta \|\bs m_{h}^{n} \times \bs H_{h}^{n}\|^{2}.
\end{align*}
\end{theorem}
\begin{proof}
Taking $\bs v_h = \bs u_{h}^{n}$ and $\bs F_h = \bs H_{h}^{n}$ in the first and third equations of \eqref{eq:FHD-full}, respectively, we have
\begin{align*}
\frac{1}{\dt}\|\bs u_{h}^{n}\|^{2} + \eta\|\nabla\bs u_{h}^{n}\|^{2} = \frac{1}{\dt}(\bs u_{h}^{n-1},\bs u_{h}^{n}) + \mu_{0}b(\bs u_{h}^{n};\bs H_{h}^{n},\bs m_{h}^{n}) + \frac{\mu_{0}}{2}(\bs u_{h}^{n} \times \bs k_{h}^{n},\bs H_{h}^{n})
\end{align*}
and
\begin{align*}
&(\delta_{t}\bs m_{h}^{n},\bs H_{h}^{n}) + b(\bs u_{h}^{n};\bs m_{h}^{n},\bs H_{h}^{n}) + \frac{1}{2}(\bs u_{h}^{n} \times \bs H_{h}^{n},\bs k_{h}^{n}) + \sigma(\div\bs m_{h}^{n},\div\bs H_{h}^{n}) \\
=& -\frac{1}{\tau}(\bs m_{h}^{n},\bs H_{h}^{n}) + \frac{\chi_{0}}{\tau}\|\bs H_{h}^{n}\|^{2} + \beta\|\bs m_{h}^{n} \times \bs H_{h}^{n}\|^{2},
\end{align*}
which further imply 
\begin{equation}\label{eq:u-dis}
\begin{split}
\frac{1}{\dt}\|\bs u_{h}^{n}\|^{2}  + \eta\|\nabla\bs u_{h}^{n}\|^{2} =& \frac{1}{\dt}(\bs u_{h}^{n-1},\bs u_{h}^{n}) + \mu_{0}(\delta_{t}\bs m_{h}^{n},\bs H_{h}^{n})  + \mu_{0}\sigma(\div\bs m_{h}^{n},\div\bs H_{h}^{n}) \\
&+ \frac{\mu_{0}}{\tau}(\bs m_{h}^{n},\bs H_{h}^{n}) - \frac{\chi_{0}\mu_{0}}{\tau}\|\bs H_{h}^{n}\|^{2} -\mu_{0} \beta \|\bs m_{h}^{n}\times\bs H_{h}^{n}\|^{2}.
\end{split}
\end{equation}
The sixth equation of \eqref{eq:FHD-full} implies that
$\bs H_{h}^{n} = \grad_{h}\varphi_{h}^{n}$.  Taking $r_h = \varphi_{h}^{n}$ in the seventh equation of \eqref{eq:FHD-full} at time levels $t_{n}$ and $t_{n-1}$ and using integration by part, we get
\begin{equation}\label{eq:mH-dis}
(\bs m_{h}^{n},\bs H_{h}^{n}) + \|\bs H_{h}^{n}\|^{2} = -(\pi_{h}^{v}\bs H_{e}^{n},\bs H_{h}^{n})
\end{equation}
and
$$
(\bs m_{h}^{n-1},\bs H_{h}^{n}) + (\bs H_{h}^{n-1},\bs H_{h}^{n}) = -(\pi_{h}^{v}\bs H_{e}^{n-1},\bs H_{h}^{n}).
$$
These two equations  give
\begin{equation}\label{eq:deltamH}
(\delta_{t}\bs m_{h}^{n},\bs H_{h}^{n}) = -\frac{1}{\dt}\|\bs H_{h}^{n}\|^{2} + \frac{1}{\dt}(\bs H_{h}^{n-1},\bs H_{h}^{n}) - (\pi_{h}^{v}\delta_{t}\bs H_{e}^{n},\bs H_{h}^{n}).
\end{equation}
Taking $r_h = \div\bs m_h^n$ in the seventh equation of \eqref{eq:FHD-full}, we get
\begin{equation}
\label{eq:divmdivH}
(\div\bs m_h^n,\div\bs H_h^n) = -\|\div\bs m_h^n\|^2 - (\div\bs H_e,\div\bs m_h^n).
\end{equation}
Using the equations \eqref{eq:u-dis}-\eqref{eq:divmdivH} , we have
\begin{equation*}\label{eq:pr-disE-1}
\begin{split}
&\frac{1}{\dt}(\|\bs u_{h}^{n}\|^{2}  +\mu_{0} \|\bs H_{h}^{n}\|^{2}) + \eta\|\nabla\bs u_{h}^{n}\|^{2}  + \mu_{0}\sigma\|\div\bs m_{h}^{n}\|^{2} \\
&\quad + \frac{\mu_{0}}{\tau}(1+\chi_{0})\|\bs H_{h}^{n}\|^{2} + \mu_{0}\beta\|\bs m_{h}^{n}\times\bs H_{h}^{n}\|^{2} 
\\
 =& -\frac{\mu_{0}}{\tau}(\pi_{h}^{v}\bs H_{e}^{n},\bs H_{h}^{n}) - \mu_{0}(\pi_{h}^{v}\delta_{t}\bs H_{e}^{n},\bs H_{h}^{n}) - \mu_{0}\sigma(\div \bs H_{e}^{n},\div\bs m_{h}^{n}) \\
&\quad + \frac{1}{\dt} \left[ (\bs u_{h}^{n-1},\bs u_{h}^{n})  + \mu_{0}(\bs H_{h}^{n-1},\bs H_{h}^{n})\right].
\end{split}
\end{equation*}
Taking $\bs F_h = \bs m_{h}^{n}$ in the third equation of \eqref{eq:FHD-full}, we have
\begin{equation*}\label{eq:pr-disE-2}
\begin{split}
&\frac{1}{\dt}\|\bs m_{h}^{n}\|^{2}  + \sigma\|\bs k_{h}^{n}\|^{2} + \sigma\|\div\bs m_{h}^{n}\|^{2} + \frac{1}{\tau}\|\bs m_{h}^{n}\|^{2} + \frac{\chi_{0}}{\tau}\|\bs H_{h}^{n}\|^{2} \\
=&   -\frac{\chi_{0}}{\tau}(\pi_{h}^{v}\bs H_{e}^{n},\bs H_{h}^{n}) + \frac{1}{\dt}(\bs m_{h}^{n-1},\bs m_{h}^{n}),
\end{split}
\end{equation*}
The above two relations show  
\begin{equation*}\label{eq:pr-disE-3}
\begin{split}
&\frac{1}{\dt}  (\|\bs u_{h}^{n}\|^{2} + \mu_{0}\|\bs H_{h}^{n}\|^{2} + \|\bs m_{h}^{n}\|^{2}) + \eta\|\nabla\bs u_{h}^{n}\|^{2} + \sigma(\mu_{0}+1)\|\div\bs m_{h}^{n}\|^{2} + \sigma \|\bs k_{h}^{n}\|^{2} \\
&\quad + \frac{1}{\tau}\|\bs m_{h}^{n}\|^{2} + \frac{1}{\tau}[\mu_{0}(1+\chi_{0}) +\chi_{0}]\|\bs H_{h}^{n}\|^{2} + \mu_{0}\beta \|\bs m_{h}^{n} \times \bs H_{h}^{n}\|^{2}  \\
=&  -\frac{\mu_{0} +\chi_{0}}{\tau} (\pi_{h}^{v}\bs H_{e}^{n},\bs H_{h}^{n}) - \mu_{0}(\pi_{h}^{v}\delta_{t}\bs H_{e}^{n},\bs H_{h}^{n}) - \mu_{0}\sigma(\div\bs H_{e}^{n},\div\bs m_{h}^{n}) \\
&\quad + \frac{1}{\dt}\left[(\bs u_{h}^{n-1},\bs u_{h}^{n})  + \mu_{0}(\bs H_{h}^{n-1},\bs H_{h}^{n}) + (\bs m_{h}^{n-1},\bs m_{h}^{n}) \right],
\end{split}
\end{equation*}
which, together with the inequality
\begin{equation*}\label{eq:pr-disE-4}
\begin{split}
&(\bs u_h^{n-1},\bs u_h^n) + \mu_0(\bs H_h^{n-1},\bs H_h^n) + (\bs m_h^{n-1},\bs m_h^n) \\
\leq &\frac{1}{2}(\|\bs u_h^n\|^2 + \mu_0\|\bs H_h^n\|^2 + \|\bs m_h^n\|^2)
+ \frac{1}{2}(\|\bs u_h^{n-1}\|^2 + \mu_0\|\bs H_h^{n-1}\|^2 + \|\bs m_h^{n-1}\|^2),
\end{split}
\end{equation*}
%
leads to 
\begin{align*}
&\mathcal E_h^n  + 2\dt \mathcal F_h^n\\ 
 \leq & \mathcal E_h^{n-1} + 2\dt \left( -\frac{\mu_{0} +\chi_{0}}{\tau} (\pi_{h}^{v}\bs H_{e}^{n},\bs H_{h}^{n}) - \mu_{0}(\pi_{h}^{v}\delta_{t}\bs H_{e}^{n},\bs H_{h}^{n}) - \mu_{0}\sigma(\div\bs H_{e}^{n},\div\bs m_{h}^{n})\right).
\end{align*}
Finally, the desired result follows from the Cauchy-Schwarz inequality and Lemma \ref{lem:geq}.
\end{proof}

\begin{remark}\label{energy-decay}
When $\bs H_e=0$, from Theorem \ref{them:energy-dis-f} we easily have the following energy decaying result:
\begin{align*}
\mathcal E_h^n \leq   \mathcal E_h^{n-1}.
\end{align*}
In particular, if  $ \mathcal F_{h}^{n} \neq 0$, then 
\begin{align*}
\mathcal E_h^n <   \mathcal E_h^{n-1}.
\end{align*}
\end{remark}

\subsection{Error analysis}
To give the error analysis of the fully discrete scheme \eqref{eq:FHD-full},  we first introduce the following notations:
\begin{align*}
&\xi_u^n: = \pi_h \bs u(\cdot,t_n) - \bs u_h^n, \quad \xi_m^n :=  I_h^d\bs m(\cdot,t_n) - \bs m_h^n,\quad \xi_z^n: =  I_h^c\bs z(\cdot,t_n) - \bs z_h^n,  \\
& \xi_k^n := I_h^c \bs k(\cdot,t_n) - \bs k_h^n,\quad \xi_H^n :=  I_h^d \bs H(\cdot,t_n) - \bs H_{h}^n,\quad \xi_p^n := Q_h\tilde p(\cdot,t_n) - \tilde p_h^n,\\
&\theta_u^n := \pi_h\bs u(\cdot,t_n) - \bs u(\cdot,t_n),\quad \theta_m^n :=  I_h^d\bs m(\cdot,t_n) - \bs m(\cdot,t_n),\quad\,\, \theta_z^n :=  I_h^c\bs z(\cdot,t_n) - \bs z(\cdot,t_n), \\ 
& \theta_k^n := I_h^c\bs k(\cdot,t_n) - \bs k(\cdot,t_n),\quad \theta_H^n: =  I_h^d \bs H(\cdot,t_n) - \bs H(\cdot,t_n),\quad \theta_p^n :=  Q_h\tilde p(\cdot,t_n) - \tilde p(\cdot,t_n).
\end{align*}

Next, we rewrite the first and third equations of \eqref{eq:FHD-weak} at the time level $t_n$ as:
\begin{equation*}
\begin{split}
&(\delta_t\pi_h\bs u^n,\bs v) - (\bs u^n \times \curl\bs u^n,\bs v) + \eta(\nabla\pi_hu^n,\nabla\bs v) \\ 
&\quad - (Q_h\tilde p,\div\bs v)  -\mu_0b(\bs v;\bs H^n,\bs m^n) - \frac{\mu_0}{2}(\bs v\times \bs k^n,\bs H^n) \\
=& (\delta_t\theta_u^n,\bs v)
+ (\delta_t\bs u^n - \partial_t\bs u^n,\bs v) + \eta(\nabla\theta_u^n,\bs v) - (\theta_p^n,\div\bs v) \qquad\forall\bs v \in \bs S_h,
\end{split}
\end{equation*}
and 
\begin{equation*}
\begin{split}
&(\delta_tI_h^d\bs m^n,\bs F_h) + b(\bs u^n;\bs m^n,\bs F_h) + \frac{1}{2}(\bs u^n \times \bs F_h,\bs k^n) - \frac{1}{2}(\curl I_h^c\bs z^n,\bs F_h)\\
&\quad + \sigma(\curl I_h^c\bs k^n,\bs F_h) + \sigma (\div I_h^d\bs m^n,\div\bs F_h) + \frac{1}{\tau} (I_h^d\bs m^n,\bs F_h) \\
&\quad -\frac{\chi_0}{\tau}(I_h^d\bs H^n,\bs F_h) + \beta(\bs m^n\times (\bs m^n \times \bs H^n),\bs F_h) \\
=& (\delta_t\theta_m^n,\bs F_h) 
+ (\delta_t\bs m^n - \partial_t\bs u^n,\bs F_h) - \frac{1}{2}(\curl\theta_z^n,\bs F_h) \\
&\quad + \sigma(\curl\theta_k^n,\bs F_h)  
+ \frac{1}{\tau}(\theta_m^n,\bs F_h) - \frac{\chi_0}{\tau}(\theta_H^n,\bs F_h) \qquad\forall~\bs F_h\in \bs V_h.
\end{split}
\end{equation*}
Subtracting the first and third equations of \eqref{eq:FHD-full} from the above two equations, respectively, we get
\begin{equation}
\label{eq:err-f-1}
\begin{split}
&(\delta_t\xi_u^n,\bs v_h) + \eta(\nabla\xi_u^n,\nabla\bs v_h) - (\xi_p^n,\div\bs v_h) \\
=& f_1^n(\bs v_h) + f_2^n(\bs v_h) 
- f_3^n(\bs v_h)  + (\delta_t\theta_u^n,\bs v_h)  \\
&\quad + (\delta_t\bs u^n - \partial_t\bs u^n,\bs v_h)+ \eta(\nabla\theta_u^n,\nabla\bs v_h) - (\theta_p^n,\div\bs v_h) \qquad\forall~\bs v_h \in \bs S_h
\end{split}
\end{equation}
and 
\begin{equation}\label{eq:err-f-2}
\begin{split}
&(\delta_t\xi_m^n,\bs F_h) - \frac{1}{2}(\curl\xi_z^n,\bs F_h) + \sigma(\curl\xi_k^n,\bs F_h) + \frac{1}{\tau}(\xi_m^n,\bs F_h) \\
&\quad + \sigma(\div\xi_m^n,\div\bs F_h) - \frac{\chi_0}{\tau}(\xi_H^n,\bs F_h)\\
 = &\frac{1}{\tau}(\theta_m^n,\bs F_h) 
+ (\delta_t\theta_m^n,\bs F_h) + (\delta_t\bs m^n - \partial_t\bs m^n,\bs F_h) - \frac{1}{2}(\curl\theta_z^n,\bs F_h)  \\
&\quad 
+ \sigma(\curl\theta_k^n,\bs F_h)  - \frac{\chi_0}{\tau}(\theta_H^n,\bs F_h) - f_4^n(\bs F_h) - f_5^n(\bs F_h) - f_6^n(\bs F_h) \quad \forall~\bs F_h \in \bs V_h.
\end{split}
\end{equation}

We define
\begin{align*}
\tilde J_2^n := &-\frac{\mu_{0}}{\tau}(\theta_{m},\xi_{H}) + \frac{\mu_{0}\chi_{0}}{\tau}(\theta_{H},\xi_{H}) - \mu_{0}(\delta_{t}\theta_{m},\xi_{H}) + \frac{1}{2}\mu_{0}(\curl\theta_{z},\xi_{H}) \\
& \quad- \sigma\mu_{0}(\curl \theta_{k},\xi_{H}) 
+ (\delta_{t}\theta_{u},\xi_{u}) + \eta(\nabla\theta_{u},\nabla\xi_{u})  + (\tilde p - \tilde p_h,\div\xi_{u})\\
&\quad + (\delta_t\bs u^n - \partial_t\bs u^n,\xi_u^n) - \mu_0(\delta_t \bs m^n - \partial_t\bs m^n,\xi_H^n),\\
\tilde J_3^n  :=& \frac{1}{\tau}(\theta_{m},\xi_{m}) - \frac{\chi_{0}}{\tau}(\theta_{H},\xi_{m}) - (\delta_{t}\theta_{m},\xi_{m}) - \frac{1}{2}(\curl \theta_{z},\xi_{m})\\
&\quad  + \sigma(\curl\theta_{k},\xi_{m}) + \sigma(\theta_{k},\xi_{k}) + (\delta_t\bs m^n - \partial_t\bs m^n,\xi_m^n).
\end{align*}
Similar to Lemma  \ref{lem:lem:err-eq-1},   the following result holds:
\begin{lemma}\label{lem:err-f-1}
We have
\begin{align*}
&(\delta_t\xi_u^n,\xi_u^n) + \mu_0(\delta_t\xi_H^n,\xi_H^n) + (\delta_t\xi_m^n,\xi_m^n) + \eta\|\nabla\xi_u^n\|^2 + \frac{1}{\tau}(\mu_0(1+\chi_0)+\chi_0)\|\xi_H^n\|^2 \\
&\quad + \sigma\|\xi_k^n\|^2 + \frac{1}{\tau}\|\xi_m^n\|^2 + \sigma\mu_0\|\div\xi_H^n\|^2 + \sigma\|\div\xi_m^n\|^2 \\
=& J_1^n + \tilde J_2^n + \tilde J_3^n + J_4^n.
\end{align*}
\end{lemma}
\begin{proof}
Taking $\bs v_h = \xi_u^n $ in \eqref{eq:err-f-1} and $\bs F_h = \xi_H^n$ in \eqref{eq:err-f-2}, respectively, we obtain
\begin{align*}
(\delta_t\xi_u^n,\xi_u^n)) + \eta\|\nabla\xi_u^n\|^2 = &f_1^n(\xi_u^n) + f_2^n(\xi_u^n) - f_3^n(\xi_u^n) 
+ (\delta_t\theta_u^n,\xi_u^n)\\
&\quad  + (\delta_t\bs u^n - \partial_t\bs u^n,\xi_u^n) 
+ \eta(\nabla\theta_u^n,\nabla\xi_u^n) + (\tilde p^n - \tilde p_h^n,\div\xi_u^n)
\end{align*}
and 
\begin{align*}
&(\delta_t\xi_m^n,\xi_H^n) + \frac{1}{\tau}(\xi_m^n,\xi_H^n) + \sigma(\div\xi_m^n,\div\xi_H^n) - \frac{\chi_0}{\tau}\|\xi_H^n\|^2 \\
=& \frac{1}{\tau}(\theta_m^n,\xi_H^n) 
+(\delta_t\theta_m^n,\xi_H^n) + (\delta_t\bs m^n - \partial_t\bs m^n,\xi_H^n) - \frac{1}{2}(\curl\theta_z^n,\xi_H^n)\\
&\quad  + \sigma(\curl\theta_k^n,\xi_H^n)  
-\frac{\chi_0}{\tau}(\theta_H^n,\xi_H^n) - f_4^n(\xi_H^n) - f_5^n(\xi_H^n) - f_6^n(\xi_H^n).
\end{align*}
These two equations indicate
\begin{equation}\label{eq:err-f-p-1}
\begin{split}
&(\delta_t\xi_u^n,\xi_u^n)) + \eta\|\nabla\xi_u^n\|^2 + \frac{\mu_0\chi_0}{\tau}\|\xi_H^n\|^2 \\
&\quad  - \mu_0(\delta_t\xi_m^n,\xi_H^n) 
-\sigma\mu_0(\div\xi_m^n,\div\xi_H^n) - \frac{\mu_0}{\tau}(\xi_m^n,\xi_H^n) \\
=& J_1^n + \tilde J_2^n.
\end{split}
\end{equation}

The last equations of \eqref{eq:FHD} and  \eqref{eq:FHD-full} imply that
\begin{align*}
& \div\bs H^n + \div\bs m^n = -\div\bs H_e^n, \\
& \div\bs H_h^n + \div\bs m_h^n = - Q_h\div\bs H_e^n.
\end{align*}
Therefore,
\begin{equation}\label{eq:hmdiff}
\div\xi_H^n+ \div\xi_m^n = -(I-Q_h)\div\bs H_e^n + \div\theta_H^n + \div\theta_m^n.
\end{equation}
Since $\xi_H^n \in \mathfrak K_h^d$, there exists a unique $\phi_h \in W_h$ such that $\xi_H^n = \grad_h\phi_h$. Testing \eqref{eq:hmdiff} with $\phi_h$, we get
$$
\|\xi_H^n\|^2 + (\xi_m^n,\xi_H^n) = 0.
$$
Subtracting \eqref{eq:hmdiff} on the time level $t_{n-1}$ from \eqref{eq:hmdiff} on the time level $t_{n}$, and testing the resultant equation with $\varphi_h$, we obtain
\begin{equation}\label{eq:err-f-p-2}
(\delta_t\xi_H^n,\xi_H^n) + (\delta_t\xi_m^n,\xi_H^n) = 0.
\end{equation}

Taking $\bs F_h = \xi_m^n$ in \eqref{eq:err-f-2} and using the fact that
$$
(\curl\xi_k^n,\xi_m^n) = \|\xi_k^n\|^2 - (\theta_k^n,\xi_k^n),
$$
we get
\begin{equation}\label{eq:err-f-p-3}
\begin{split}
(\delta_t\xi_m^n,\xi_m^n)) + \sigma\|\xi_k^n\|^2 + \sigma\|\div\xi_m^n\|^2 + \frac{1}{\tau}\|\xi_m^n\|^2 
 + \frac{\chi_0}{\tau}\|\xi_H\|^2 = \tilde J_3^n + J_4^n.
\end{split}
\end{equation}
Applying \eqref{eq:err-f-p-2} and adding \eqref{eq:err-f-p-1} and \eqref{eq:err-f-p-3} together, we finally obtain the desired result.
\end{proof}

By following a similar   proof  as that of  Lemma \ref{lem:J23}, we can estimate $\tilde J_2^n$ and $\tilde J_3^n$ as follows:
\begin{align}
\label{eq:J2-f} \tilde J_{2}^n & \lesssim h^{l+1}\left(   \|\xi_{u}^n\| + \|\nabla\xi_{u}^n\| +\|\xi_{H}^n\| +\|\xi_{p}^n\|)\right)  +  \dt  \|\xi_H^n\|,  
\\
\label{eq:J3-f}\tilde J_{3}^n & \lesssim h^{l+1} \left( \|\xi_{m}^n\| +  \|\xi_{k}^n\| \right) +  \dt\|\xi_m^n\|.
\end{align}

We turn to the estimate of $\xi_p^n$, and have the following result.
 \begin{lemma}
 \label{lem:xipn}
Under Assumption \ref{assum-II},
  for any $1\leq L \leq N$ we have
 \begin{align*}
\dt\sum\limits_{n = 1}^L\|\xi_{p}^n\|  \lesssim &\|\xi_u^L\| +   h^{l+1}+  \dt +     h^{l+1} \max\limits_{1\leq n\leq L}\|\div\bs H^n\|_{L^3}+   h^{l+1/2} \|\xi_k^n\|  \\
&\quad  +  \dt\sum\limits_{n = 1}^{L}\left(\|\nabla\xi_{u}^n\|  + \|\xi_{H}^n\|_{\div}  + \|\xi_k^n\|\right) + \max\limits_{ 1\leq n \leq L}\|\div\bs H_h^n\|_{L^3}\ \dt \sum\limits_{n = 1}^L \|\xi_m^n\|\\
&\quad  + \dt \sum\limits_{n=1}^L \left(\|\xi_k^n\|\|\div\xi_H^n\| + \|\nabla\xi_u^n\|^2 \right).
\end{align*}
 \end{lemma}
 \begin{proof}
 The inf-sup condition \eqref{eq:inf-sup} and \eqref{eq:err-f-1} imply that for any $1\leq L \leq N$, we have
\begin{align*}
\dt \sum\limits_{n = 1}^L\|\xi_{p}^n\| & \lesssim \dt\sup\limits_{\bs v_{h}\in \bs S_{h}}\sum\limits_{n=1}^{L}\frac{(\xi_{p}^n,\nabla\cdot\bs v_{h})}{\|\bs v_{h}\|_{1}} \\
& \lesssim \|\xi_{u}^L\| +  \dt\sum\limits_{n = 1}^{L}\|\nabla\xi_{u}^n\| +   h^{l+2}\|\bs u\|_{L^\infty(H^{l+2})}  + \dt \|\bs u\|_{W^{2,\infty}(L^2)} \\
&\quad + h^{l+2}\|\bs u_0\|_{l+2}+  h^{l+1}\|\bs u\|_{L^\infty(H^{l+2})} + h^{l+1}\|\tilde p\|_{L^\infty(H^{l+1})}  +\tilde T_1^n,
\end{align*}
with
$$
\tilde T_1^n = \sup\limits_{\bs v_{h} \in \bs S_{h}}\dt\sum\limits_{n = 1}^{L}\frac{|f_{1}^n(\bs v_{h}) +  f_{2}^n(\bs v_{h}) - f_3^n(\bs v_h)|}{\|\bs v_{h}\|_{1}},
$$
where
\begin{equation*}
\begin{array}{l}
f_{1}^n(\bs v_h) = \mu_0b(\bs v_h;\bs H(\cdot, t_n),\bs m(\cdot, t_n)) - \mu_0b(\bs v_h;\bs H_h^n,\bs m_h^n),
\\
f_2^n(\bs v_h) = \frac{\mu_0}{2}(\bs v_h\times \bs k(\cdot, t_n),\bs H(\cdot, t_n)) - \frac{\mu_0}{2}(\bs v_h\times \bs k_h^n,\bs H_h^n),
\\
f_3^n(\bs v_h) = (\bs u(\cdot, t_n) \times \curl\bs u(\cdot, t_n) - \bs u_h^n\times\curl\bs u_h^n,\bs v_h).
\end{array}
\end{equation*}
By following a similar routine as in the proof of Lemma \ref{lem:xip}, for any $\bs v_h \in \bs S_h$ we can obtain
\begin{align*}
|f_1^n(\bs v_h)|  \lesssim &  \|\bs v_h\|\left(  h^{l+1}+\|\xi_H^n\|_{\div} \right)   +\left( h^{l+1}+\|\xi_m^n\|\right)\|\nabla\bs v_h\| \|\div\bs H_h^n\|_{L^3} , \\
|f_2^n(\bs v_h)| \lesssim & \|\bs v_h\|\left( h^{l+1}+\|\xi_k^n\| +  \|\xi_H^n\|\right) +  \left(h^{l+1/2}+\|\div\xi_H^n\|\right)\|\xi_k^n\| \|\nabla\bs v_h\|      , \\
|f_3^n(\bs v_h)| \lesssim & \|\bs v_h\| \left( h^{l+1}+\|\xi_u^n\|+  \|\nabla\xi_u^n\|   \right)  + \|\nabla\bs v_h\|  \|\nabla\xi_u^n\|^2 .
\end{align*}
Then the desired result follows.
 \end{proof}
 
 Denote 
 $$
\mathfrak E_h^n := \|\xi_{u}^n\|^{2} + \|\xi_{H}^n\|^{2} + \|\xi_{m}^n\|^{2},
$$
$$
\mathfrak F_h^n: = \|\nabla\xi_{u}^n\|^{2} + \|\xi_{k}^n\|^{2} + \|\xi_{H}^n\|_{\div}^{2} + \|\xi_{m}^n\|_{\div}^{2} .
$$
Then we have the following estimate: 
 \begin{lemma}\label{lem:xi-estimates-f}
Under Assumption \ref{assum-II} and the condition  
 $\max\limits_{1\leq n\leq N}\|\div\bs H_h^n\|_{L^3} \leq C$, for any $1\leq L\leq N$  we have
$$
\dt\sum\limits_{n=1}^L\left( \mathfrak E_h^n + \dt\sum_{\gamma = 1}^{n}\mathfrak F_h^\gamma\right) \lesssim h^{2(l+1)} + \dt^2.
$$
\end{lemma}
\begin{proof}
By Lemmas \ref{lem:err-f-1}, \ref{lem:J2} and \ref{lem:J5} and the inequalities \eqref{eq:J2-f} and \eqref{eq:J2-f}, we have
\begin{equation}\label{eq:lem-pro-f}
\begin{split}
& \frac{1}{2\dt}(\|\xi_{u}^n\|^2 + \mu_{0}\|\xi_{H}^n\|^2  + \|\xi_{m}^n\|^2 - \|\xi_u^{n-1}\|^2 - \mu_0\|\xi_H^{n-1}\|^2 - \|\xi_m^{n-1}\|^2) \\
&\quad  + \eta\|\nabla\xi_{u}^n\|^{2} 
+ \frac{1}{\tau}(\mu_{0}(1+\chi_{0})+ \chi_{0})\|\xi_{H}^n\|^{2}  
+ \sigma\|\xi_{k}^n\|^{2} \\ 
&\quad + \frac{1}{\tau}\|\xi_{m}^n\|^{2} + \sigma\mu_{0}\|\div\xi_H^n\|^2
+ \sigma\|\div\xi_{m}^n\|^{2}  \\
\lesssim& h^{l+1}  \mathfrak T_{1}^n + h^{l+1/2}\mathfrak T_{2}^n + \mathfrak T_{3}^n + \dt \mathfrak T_4^n,
\end{split}
\end{equation}
with 
\begin{align*}
\mathfrak T_{1}^n &: = \|\nabla\xi_{u}^n\| + \|\xi_H ^n\|_{\div} + \|\xi_m^n\|_{\div} + \|\xi_k^n\| + \|\xi_z^n\| + \|\xi_p^n\|,\\
\mathfrak T_{2}^n & := \|\nabla\xi_u^n\|(\|\xi_k^n\|+\|\xi_H^n\| + \|\xi_m^n\|_{\div}) 
 + \|\xi_m^n\|(\|\xi_k^n\| + \|\div\xi_m^n\|) ,\\
\mathfrak T_{3}^n &: = \|\xi_{u}^n\|(\|\xi_{k}^n\| + \|\xi_{m}^n\|_{\div}+ \|\nabla\xi_u^n\|)  + \|\xi_m^n\|_{\div}\|\div\xi_H^n\| + \|\xi_H^n\|\|\xi_k^n\|  \\
&\qquad + \|\xi_{m}^n\|^2 + \|\xi_m^n\| \|\xi_H^n\|,\\
\mathfrak T_4^n&: = \|\xi_H^n\| +\|\xi_m^n\|.
\end{align*}
%

For any $1\leq L \leq N$, using Lemma \ref{lem:xiz} and \ref{lem:xipn}, we have
\begin{align*}
\dt\sum\limits_{n = 1}^{L}\mathfrak T_{1}^n  & \lesssim h^{l+1} + \|\xi_{u}^L\| + \dt\sum\limits_{n = 1}^{L} \left(\|\xi_{H}^n\|_{\div} + \|\xi_{m}^n\|_{\div} + \|\nabla\xi_{u}^n\| + \|\xi_{k}^n\| \right).
\end{align*}
%
Summing up \eqref{eq:lem-pro-f} with $n$ from $1$ to $L$, we get
\begin{align*}
&\mathfrak E_h^L  + \dt\sum\limits_{n = 1}^{L}\mathfrak F_h^n
 =h^{l+1}\dt \sum\limits_{n = 1}^{L}\mathfrak T_{1}^n + h^{l+1/2}\dt\sum_{n = 1}^{L}\mathfrak T_{2}^n  + \dt \sum_{n = 1}^{L}\mathfrak T_{3}^n + \dt^2\sum\limits_{n = 1}^L\mathfrak T_4^n \\
 \lesssim&  h^{2(l+1)} + h^{l+1}(\mathfrak E_h^L)^{1/2} + \dt \sum\limits_{n = 1}^{L}(\mathfrak E_h^n)^{1/2}(\mathfrak F_h^n)^{1/2}   + \dt^2(\sum\limits_{n = 1}^L \mathfrak F_h^n)^{1/2}\\
 \lesssim & h^{2(l+1)} + h^{l+1}(\mathfrak E_h^n)^{1/2} + \left(\dt \sum\limits_{n = 1}^{L} \mathfrak E_h^n \right)^{1/2} \left(\dt \sum\limits_{n = 1}^{L}\mathfrak F_h^n  \right)^{1/2}+ \dt^2(\sum\limits_{n = 1}^L \mathfrak F_h^n)^{1/2}.
\end{align*}
Therefore, we obtain
\begin{align*}
\left( \mathfrak E_h^L + \dt\sum\limits_{n = 1}^{L}\mathfrak F(r) \right)^{1/2} & \lesssim h^{l+1} + \dt+ \left(\dt\sum\limits_{n = 1}^{L}(\mathfrak E_h^n + \dt \sum_{l = 1}^{n}\mathfrak F_h^l(l)) \right)^{1/2},
\end{align*}
and then
$$
\dt\sum\limits_{L = 1}^{M}\left(\mathfrak E_h^L + \dt \sum\limits_{n = 1}^{L}\mathfrak F_h^n  \right)  \lesssim Th^{2(l+1)} + T\dt^2 + \dt\sum\limits_{L=1}^{M}\left[\dt \sum\limits_{n = 1}^{L} \left(\mathfrak E_h^n +\dt \sum\limits_{l = 1}^{n}\mathfrak F_h^l \right) \right],
$$
which, together with  the Gronwall's inequality \cite[Theorem 1.2.2]{Pachpatte1998}, implies the desired result.
\end{proof}

Using Lemma \ref{lem:xi-estimates-f} and the triangular inequality, we have the following error estimates results for the fully discrete scheme \eqref{eq:FHD-full}.
\begin{theorem}\label{final-theorem}
Under Assumption \ref{assum-II} and the condition  
 $\max\limits_{1\leq n\leq N}\|\div\bs H_h^n\|_{L^3} \leq C$, for any $1\leq L \leq N$ we have 
$$
\dt\sum\limits_{n = 1}^{L}\left( \text{err}_{1}^n + \dt\sum\limits_{l = 1}^{n} \text{err}_{2}^n\right) \lesssim h^{2(l+1)} + \dt^2
$$
and
$$
\dt\sum\limits_{n = 1}^{L}\|\tilde p^n - \tilde p_{h}^n\| \lesssim h^{2(l+1)} + \dt^2.
$$
\end{theorem}

\section{Numerical experiments}
In this section, we provide three numerical examples to verify the performance of the fully discrete scheme  \eqref{eq:FHD-full}. 
 The numerical experiments  are performed by using iFEM package~\cite{Chen2009}, and the nonlinear system  \eqref{eq:FHD-full} is solved by the following quasi-Newton iteration with $M=2$:
\begin{alg}
Given $\bs u_h^{n-1}$ and $\bs m_h^{n-1}$,   to find $\bs u_h^n$, $\tilde p_h^n$, $\bs m_h^n$, $\bs z_h^n$, $\bs k_h^n$,  $\bs H_h^n$,  and $\varphi_h^n$ through  three  steps:
\begin{enumerate}
  \item[Step 1.] Let $\bs u_h^{-} = \bs u_h^{n-1}$ and $\bs m_h^{-} = \bs m_h^{n-1}$.
  \item[Step 2.] For $\rho = 1,~2,\dots,~M$ do
  \begin{enumerate}
      \item Solving the saddle point system: Find $\bs H_h \in \bs V_h$ and $\varphi_h \in W_h$ such that
      $$
      \left\{
      \begin{array}{ll}
      (\bs H_h,\bs G_h) + (\varphi_h,\div\bs G_h) = 0 &\forall~~\bs G_h \in \bs V_h, \\
      (\div\bs H_h,r_h) = -(\div\bs H_e^n,r_h) - (\div\bs m_h^{-},r_h) &\forall~~r_h \in W_h;
      \end{array}
      \right.
      $$
      \item Solving the magnetization equation: Find $\bs m_h \in \bs V_h$, $\bs z_h \in \bs U_h$ and $\bs k_h\in \bs U_h$ such that for all $\bs F_h\in \bs V_h$, $\bs \zeta_h\in \bs U_h$ and $\bs\kappa_h\in \bs U_h$, there hold
     \begin{align*}
&\left(1+\frac{\dt}{\tau}\right)(\bs m_h,\bs F_h)  + \sigma\dt(\div\bs m_h,\div\bs F_h) 
     + \dt b(\bs u_h^{-};\bs m_h, \bs F_h) \\ 
     &\quad + \sigma\dt(\curl\bs k_h,\bs F_h) + \frac{\dt}{2}(\bs u_h^-\times \bs F_h,\bs k_h)  
      + \beta\dt(\bs m_h\times(\bs m_h^{-}\times \bs H_h),\bs F_h)\\
     &\quad   + \beta\dt(\bs m_h^{-}\times(\bs m_h\times \bs H_h),\bs F_h)  
       - \frac{\dt}{2}(\curl\bs z_h,\bs F_h)  \\ 
     &\ =
      (\bs m_h^{n-1},\bs F_h) + \frac{\chi_0}{\tau}\dt (\bs H_h,\bs F_h) 
      + \beta\dt(\bs m_h^-\times(\bs m_h^-\times \bs H_h),\bs F_h), \\
  &    (\bs z_h,\bs\zeta_h) -(\bs u_h^-\times \bs m_h,\bs\zeta_h) = 0, \\
   &   (\bs k_h,\bs\kappa_h) -(\bs m_h,\curl\bs\kappa_h) = 0;
      \end{align*}
     \item Solving the Navier-Stokes equation: Find $\bs u_h \in \bs S_h$ and $\tilde p_h \in L_h$ such that for all $\bs v_h\in \bs S_h$ and $q_h\in L_h$, there hold
\begin{align*}
(\bs u_h,\bs v_h) &+ \dt\eta(\nabla\bs u_h,\nabla\bs v_h) - \dt(\tilde p_h,\div\bs v_h) =  \dt \mu_0 b(\bs v_h;\bs H_h,\bs m_h)   
\\& +  \frac{\mu_0\dt}{2}(\bs v_h \times \bs k_h,\bs H_h) 
    + (\bs u_h^{n-1},\bs v_h) + \dt(\bs u_h^{-}\times\curl\bs u_h^{-},\bs v_h),\\
  & (\div\bs u_h, q_h)=  0.
\end{align*}
    \item Let $\bs u_h^- = \bs u_h$ and $\bs m_h^- = \bs m_h$.
  \end{enumerate}
  \item[Step 3.] Let $\bs u_h^n = \bs u_h$, $\tilde p_h^n = \tilde p_h$, $\bs m_h^n = \bs m_h$, $\bs H_h^n = \bs H_h$, $\bs z_h^n = \bs z_h$ and $\bs k_h^n = \bs k_h$.
\end{enumerate}
\end{alg}

\begin{remark}
From the convergence theory of Newton-type methods~\cite{Gil2007,Suli2003}, we can see that the iterative solution of Algorithm 1 will converge to the exact solution, provided that   the iteration number $M$ is big enough and the initial guess is nearby the exact solution. In fact, Step 1  in Algorithm 1 ensures     the initial guess to be close to the exact solution, and in all  the  subsequent numerical examples    we only need to choose $M=2$ to attain the optimal convergence of the scheme.     

\end{remark}

In the numerical scheme  \eqref{eq:FHD-full},  we use the mini-element   pair $(\mathcal P_1 \oplus \text{bubble})$-$\mathcal P_1$~\cite{Arnoldmini1984} to discretize the variables $\bs u$ and $\tilde p$,  the lowest order face element $RT_0$~\cite{Raviart;Thomas1977} to discretize   $\bs m$ and $\bs H$,  $\mathcal P_0$-element to discretize $\varphi$,  and   the lowest order edge element $NE_0$~\cite{Nedelec1980,Nedelec1986} to discretize the variable $\bs z$ and $\bs k$.  Such a   combination of finite element spaces corresponds to $l=0$. 

In   Examples  \ref{ex-1},  \ref{ex-2}  and \ref{ex-3},  we take   $\Omega = [0,1]^3$  and use  $N\times N\times N$	uniform tetrahedral meshes (cf.  Figure \ref{Fig:mesh}) with $N=4,8,16,32.$   In first two examples,   we take    the temporal step size as $\dt=h/\sqrt{3}=1/N$.  With these settings, we easily see from Theorem \ref{final-theorem} that the theoretical accuracy of the scheme is $\mathcal{O}(h+\dt)$.

\begin{figure}[!h]
\centering
\includegraphics[height=4.5cm,width=9cm]{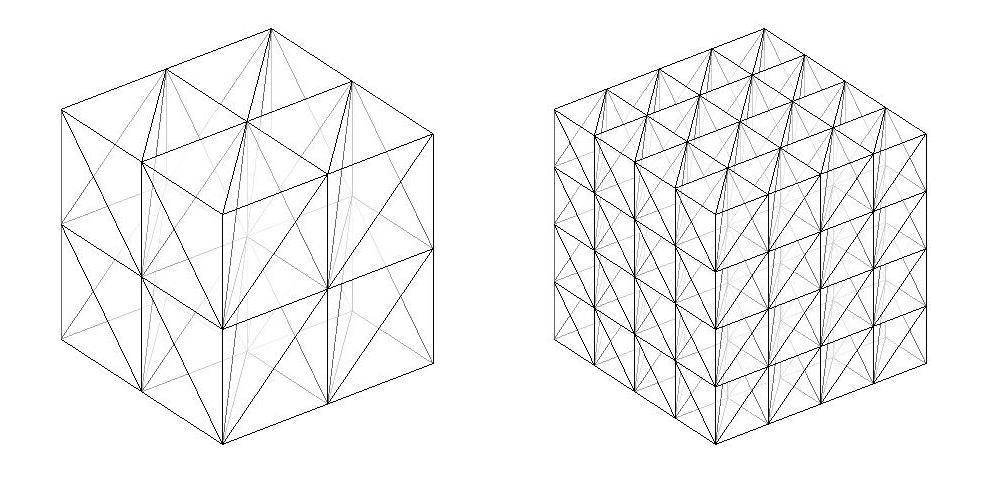}
\caption{ The domain $\Omega = [0,1]^3$:    $2\times2\times 2$   (left)     and $4\times 4\times 4$ (right) meshes.}\label{Fig:mesh}
\end{figure}

\begin{example}\label{ex-1}
  The exact solution,  $(\bs u, \bs m, \bs H, \tilde p)$,  of   the FHD model \eqref{eq:FHD} with $T=4$ is given by
\begin{align*}
& \bs u(t,x,y,z) = \sin(t)  \left(\sin(\pi y), \ 
\sin(\pi z), \ 
\sin(\pi x)
\right)^\intercal, \\
&  \bs m(t,x,y,z) = \sin(t)\left(\sin(\pi x)\sin(\pi y)\sin(\pi z),\ 
0, \ 
0\right)^\intercal, \\
& \bs H(t,x,y,z) = 200\sin(t) \begin{pmatrix}
(2x-1)(x^2-x)(y^2-y)^2(z^2-z)^2 \\
(x^2-x)^2(2y-1)(y^2-y)(z^2-z)^2 \\
(x^2-x)^2(y^2-y)^2(2z-1)(z^2-z)
\end{pmatrix},\\
& \tilde p(x,y,z) = 120 x^2 yz - 40y^3z - 40 yz^3.
\end{align*}
    The parameters $\sigma,\ \mu_0,\ \eta,\ \chi_0,\ \beta$ and $\tau$ are all chosen as $1$. 
   
    Numerical results of the relative errors of the discrete solutions at the ending time $T$ are listed in Table \ref{tab:ex3D-1}. In addition,   the exact   energy $E(t) = \|\bs u\|^2 + \mu_0\|\bs H\|^2 + \|\bs m\|^2$ and  its numerical energy at different time levels are plotted   in Figure \ref{Fig:ex3D}.

\begin{table}[h!]
\footnotesize
\begin{center}
\caption{Errors and convergence orders at ending time $T = 4$ for Example \ref{ex-1}: $\dt=1/N$.}
\label{tab:ex3D-1}
\begin{tabular}{c|c|c|c|c|c}
\hline
\hline  $N$ & $\frac{\|\bs u - \bs u_{h}\|}{\|\bs u\|}$  &$ \frac{|\bs u - \bs u_{h}|_1}{|\bs u|_1} $ & $\frac{\|\tilde p - \tilde p_h\| }{\|\tilde p\|}$&$ \frac{\|\bs m - \bs m_{h}\|}{\|\bs m\|}$ & $\frac{\|\div(\bs m - \bs m_{h})\|}{\|\div\bs m\|}$ \\
\hline
$4$ & $0.0557$ & $0.2369$ & $0.0629$ &    $0.3205$ & $0.2849$  \\
\hline
$8$ & $0.0140$ & $0.1131$ & $0.0146$ & $0.1638$ &    $0.1446$  \\
\hline
$ 16$ & $0.0035$ & $0.0557$ & $0.0037$ & $0.0824$ & $0.0740$ \\
\hline
$ 32$& $0.0009$ & $0.0277$ & $0.0010$ &  $0.0413$ &
$0.0388$ \\
\hline
order & $1.9959$ & $1.0136$ & $1.9506$ & $0.9948$ & $0.9486$ \\ 
\hline
\hline
$N$ & $\frac{\|\bs H - \bs H_h\|}{\|\bs H\|}$ & $\frac{\|\div(\bs H - \bs H_h)\|}{\|\div\bs H\|}$ & $\frac{\|\bs z - \bs z_h\|}{\|\bs z\|}$ & $\frac{\|\bs k - \bs k_h\|}{\|\bs k\|}$ & \\
\hline
$ 4$ &  $0.6102$ & $0.5589$ &    $0.3248$ & $0.3090$ & \\
\hline
$ 8$ & $0.2553$ & $0.2839$ & $0.1624$ &   $0.1599$ & \\
\hline
$ 16$  & $0.1196$ & $0.1519$ &   $0.0818$ & $0.0817$  & \\
\hline
$ 32$ & $0.0585$& $0.0868$&   $0.0411$ &  $0.0420$ & \\
\hline
order & $1.0629$ & $0.8551$ & $0.9909$ &    $0.9649$  & \\
\hline\hline
\end{tabular}
\end{center}
\end{table}

\begin{figure}[!h]
\centering
\includegraphics[width=0.80\textwidth]{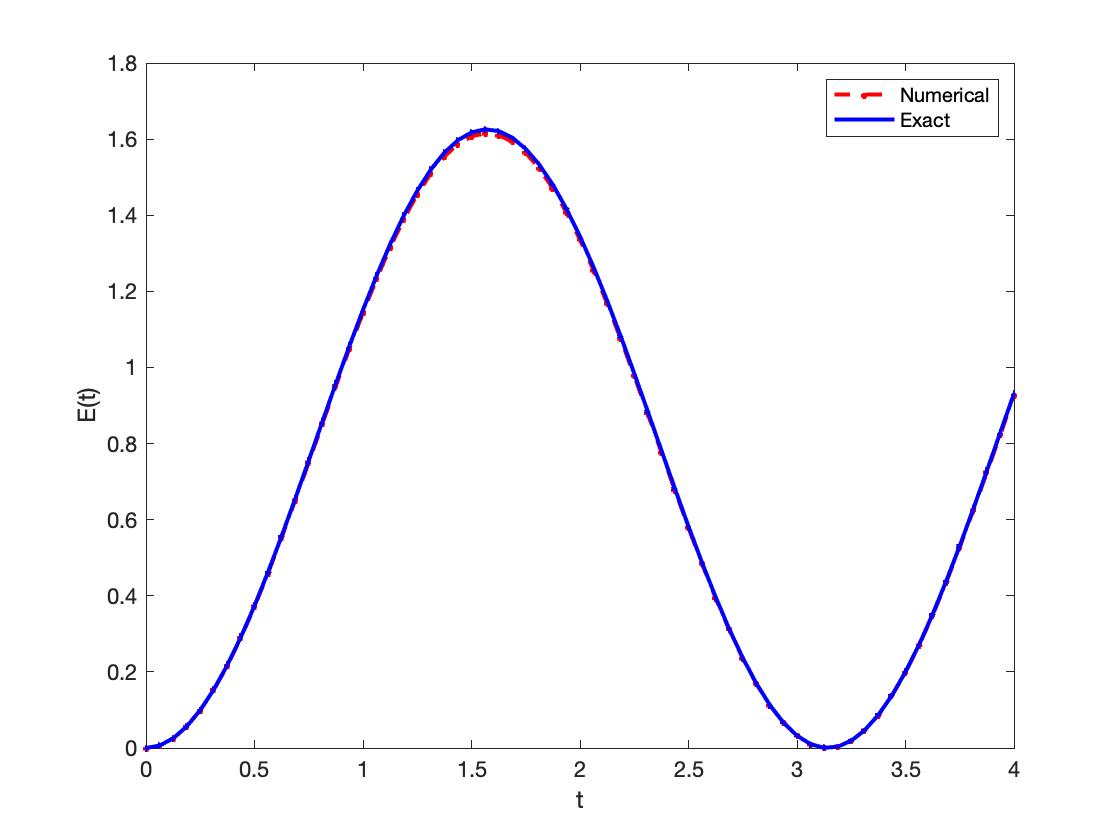}
\caption{The exact   energy $E(t) $ and  its numerical energy at different time levels for Example \ref{ex-1}:  $  \dt = 1/N, \ N=16$.}\label{Fig:ex3D}
\end{figure}

\end{example}

\begin{example}\label{ex-2}
  The exact solution  of   the FHD model \eqref{eq:FHD} with $T=1$ is given by
\begin{align*}
& \bs u(t,x,y,z) = \sin(t)\left(\sin(\pi y)\sin(\pi z),\ \sin(\pi z)\sin(\pi x),\ \sin(\pi x)\sin(\pi y)\right)^\intercal,\\
&\bs m(t,x,y,z) = \sin(t) \left(0,\ (x^2-x)(y^2-y)(z^2-z),\ 0\right)^\intercal,\\
&\bs H(t,x,y,z) = 200\sin(t) \begin{pmatrix}
(2x-1)(x^2-x)(y^2-y)^2(z^2-z)^2 \\
(x^2-x)^2(2y-1)(y^2-y)(z^2-z)^2 \\
(x^2-x)^2(y^2-y)^2(2z-1)(z^2-z) \\
\end{pmatrix},\\
&\tilde p(x,y,z) = 120 x^2yz -40y^3z - 40y z^3.
\end{align*}
   The parameters $\sigma,\ \mu_0,\ \eta,\ \chi_0,\ \beta$ and $\tau$ are all chosen as $1$. 
    Numerical results of the relative errors  of the discrete solutions  at the ending time $T$ are listed in Table \ref{tab:ex3D-2}.
    
    \begin{table}[h!]
\footnotesize
\begin{center}
\caption{Errors and convergence orders at ending time $T = 1$ for Example \ref{ex-2}: $\dt=1/N$.}
\label{tab:ex3D-2}
\begin{tabular}{c|c|c|c|c|c}
\hline
\hline  $N$ & $\frac{\|\bs u - \bs u_{h}\|}{\|\bs u\|}$  &$ \frac{|\bs u - \bs u_{h}|_1}{|\bs u|_1} $ & $\frac{\|\tilde p - \tilde p_h\| }{\|\tilde p\|}$&$ \frac{\|\bs m - \bs m_{h}\|}{\|\bs m\|}$ & $\frac{\|\div(\bs m - \bs m_{h})\|}{\|\div\bs m\|}$ \\
\hline
$4$ & $0.1204$ & $0.3850$ & $0.0654$ &    $0.3192$ & $0.2951$  \\
\hline
$8$ & $0.0313$ & $0.1933$ & $0.0156$ & $0.1645$ &    $0.1516$  \\
\hline
$16$ & $0.0079$ & $0.0967$ & $0.0040$ & $0.0829$ & $0.0784$ \\
\hline
$32$& $0.0020$ & $0.0484$ & $0.0011$ &  $0.0415$ &
$0.0418$ \\
\hline
order & $1.9722$ & $0.9977$ & $1.9696$ & $0.9809$ & $0.9402$ \\ 
\hline
\hline
$N$ & $\frac{\|\bs H - \bs H_h\|}{\|\bs H\|}$ & $\frac{\|\div(\bs H - \bs H_h)\|}{\|\div\bs H\|}$ & $\frac{\|\bs z - \bs z_h\|}{\|\bs z\|}$ & $\frac{\|\bs k - \bs k_h\|}{\|\bs k\|}$ & \\
\hline
$4$ &  $0.4431$ & $0.4616$ &    $0.3451$ & $0.3168$ & \\
\hline
$8$ & $0.2271$ & $0.2374$ & $0.1706$ &   $0.1672$ & \\
\hline
$16$  & $0.1144$ & $0.1196$ &   $0.0857$ & $0.0863$  & \\
\hline
$32$ & $0.0573$& $0.0599$&   $0.0431$ &  $0.0447$ & \\
\hline
order & $0.9834$ & $0.9817$ & $1.0000$ &    $0.9414$  & \\
\hline\hline
\end{tabular}
\end{center}
\end{table}
\end{example}

\begin{example}[Energy test]\label{ex-3}
This example is to investigate the energy decaying phenomenon of the scheme (cf. Theorem \ref{them:energy-dis-f} and Remark \ref{energy-decay}).  We consider   the FHD model \eqref{eq:FHD} with the initial value functions
$$
\bs u_0(x,y,z) = \begin{pmatrix}
\sin(\pi y)\sin(\pi z) \\ \sin(\pi z)\sin(\pi x) \\ \sin(\pi x)\sin(\pi y)\end{pmatrix},
\qquad \bs m_0(x,y,z) = \begin{pmatrix}
\sin(\pi x)\sin(\pi y)\sin(\pi z) \\ 0\\0
\end{pmatrix},
$$
and the external magnetic field $\bs H_e = \bs 0$.
We show  in Figure \ref{Fig:ex3DEner} the discrete energy curve at  the spatial and temporal meshes with $N =  32$ and $\dt = 1/64$.

\begin{figure}[h!]
\centering
\includegraphics[width=0.80\textwidth]{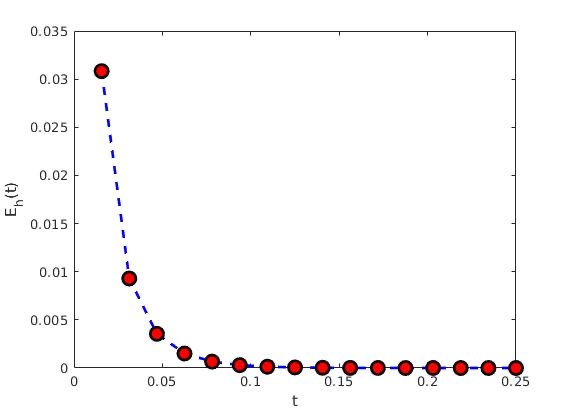}
\caption{The energy $\tilde{\mathcal E}_h$ curve with $N =  32$ and $\dt = 1/64$.}\label{Fig:ex3DEner}
\end{figure}
\end{example}

From Tables \ref{tab:ex3D-1} and \ref{tab:ex3D-2} and Figures \ref{Fig:ex3D} and \ref{Fig:ex3DEner}, we have the following observations:
\begin{itemize}
  \item The errors in $L^2$ norm for $\bs u$ and $\tilde p$ have the second order of convergence rates, which is better than the theoretical prediction, since the finite element spaces for   $\bs u$ and $\tilde p$ contain the piecewise polynomials of degree up to $1$. 
  \item The errors in $H^1$ semi-norm for $\bs u$, $L^2$ and $H(\div)$ norms for $\bs H$ and $\bs m$, and $L^2$ norm for $z$ and $k$, all have the first (optimal) order rates.
  \item The numerical energy curve in Figure \ref{tab:ex3D-2} fits the exact one almost exactly, which means that our algorithm preserves the energy of the FHD model.
  \item When there is no source term, i.e. the external magnetic field $\bs H_e = \bs 0$,   the   discrete energy $\tilde{\mathcal E}_h$ decays with time, which is conformable to  the theoretical prediction  in Theorem \ref{them:energy-dis-f}; see also Remark \ref{energy-decay}.
\end{itemize}


\end{document}